\documentclass[12pt]{amsart}
\usepackage{amsfonts, amssymb, latexsym, hyperref, xcolor, tikz,tikz-cd,transparent,adjustbox}
\usepackage{ytableau}
\usepackage{quiver}

\newtheorem{theorem}{Theorem}[section]
\newtheorem{proposition}[theorem]{Proposition}
\newtheorem{lemma}[theorem]{Lemma}
\newtheorem{maindef}[theorem]{Main Definition}

\newtheorem{defthm}[theorem]{Definition-Theorem}

\newtheorem{claim}[theorem]{Claim}
\newtheorem*{claim*}{Claim}
\newtheorem{corollary}[theorem]{Corollary}
\newtheorem{Main Conjecture}[theorem]{Main Conjecture}

\newtheorem{conjecture}[theorem]{Conjecture}
\newtheorem{problem}[theorem]{Problem}
\theoremstyle{definition}
\newtheorem{definition}[theorem]{Definition}
\theoremstyle{remark}

\newtheorem{example}[theorem]{Example}

\newtheorem{remark}[theorem]{Remark}
\theoremstyle{plain}

\newtheorem{MainThm}[theorem]{Main Theorem}

\newcommand\complexes{{\mathbb C}}
\newcommand\integers{{\mathbb Z}}

\newcommand{\cellsize}{12}
\newlength{\cellsz} 
\newsavebox{\cell}
\sbox{\cell}{\begin{picture}(\cellsize,\cellsize)
\put(0,0){\line(1,0){\cellsize}}
\put(0,0){\line(0,1){\cellsize}}
\put(\cellsize,0){\line(0,1){\cellsize}}
\put(0,\cellsize){\line(1,0){\cellsize}}
\end{picture}}
\newcommand\cellify[1]{\def\thearg{#1}\def\nothing{}%
\ifx\thearg\nothing
\vrule width0pt height\cellsz depth0pt\else
\hbox to 0pt{\usebox{\cell} \hss}\fi%
\vbox to \cellsz{
\vss
\hbox to \cellsz{\hss$#1$\hss}
\vss}}
\newcommand\tableau[1]{\vtop{\let\\\cr
\baselineskip -16000pt \lineskiplimit 16000pt \lineskip 0pt
\ialign{&\cellify{##}\cr#1\crcr}}}

\newcommand{\sellsize}{5}
\newlength{\sellsz} 
\newsavebox{\sell}
\sbox{\sell}{\begin{picture}(\sellsize,\sellsize)
\put(0,0){\line(1,0){\sellsize}}
\put(0,0){\line(0,1){\sellsize}}
\put(\sellsize,0){\line(0,1){\sellsize}}
\put(0,\sellsize){\line(1,0){\sellsize}}
\end{picture}}
\newcommand\sellify[1]{\def\thearg{#1}\def\nothing{}%
\ifx\thearg\nothing
\vrule width0pt height\sellsz depth0pt\else
\hbox to 0pt{\usebox{\sell} \hss}\fi%
\vbox to \sellsz{
\vss
\hbox to \sellsz{\hss$#1$\hss}
\vss}}
\newcommand\stableau[1]{\vtop{\let\\\cr
\baselineskip -16000pt \lineskiplimit 16000pt \lineskip 0pt
\ialign{&\sellify{##}\cr#1\crcr}}}

\newcommand{\gap}{\hspace{1in} \\ \vspace{-.2in}}

\hypersetup{colorlinks, linkcolor={red!80!black},
citecolor={blue!80!black}, urlcolor={blue!80!black}}

\newcommand{\excise}[1]{}

\font\co=lcircle10

\def\jr{\smash{\raise2pt\hbox{\co \rlap{\rlap{\char'005} \char'007}}
               \raise6pt\hbox{\rlap{\vrule height5pt}}
               \raise2pt\hbox{\rlap{\hskip4pt \vrule height0.4pt depth0pt
                width5.7pt}}
               \raise2pt\hbox{\rlap{\hskip-9.5pt \vrule height.4pt depth0pt
                width6.2pt}}
               \lower6pt\hbox{\rlap{\vrule height4.5pt}}}}
\def\rj{\smash{\raise2pt\hbox{\co \rlap{\rlap{\char'004} \char'006}}
               \raise6pt\hbox{\rlap{\vrule height5pt}}
               \raise2pt\hbox{\rlap{\hskip4pt \vrule height0.4pt depth0pt
                width5.7pt}}
               \raise2pt\hbox{\rlap{\hskip-9.5pt \vrule height.4pt depth0pt
                width6.2pt}}
               \lower6pt\hbox{\rlap{\vrule height4.5pt}}}}
\def\je{\smash{\raise2pt\hbox{\co \rlap{\rlap{\char'005}
                \phantom{\char'007}}}\raise6pt\hbox{\rlap{\vrule height5pt}}
               \raise2pt\hbox{\rlap{\hskip-9.5pt \vrule height.4pt depth0pt
                width6.2pt}}}}
\def\ej{\smash{\raise2pt\hbox{\co \rlap{\rlap{\char'004}\phantom{\char'006}}}
               \raise2pt\hbox{\rlap{\hskip-9.5pt \vrule height.4pt depth0pt
                width6.2pt}}
               \lower6pt\hbox{\rlap{\vrule height4.5pt}}}}
\def\er{\smash{\raise2pt\hbox{\co \rlap{\rlap{\phantom{\char'005}} \char'007}}
               \raise2pt\hbox{\rlap{\hskip4pt \vrule height0.4pt depth0pt
                width5.7pt}}
               \lower6pt\hbox{\rlap{\vrule height4.5pt}}}}
\def\re{\smash{\raise2pt\hbox{\co \rlap{\rlap{\phantom{\char'004}} \char'006}}
               \raise6pt\hbox{\rlap{\vrule height5pt}}
               \raise2pt\hbox{\rlap{\hskip4pt \vrule height0.4pt depth0pt
               width5.7pt}}}}
\def\+{\smash{\lower6pt\hbox{\rlap{\vrule height17pt}}
                \raise2pt
                \hbox{\rlap{\hskip-9pt \vrule height.4pt depth0pt
                width18.7pt}}}}
\def\hor{\smash{\raise2pt\hbox{\rlap{\hskip-9.5pt \vrule height.4pt depth0pt
                width19.2pt}}}}
\def\ver{\smash{\lower6pt\hbox{\rlap{\vrule height17pt}}}}
\def\ho{\smash{\hbox{\rlap{\vrule height5pt}}
                \raise2pt
                \hbox{\rlap{\hskip-9pt \vrule height.4pt depth0pt
                width18.7pt}}}}

\def\textcross{\ \smash{\lower4pt\hbox{\rlap{\hskip4.15pt\vrule height14pt}}
                \raise2.8pt\hbox{\rlap{\hskip-3pt \vrule height.4pt depth0pt
                width14.7pt}}}\hskip12.7pt}
\def\textelbow{\ \hskip.1pt\smash{\raise2.75pt%
                \hbox{\co \hskip 4.15pt\rlap{\rlap{\char'004} \char'006}
                \lower6.8pt\rlap{\vrule height3.5pt}
                \raise3.6pt\rlap{\vrule height3.5pt}}
                \raise2.8pt\hbox{%
                  \rlap{\hskip-7.15pt \vrule height.4pt depth0pt width3.5pt}%
                  \rlap{\hskip4.05pt \vrule height.4pt depth0pt width3.5pt}}}
                \hskip8.7pt}

\newcommand{\C}{\mathbb{C}}
\newcommand{\Z}{\mathbb{Z}}

\newcommand{\inv}{^{-1}}

\newcommand{\desc}{\mathrm{Desc}}
\newcommand{\init}{\mathrm{init}}

\DeclareRobustCommand
  \Compactcdots{\mathinner{{\cdot}\mkern1mu{\cdot}\mkern1mu{\cdot}}}

\DeclareRobustCommand
  \Compactldots{\mathinner{{\ldotp}\mkern1mu{\ldotp}\mkern1mu{\ldotp}}}

\begin{document}
\pagestyle{plain}
\title{Representations from matrix varieties, and filtered RSK}
\author{Abigail Price}
\author{Ada Stelzer}
\author{Alexander Yong}
\address{Dept.~of Mathematics, U.~Illinois at Urbana-Champaign, Urbana, IL 61801, USA}
\email{price29@illinois.edu, astelzer@illinois.edu, ayong@illinois.edu}
\date{October 6, 2025}

\begin{abstract}

Matrix Schubert varieties [Fulton ’92] carry natural actions of Levi groups. Their coordinate rings thereby decompose as direct sums of irreducible representations. What is a combinatorial rule for the multiplicities of these irreducibles? When the Levi group is a torus, [Knutson--Miller '04] gives an answer. We offer a general solution, a common refinement of the multigraded Hilbert series, the Cauchy identity, and the Littlewood-Richardson rule. Our result applies to 
``bicrystalline'' algebraic varieties; we define these using the operators of [Kashiwara '95] and of 
[Danilov--Koshevoi '05, van Leeuwen '06]. The proof introduces a ``filtered’’ generalization of the RSK correspondence.
\end{abstract}

\maketitle

\section{Introduction}\label{sec:intro}

\subsection{Overview}
Fix finite-dimensional vector spaces  $V$ and $W$ over ${\mathbb C}$; ${\bf GL}=GL(V)\times GL(W)$ acts naturally on the tensor product $V\boxtimes W$, as does the maximal torus ${\bf T}=T(V)\times T(W)$. These groups then act on the coordinate ring ${\rm Sym}^*(V\boxtimes W)$ and the \emph{Cauchy identity} equates the ${\bf T}$-weight space decomposition of this
${\bf GL}$-module to its decomposition into ${\bf GL}$-irreducibles.
This identity is equivalent to \emph{Schur--Weyl duality} of irreducible representations of general linear and symmetric groups, and to the \emph{First Fundamental Theorem of Invariant Theory for $GL_n$} (see, e.g., \cite{Howe}). In symmetric function theory
it implies the existence of a bijection between nonnegative integer matrices and pairs of semistandard Young tableaux, which is realized by the \emph{Robinson--Schensted--Knuth correspondence} (RSK).

We propose a two-fold generalization. In one direction, replace $GL(V)$ with the \emph{Levi subgroups} $L$ associated to parabolic subgroups $P\subseteq GL(V)$. These Levi groups give an
``interpolation'' of reductive groups between $GL(V)$ and $T(V)$; if $P=GL(V)$ then $L=GL(V)$ and if $P=B$ is a Borel subgroup containing $T(V)$, then $L=T(V)$. We replace ${\bf GL}$ with ${\bf L}=L(V)\times L(W)$. The second direction lifts from (multi)linear algebra to algebraic geometry:
replace the affine space $V\boxtimes W$ with a variety ${\mathfrak X}\subseteq V\boxtimes W$ on which ${\bf L}$ acts rationally. Hence $\mathbf{L}$ acts on the coordinate ring $\C[{\mathfrak X}]$. In commutative algebra, the 
${\bf T}$-character is the \emph{(${\mathbb Z}^m\times {\mathbb Z}^n$- multigraded) Hilbert series} of $\C[{\mathfrak X}]$. Writing this Hilbert series in terms of irreducible characters of ${\bf L}$ \emph{abstractly} generalizes the
Cauchy identity.

Concretely, we generalize the Cauchy identity by introducing a  ``filtered'' refinement of RSK.  Kashiwara's \emph{crystal operators} (see, e.g., \cite{BS}) ``pull-back'' to \emph{bicrystal operators} on the set ${\sf Mat}_{m, n}(\Z_{\geq 0})$ of nonnegative integer matrices, by work of Danilov--Koshevoi \cite{bicrystal2} and of van Leeuwen \cite{bicrystal1}. The Hilbert series of ${\mathbb C}[{\mathfrak X}]$ is the generating series of torus weights for a basis of \emph{standard monomials}. Identifying $V\boxtimes W$ with the space ${\sf Mat}_{m,n}$ of $m\times n$ complex matrices, this basis is encoded by ${\sf Mat}_{m, n}(\Z_{\geq 0})$. Define ${\mathfrak X}\subseteq {\sf Mat}_{m,n}$ to be ${\bf L}$-\emph{bicrystalline} if it is ${\bf L}$-stable and its standard monomial basis is closed under the bicrystal operators. Our main result is that for such ${\mathfrak X}$,  filtered RSK determines multiplicities of the irreducibles in the ${\bf L}$-module ${\mathbb C}[{\mathfrak X}]$. 

Questions about \emph{determinantal varieties} motivate this work. The space ${\mathfrak X}={\mathfrak X}_k$ of rank $\leq k$ matrices in ${\sf Mat}_{m,n}$ appears in invariant theory, representation theory, and algebraic geometry; see, e.g., \cite{Harris, Procesi, Weyman}. Much has been achieved in understanding them, including their Hilbert series. Each ${\mathfrak X}_k$ is ${\bf GL}$-bicrystalline. More generally, our main source of 
${\bf L}$-bicrystalline varieties, for various ${\bf L}\supseteq {\bf T}$, are ${\mathfrak X}\subset {\sf Mat}_{m,n}$ that are ${\bf B}=B_m\times B_n$ stable, where $B_m\leq GL_m$ and $B_n\leq GL_n$ are the Borel groups of lower triangular matrices. These (possibly reducible) ${\bf B}$-stable varieties include
the \emph{matrix Schubert varieties} of Fulton \cite{Fulton:duke}; our main results are new even for these cases.

This paper initiates our use of the ${\bf L}$-actions on matrix Schubert varieties to attack the open problem of 
determining their Betti numbers and minimal free resolutions, extending the \emph{Kempf--Lascoux--Weyman complexes} \cite{Weyman}. Much of this homological information is encoded in the $\mathbf{L}$-irreducible decomposition
of their coordinate rings, our focus here.
Our results simultaneously generalize the Cauchy identity, the Littlewood--Richardson rule, and the Knutson--Miller Hilbert series formula for matrix Schubert varieties \cite{KM:adv, Knutson.Miller}.
  
\subsection{Filtered RSK}\label{subsec:introcomb}

In the language of symmetric polynomials, the \emph{Cauchy identity} states
\begin{equation}\label{eqn:Cauchy}
\prod_{\mbox{\tiny\ensuremath{\begin{array}{c} 1\leq i\leq m \\ 1\leq j\leq n\end{array}}}} \frac{1}{1-x_i y_j} = \sum_{\lambda} s_{\lambda}(x_1,\ldots,x_m)
s_{\lambda}(y_1,\ldots,y_n);
\end{equation}
where $\lambda$ is an integer partition and $s_{\lambda}(t_1,\ldots, t_k)$ is the \emph{Schur polynomial}, i.e., the generating series for \emph{semistandard Young tableaux} (SSYT) of shape $\lambda$ with entries
in $[k]:~=~\{1,2,\ldots,k\}$:
\begin{equation}
\label{eqn:Schurdef}
s_{\lambda}(t_1,\ldots,t_k):=\sum_{T} \prod_{i=1}^k t_i^{\#i \text{ in $T$}}.
\end{equation}
To prove \eqref{eqn:Cauchy} combinatorially, one defines a map {\sf RSK} that sends a matrix $M\in{\sf Mat}_{m, n}(\Z_{\geq0})$ to a pair of SSYT. This map uses the \emph{row word} and \emph{column word} of $M$.

\begin{definition}\label{def:matrixwords}
    The \emph{row word} of a matrix $M = [m_{ij}]\in {\sf Mat}_{m,n}({\mathbb Z}_{\geq 0})$, denoted ${\sf row}(M)$, records $m_{ij}$ copies of the row index $i$ for each entry $m_{ij}$ of $M$. The entries are read along columns top-to-bottom, left to  right. The \emph{column word} ${\sf col}(M)$ is 
    formed by recording $m_{ij}$ copies of the column index $j$ for each entry $m_{ij}$ of $M$. The entries are read across the rows in English reading order.
\end{definition}
\begin{example}
    If $M = \begin{bmatrix}0 & 1 \\ 2 & 3\end{bmatrix}$ then ${\sf row}(M) = 221222$ and ${\sf col}(M) = 211222$.  
\end{example}

In Section~\ref{sec:tableau} we recall the \emph{reading word} of a semistandard tableau $T$, denoted ${\sf word}(T)$ (Definition~\ref{def:readingword}), the \emph{Knuth equivalence} relation $\sim_K$ on words (Definition~\ref{def:knuthequiv}), and the fact that the Knuth equivalence class of any word contains ${\sf word}(T)$ for exactly one semistandard Young tableau $T$ (Theorem~\ref{thm:jdt}). With these notions, it is convenient to phrase the traditional RSK correspondence in this manner:

\begin{definition}[RSK correspondence]\label{alg:rsk}
    Let $M\in {\sf Mat}_{m,n}({\mathbb Z}_{\geq 0})$, and $P$ and $Q$ be the unique SSYT
    with ${\sf row}(M)\sim_K{\sf word}(P)$ and ${\sf col}(M)\sim_K{\sf word}(Q)$. Define 
    ${\sf RSK}(M) = (P|Q)$.
\end{definition}

Although non-obvious from Definition~\ref{alg:rsk}, ${\sf RSK}$ gives a bijection between matrices $M\in {\sf Mat}_{m,n}({\mathbb Z}_{\geq 0})$ and pairs $(P|Q)$ of semistandard tableaux of the same shape~$\lambda$. 
Grouping monomials corresponding to $M$ according to the shape of the tableaux in ${\sf RSK}(M)$ proves the classical Cauchy identity (\ref{eqn:Cauchy}). This shape-symmetry of ${\sf RSK}$ is not assumed, nor even true,
in what follows.

We will generalize the proof of (\ref{eqn:Cauchy}) by refining ${\sf RSK}$. This refinement uses the \emph{filtration} of a word $w$ in $[m]$ for a sequence of nonnegative integers ${\textbf I} = \{0 = i_0<\dots<i_r = m\}$. 

\begin{definition}\label{def:filtration}
    The \emph{${\mathbf I}$-filtration} of a word $w$ is the tuple of words ${\sf filter}_{{\mathbf I}}(w) = (w^{(1)},\dots, w^{(r)})$, where $w^{(k)}$ is the subword of $w$ consisting of all letters in the interval $[i_{k-1}+1, i_k]$.
\end{definition}

\begin{maindef}[Filtered RSK]\label{alg:filteredrsk}
    Let $M\in {\sf Mat}_{m,n}({\mathbb Z}_{\geq 0})$ and fix two integer sequences 
    \begin{equation}\label{eqn:twointseq}
    {\mathbf I} = \{0 = i_0<\dots<i_r = m\} \text{ \  and \ } {\mathbf J} = \{0 = j_0<\dots<j_s = n\}.
    \end{equation}
    Define ${\sf filterRSK}_{{\mathbf I}|{\mathbf J}}(M)= (\underline{P}|\underline{Q}):= (P^{(1)},\dots,P^{(r)}|Q^{(1)},\dots,Q^{(s)})$,
    where $P^{(a)}$ and $Q^{(b)}$ are the unique SSYT 
    with $P^{(a)}\sim_K {\sf filter}_{{\mathbf I}}({\sf row}(M))^{(a)}$ and $Q^{(b)}\sim_K {\sf filter}_{{\mathbf J}}({\sf col}(M))^{(b)}$, respectively.
\end{maindef}

Using \emph{row insertion} one can algorithmically compute the tableaux appearing in the definitions
of {\sf RSK} and {\sf filterRSK}; see Section~\ref{sec:tableau} and specifically Proposition~\ref{prop:filterrskalg}.

\begin{example}
Let ${\mathbf I} = \{0,1,3\} = {\mathbf J}$ and let  $M = \begin{bmatrix}
        0 & 2 & 0 \\
        2 & 1 & 0 \\
        2 & 0 & 0
    \end{bmatrix}$. Then ${\sf row}(M)= 2233112$ and ${\sf col}(M) = 2211211$, so \[{\sf filter}_{\mathbf{I}}({\sf row}(M)) = (11,22332) \text{\ \  and \ \ } {\sf filter}_{\mathbf{J}}({\sf col}(M)) = (1111,222).\] Now, consider the tuple of tableaux
    \[(P^{(1)},P^{(2)}|Q^{(1)},Q^{(2)})=\left(\ \tableau{1 & 1}, \tableau{2 & 2 & 2 & 3\\ 3}\ \bigg\vert \ 
        \tableau{1 & 1 & 1 & 1}, \tableau{2 & 2 & 2}\ \right).\]
    The reading words of these tableaux are $(11,32223|1111,222)$. Since $32223\sim_K 22332$, ${\sf filterRSK(M)} = (P^{(1)},P^{(2)}|Q^{(1)},Q^{(2)})$. 
\end{example}

\subsection{The main result}\label{subsec:intromain}
Let $GL_k=GL_k({\mathbb C})$ be the general linear group of invertible $k\times k$ matrices. Then ${\bf GL}:=GL_m\times GL_n$ acts on ${\sf Mat}_{m,n}$ by row and column
operations, i.e., if $(p,q)\in {\bf GL}$ and $M\in {\sf Mat}_{m,n}$,
\begin{equation}
\label{eqn:howmult}
(p,q)\cdot M= p^{-1}M(q^{-1})^T.
\end{equation}
The coordinate ring ${\mathbb C}[{\sf Mat}_{m,n}]$ is thus endowed with a ${\bf GL}$-module structure: if $(p,q)\in {\bf GL}$, $M\in {\sf Mat}_{m,n}$, and $f\in {\mathbb C}[{\sf Mat}_{m,n}]$ then
\begin{equation}
\label{eqn:coordringaction}
((p, q)\cdot f)(M) := f(p Mq^T).
\end{equation}
If $f$ is a degree-$d$ polynomial then so is $(p, q)\cdot f$ for all $(p, q)\in\mathbf{GL}$, so in fact the grade-$d$ component ${\mathbb C}[{\sf Mat}_{m,n}]_d$ of ${\mathbb C}[{\sf Mat}_{m,n}]$ is a finite-dimensional representation of ${\bf GL}$. Any finite-dimensional representation of $GL_k$ is a direct sum of irreducible \emph{Weyl modules} $V_\lambda$, where $\lambda$ is a partition with at most $k$ rows, i.e., $\ell(\lambda)\leq k$. The irreducible representations of ${\bf GL}$ are $V_{\lambda}\boxtimes V_{\mu}$, and
there exist $c_{\lambda|\mu}\in {\mathbb Z}_{\geq 0}$ indexed by partition-pairs $(\lambda|\mu)$ 
such that
\begin{equation}\label{eqn:repdecomp}
    {\mathbb C}[{\sf Mat}_{m,n}]_d \cong_{\mathbf{GL}} \bigoplus_{\lambda|\mu  : \  \ell(\lambda)\leq m, 
    \ell(\mu)\leq n}(V_{\lambda}\boxtimes V_{\mu})^{\oplus c_{\lambda|\mu}}.
\end{equation}

The ${\bf T}$-character of ${\mathbb C}[{\sf Mat}_{m,n}]$ is the lefthand side of \eqref{eqn:Cauchy}. Since the character of $V_{\lambda}\boxtimes V_{\mu}$ is $s_{\lambda}(x_1,\ldots,x_m)s_{\mu}(y_1,\ldots,y_n)$, it follows
that there is an identity of polynomials
\[\prod_{\mbox{\tiny\ensuremath{\begin{array}{c} 1\leq i\leq m \\ 1\leq j\leq n\end{array}}}} \frac{1}{1-x_i y_j} = 
\sum_{\lambda|\mu} c_{\lambda|\mu}s_{\lambda}(x_1,\ldots,x_m)
s_{\mu}(y_1,\ldots,y_n).\]
Treating the lefthand side as the generating series for nonnegative matrices, and using the definition (\ref{eqn:Schurdef}) in terms of semistandard Young tableaux, one deduces from ${\sf RSK}$ that $c_{\lambda|\mu}=0$ unless
$\lambda=\mu$, in which case $c_{\lambda|\lambda}=1$, in agreement with (\ref{eqn:Cauchy}).

We now generalize (\ref{eqn:Cauchy}) and the proof just outlined to a type of (possibly reducible) variety ${\mathfrak X}\subseteq {\sf Mat}_{m,n}$. We use reductive groups that ``interpolate'' between ${\bf T}$ and ${\bf GL}$:
\begin{definition}\label{def:levi}
    Let $\mathbf{L}_{{\mathbf I}|{\mathbf J}} = L_{\mathbf I}\times L_{\mathbf J}\leq \mathbf{GL}$, where
    $L_{\mathbf I} = GL_{i_1-i_0}\times GL_{i_2-i_1}\times...\times GL_{i_r-i_{r-1}}$.
\end{definition}

By subgroup restriction, each ${\bf L}_{{\mathbf I}|{\mathbf J}}$ acts on ${\sf Mat}_{m,n}$ via (\ref{eqn:howmult}). 
A variety ${\mathfrak X}\subset {\sf Mat}_{m,n}$ is \emph{$\mathbf{L}_{{\mathbf I}|{\mathbf J}}$-stable} if this
action fixes ${\mathfrak X}$.

Our technical condition on ${\mathfrak X}$ concerns a compatibility of 
Gr\"obner basis theory with Kashiwara's crystal basis theory. Identify
${\mathbb C}[{\sf Mat}_{m,n}]\cong {\mathbb C}[z_{ij}]_{1\leq i \leq m,\  1\leq j\leq n}:=S$.
Let $I=I({\mathfrak X})$ be the ideal of $X$, hence ${\mathbb C}[{\mathfrak X}]\cong S/I(\mathfrak{X})$. Associated to any term  order $<$ on the monomials of $S$ is a set of standard monomials ${\sf Std}_{<}(S/I(\mathfrak{X}))$ (taken with coefficient~$1$) 
of~$S$. These monomials form a vector space basis of $S/I(\mathfrak{X})$ over ${\mathbb C}$; see Theorem~\ref{thm:stdbasis}. We identify each ${\sf m}\in {\sf Std}_{<}(S/I(\mathfrak{X}))$ with its exponent vector, represented as 
a matrix in ${\sf Mat}_{m,n}({\mathbb Z}_{\geq 0})$. Thus under {\sf RSK}, ${\sf m}$ corresponds to a pair of SSYT $(P|Q)$ of the same shape~$\lambda$. Kashiwara's raising and lowering operators $e_i, f_i$ on SSYT define a connected crystal graph structure on the SSYT of shape $\lambda$; see Theorem~\ref{thm:crystalgraph}.
Danilov--Koshevoi \cite{bicrystal2} and van Leeuwen \cite{bicrystal1}  
explain how Kashiwara's operators ``pull-back" to four \emph{bicrystal operators} $e_i^{\sf row}, e_j^{\sf col}, f_i^{\sf row}, f_j^{\sf col}$ on 
${\sf Mat}_{m,n}(\Z_{\geq 0})$; see Definition~\ref{def:feb11bicrystal}.

\begin{definition}
${\mathfrak X}\subseteq {\sf Mat}_{m,n}$ is ${\bf L}_{{\mathbf I}|{\mathbf J}}$-\emph{bicrystal closed} if there exists a term order $<$ such that $e_i^{\sf row}({\sf m}),f_i^{\sf row}({\sf m}) \in  {\sf Std}_{<}(S/I(\mathfrak{X}))\cup \{\varnothing\}$ for ${\sf m} \in {\sf Std}_{<}(S/I(\mathfrak{X})), i\not\in \mathbf{I}$,
and $e_j^{\sf col}({\sf m}),f_j^{\sf col}({\sf m})\in {\sf Std}_{<}(S/I(\mathfrak{X}))\cup\{\varnothing\}$ for ${\sf m} \in {\sf Std}_{<}(S/I(\mathfrak{X})), j\not\in \mathbf{J}$.
\end{definition}
    
\begin{definition} $\mathfrak{X}\subseteq{\sf Mat}_{m, n}$ is ${\mathbf L}_{{\mathbf I}|{\mathbf J}}$-\emph{bicrystalline} if
it is $\mathbf{L}_{{\mathbf I}|{\mathbf J}}$-stable and bicrystal closed.
\end{definition}

If ${\mathfrak X}\subseteq {\sf Mat}_{m, n}$ is $\mathbf{L}_{{\mathbf I}|{\mathbf J}}$-stable, $\C[{\mathfrak X}]$ is a $\mathbf{L}_{{\mathbf I}|{\mathbf J}}$-module by (\ref{eqn:coordringaction}). Levi groups being reductive, their representations are completely reducible, that is, their representations are a direct sum of irreducibles. The irreducible representations of $\mathbf{L}_{{\mathbf I}|{\mathbf J}}$ are indexed by tuples of partitions $(\underline{\lambda}|\underline{\mu})$. Their characters are products of \emph{split-Schur polynomials} $s_{\underline{\lambda}}(\mathbf{x})s_{\underline{\mu}}(\mathbf{y})$, where
\[s_{\underline{\lambda}}(\mathbf{x}):=s_{\lambda^{(1)}}(x_1,\dots, x_{i_1})s_{\lambda^{(2)}}(x_{i_1+1},\dots, x_{i_2})\cdots s_{\lambda^{(r)}}(x_{i_{r-1}+1},\dots, x_m).\] 

There is a formula for the ${\bf T}$-character of ${\mathbb C}[{\mathfrak X}]$, denoted $\chi_{{\mathbb C}[{\mathfrak X}]},$ using the basis of standard monomials. If 
${\sf m}=\prod_{1\leq i\leq m, 1\leq j\leq n} z_{ij}^{a_{ij}}\in {\sf Std}_{<}(S/I(\mathfrak{X}))$ 
let
\begin{equation}
\label{eqn:whatiswt}
{\sf wt}({\sf m})=\prod_{1\leq i\leq m, 1\leq j\leq n} (x_i y_j)^{a_{ij}}.
\end{equation}
The righthand side of (\ref{eqn:char=hilb}) is the ${\mathbb Z}^m\times {\mathbb Z}^n$-multigraded Hilbert series of ${\mathbb C}[{\mathfrak X}]$.  Equality follows from the general fact identifying Hilbert series with torus characters; see Section~\ref{sec:reptheory}:
\begin{equation}\label{eqn:char=hilb}
\chi_{{\mathbb C}[{\mathfrak X}]} = \sum_{{\sf m}\in {\sf Std}_{<}(S/I(\mathfrak{X}))} {\sf wt}({\sf m}).
\end{equation}

\begin{definition}\label{def:highestweight}
    The \textit{highest-weight tableau} of \emph{shape} $\lambda$ with entries from $[a, b]$ is the tableau $T_{\lambda, [a, b]}$ of shape $\lambda$ taking the constant value $a-1+i$ on each row $i$.
    The \textit{highest-weight tableau-tuple} $(T_{\underline\lambda}|T_{\underline\mu})$ of \emph{shape} $(\underline\lambda|\underline\mu)$ has components $T_{\underline\lambda}^{(k)} := T_{\lambda^{(k)}, [i_{k-1}+1, i_k]}$ and $T_{\underline\mu}^{(k)} := T_{\mu^{(k)}, [j_{k-1}+1, j_k]}$.
\end{definition}

\begin{MainThm}[Generalized Cauchy identity] \label{thm:main}
  If $\mathfrak{X}\subseteq{\sf Mat}_{m, n}$ is ${\bf L}_{{\mathbf I}|{\mathbf J}}$-bicrystalline, as witnessed by the term order $<$, let $c^{\mathfrak{X}}_{\underline\lambda|\underline\mu}\!=\!\#\{{\sf m}\in {\sf Std}_{<}(S/I(\mathfrak{X})) :  {\sf filterRSK}_{{\mathbf I}|{\mathbf J}}({\sf m})=
  (T_{\underline\lambda}|T_{\underline\mu})\}$. Then  
  \begin{equation}\label{eqn:themainequality}
  \chi_{{\mathbb C}[{\mathfrak X}]}
  = \sum_{\underline{\lambda}, \underline{\mu}} c^{\mathfrak X}_{\underline{\lambda}|\underline{\mu}}s_{\underline{\lambda}}(\mathbf{x})s_{\underline{\mu}}(\mathbf{y}).
  \end{equation}
\end{MainThm}

\begin{example}\label{exa:isgeneralizedLR}
For any ${\bf I}|{\bf J}$, ${\mathfrak X}={\sf Mat}_{m,n}$ is ${\bf L}_{{\mathbf I}|{\mathbf J}}$-stable, and, trivially, 
${\bf L}_{{\mathbf I}|{\mathbf J}}$-bicrystalline. Letting $c_{\underline\lambda|\underline\mu}:=c_{\underline\lambda|\underline\mu}^{{\sf Mat}_{m,n}}$, we have
\begin{equation}\label{eqn:filteredCauchy}
    \prod_{\mbox{\tiny\ensuremath{\begin{array}{c} 1\leq i\leq m \\ 1\leq j\leq n\end{array}}}} \frac{1}{1-x_i y_j} = \sum_{\underline{\lambda},\underline{\mu}} c_{\underline{\lambda}|\underline{\mu}}s_{\underline{\lambda}}(\mathbf{x})
    s_{\underline{\mu}}(\mathbf{y}).
\end{equation}
\end{example}

\begin{remark}[\eqref{eqn:filteredCauchy} and Littlewood--Richardson coefficients]
Let ${\bf I}=\{0<t<m\}$ and ${\bf J}=\{0,n\}$. In this case, \eqref{eqn:filteredCauchy} becomes
\[
\prod_{\mbox{\tiny\ensuremath{\begin{array}{c} 1\leq i\leq m \\ 1\leq j\leq n\end{array}}}} \frac{1}{1-x_i y_j} = \sum_{\alpha,\beta} 
c_{\alpha,\beta}^{\lambda}
s_{\alpha}(x_1,\ldots,x_t)
s_{\beta}(x_{t+1},x_{t+2},\ldots,
x_m)
s_{\lambda}(y_1,\ldots,y_n),
\]
where $c_{\alpha,\beta}^{\lambda}$ is the \emph{Littlewood--Richardson coefficient} (in its coproduct role in the
Hopf algebra of symmetric functions); see Example~\ref{LR}.
So, Theorem~\ref{thm:main} is a generalized Littlewood--Richardson rule. We speculate that many properties
of Littlewood--Richardson coefficients hold, in ``good'' cases, for  
$c_{\underline\lambda|\underline\mu}^{\mathfrak X}$. These include \emph{saturation}, \emph{semigroup}, and \emph{SNP} properties, by analogy with work of
Knutson--Tao \cite{Knutson.Tao}, Zelevinsky \cite{Zelevinsky}, and of Monical, Tokcan with the third author \cite{MTY}, respectively.
\end{remark}

\subsection{Application to determinantal varieties}
Let $B_k\subseteq GL_k$ denote the Borel group of lower triangular matrices, and let $\mathbf{B} = B_m\times B_n\leq \mathbf{GL}$. The $\mathbf{B}$-stable varieties $\mathfrak{X}\subseteq{\sf Mat}_{m, n}$ are known; they are finite unions of the $\mathbf{B}$-orbit closures of partial permutation matrices. These $\mathbf{B}$-orbit closures are called \emph{matrix Schubert varieties} and have been extensively studied since their introduction by Fulton in \cite{Fulton:duke}. Knutson--Miller (\cite{Knutson.Miller}, \cite{KM:adv}) and Knutson \cite{Knutson} (see also \cite{KMY}) describe their standard monomials  with respect to appropriate term orders.

We characterize the groups $\mathbf{L}_{{\mathbf I}|{\mathbf J}}$ acting on a given $\mathbf{B}$-stable variety $\mathfrak{X}\subseteq{\sf Mat}_{m, n}$ and show they are bicrystalline, so ${\sf filterRSK}$ computes their $\mathbf{L}_{{\mathbf I}|{\mathbf J}}$-characters by Theorem~\ref{thm:main}. The theorem statement below, our second main result, uses notation for partial permutations and their descents from Section~\ref{subsec:MSVs}.

\begin{MainThm}\label{thm:crystalclosure}
  Let $\mathfrak{X}$ be a union of the matrix Schubert varieties defined by partial permutations $w_{(1)},w_{(2)},\dots, w_{(k)}$. Then $\mathfrak{X}$ is $\mathbf{L}_{{\mathbf I}|{\mathbf J}}$-stable if the descent positions satisfy $\desc_{{\sf row}}(w_{(i)})\subseteq {\bf I}$ and $\desc_{{\sf col}}(w_{(i)})\subseteq {\bf J}$ for $1\leq i\leq k$. Moreover $\mathfrak{X}$ is ${\bf L}_{{\mathbf I}|{\mathbf J}}$-bicrystalline.
\end{MainThm}

Theorem~\ref{thm:crystalclosure} generalizes the study of $\mathbf{GL}$-stable varieties $\mathfrak{X}\subset{\sf Mat}_{m, n}$, which are exactly the \emph{classical determinantal varieties} $\mathfrak{X}_k$ of rank $\leq k$ matrices. The $\mathbf{T}$-characters of these varieties have been long studied (see, e.g., \cite{Abhyankar, dCP, DRS, Galigo, Ghorpade, Herzog.Trung}). See Example~\ref{exa:classical}.

\begin{example}
Let ${\mathfrak X}$ be the space of $4\times 4$ matrices $M$ such that its northwest $1\times 1$ matrix has rank
$0$ and its northwest $4\times 4$
submatrix has rank $\leq 3$. In the notation of Section~\ref{subsec:MSVs}, this is the matrix Schubert variety
$\mathfrak{X}_w$ where  $w$ is the partial permutation $213\infty$. Now, $\mathfrak{X}_w$ is stable under the action of the Levi group $\mathbf{L}_{{\mathbf I}|{\mathbf J}}$ with ${\bf I} = \{0, 1, 4\} = {\bf J}$:  applying row operations on the first row (respectively, column), or on the second through fourth rows (respectively, columns), of a matrix $M\in {\mathfrak X}$ returns matrices in 
${\mathfrak X}$. Here \[I({\mathfrak X})=\left\langle z_{11}, \left|\begin{matrix} 
z_{11} & z_{12} & z_{13} & z_{14}\\
z_{21} & z_{22} & z_{23} & z_{24}\\
z_{31} & z_{32} & z_{33} & z_{34}\\
z_{41} & z_{42} & z_{43} & z_{44}\end{matrix}\right|
\right\rangle;\] this is an instance of a theorem of Fulton \cite{Fulton:duke}, recapitulated as
Theorem~\ref{theorem:Fultonsresult}. 
Let $<$ be the pure lexicographic order obtained by setting $z_{ab}>z_{cd}$ if $a<c$, or $a=c$ and $b>d$. 
Under this term order, the generators of $I$ form a Gr\"obner basis; this is a case of the Knutson--Miller Gr\"obner basis theorem from \cite{Knutson.Miller} which is restated here as Theorem~\ref{thm:KMgrob}.

Consequently,  both \[A = \begin{bmatrix}
    0 & 0 & 1 & 0\\
    0 & 2 & 0 & 0\\
    1 & 0 & 0 & 0\\
    0 & 0 & 0 & 0
\end{bmatrix} \text{\  \ and \ \ } B = \begin{bmatrix}
    0 & 1 & 0 & 0\\
    1 & 0 & 1 & 0\\
    0 & 1 & 0 & 0\\
    0 & 0 & 0 & 0
\end{bmatrix}\] correspond to standard monomials of $\C[{\mathfrak X}_w]$. 
Since ${\sf row}(A)={\sf col}(A)=3221$ and ${\sf row}(B)={\sf col}(B)=2132$, applying ${\sf filterRSK}_{{\mathbf I}|{\mathbf J}}$ to either matrix yields
\[T_{\underline\lambda|\underline\mu} = \left(\ \tableau{1}, \tableau{2 & 2\\ 3} \bigg\vert \ 
        \tableau{1}, \tableau{2 & 2\\ 3} \right).\] 
        $A$ and $B$ are the only two such matrices, so $c_{\underline{\lambda}|\underline{\mu}}^{{\mathfrak X}_w} = 2$. Additionally, $\chi_{{\mathbb C}[{\mathfrak X}_{213\infty}]}$ starts as:
\ytableausetup{boxsize=0.25em,aligntableaux=center}
\begin{multline*} 
s_{(\emptyset,\emptyset)}({\bf x})s_{(\emptyset,\emptyset)}({\bf y}) + s_{(\emptyset,\ydiagram{1})}({\bf x})s_{(\emptyset,\ydiagram{1})}({\bf y}) + s_{(\ydiagram{1},\emptyset)}({\bf x})s_{(\emptyset,\ydiagram{1})}({\bf y}) + s_{(\emptyset,\ydiagram{1})}({\bf x})s_{(\ydiagram{1},\emptyset)}({\bf y}) + s_{(\emptyset,\ydiagram{1,1})}({\bf x})s_{(\emptyset,\ydiagram{1,1})}({\bf y}) +\\
s_{(\ydiagram{1},\ydiagram{1})}({\bf x})s_{(\emptyset,\ydiagram{1,1})}({\bf y}) + s_{(\emptyset,\ydiagram{2})}({\bf x})s_{(\emptyset,\ydiagram{2})}({\bf y}) + s_{(\ydiagram{1},\ydiagram{1})}({\bf x})s_{(\emptyset,\ydiagram{2})}({\bf y}) + s_{(\ydiagram{2},\emptyset)}({\bf x})s_{(\emptyset,\ydiagram{2})}({\bf y}) + s_{(\emptyset,\ydiagram{1,1})}({\bf x})s_{(\ydiagram{1},\ydiagram{1})}({\bf y}) +\\ s_{(\emptyset,\ydiagram{2})}({\bf x})s_{(\ydiagram{1},\ydiagram{1})}({\bf y}) + s_{(\ydiagram{1},\ydiagram{1})}({\bf x})s_{(\ydiagram{1},\ydiagram{1})}({\bf y}) + s_{(\emptyset,\ydiagram{2})}({\bf x})s_{(\ydiagram{2},\emptyset)}({\bf y}) + s_{(\emptyset,\ydiagram{1,1,1})}({\bf x})s_{(\emptyset,\ydiagram{1,1,1})}({\bf y}) + s_{(\ydiagram{1},\ydiagram{1,1})}({\bf x})s_{(\emptyset,\ydiagram{1,1,1})}({\bf y}) +\\
s_{(\emptyset,\ydiagram{2,1})}({\bf x})s_{(\emptyset,\ydiagram{2,1})}({\bf y}) + s_{(\ydiagram{1},\ydiagram{1,1})}({\bf x})s_{(\emptyset,\ydiagram{2,1})}({\bf y}) + s_{(\ydiagram{1},\ydiagram{2})}({\bf x})s_{(\emptyset,\ydiagram{2,1})}({\bf y}) + s_{(\ydiagram{2},\ydiagram{1})}({\bf x})s_{(\emptyset,\ydiagram{2,1})}({\bf y}) + s_{(\emptyset,\ydiagram{1,1,1})}({\bf x})s_{(\ydiagram{1},\ydiagram{1,1})}({\bf y}) +\\
s_{(\emptyset,\ydiagram{2,1})}({\bf x})s_{(\ydiagram{1},\ydiagram{1,1})}({\bf y}) + s_{(\ydiagram{1},\ydiagram{1,1})}({\bf x})s_{(\ydiagram{1},\ydiagram{1,1})}({\bf y}) + s_{(\ydiagram{1},\ydiagram{2})}({\bf x})s_{(\ydiagram{1},\ydiagram{1,1})}({\bf y}) + s_{(\emptyset,\ydiagram{3})}({\bf x})s_{(\emptyset,\ydiagram{3})}({\bf y}) + s_{(\ydiagram{1},\ydiagram{2})}({\bf x})s_{(\emptyset,\ydiagram{3})}({\bf y}) +\\
s_{(\ydiagram{2},\ydiagram{1})}({\bf x})s_{(\emptyset,\ydiagram{3})}({\bf y}) + s_{(\ydiagram{3},\emptyset)}({\bf x})s_{(\emptyset,\ydiagram{3})}({\bf y}) + s_{(\emptyset,\ydiagram{2,1})}({\bf x})s_{(\ydiagram{1},\ydiagram{2})}({\bf y}) + s_{(\ydiagram{1},\ydiagram{1,1})}({\bf x})s_{(\ydiagram{1},\ydiagram{2})}({\bf y}) +\\
s_{(\emptyset,\ydiagram{3})}({\bf x})s_{(\ydiagram{1},\ydiagram{2})}({\bf y})+ s_{(\ydiagram{1},\ydiagram{2})}({\bf x})s_{(\ydiagram{1},\ydiagram{2})}({\bf y}) + s_{(\ydiagram{2},\ydiagram{1})}({\bf x})s_{(\ydiagram{1},\ydiagram{2})}({\bf y}) + s_{(\emptyset,\ydiagram{2,1})}({\bf x})s_{(\ydiagram{2},\ydiagram{1})}({\bf y}) +\\
s_{(\emptyset,\ydiagram{3})}({\bf x})s_{(\ydiagram{2},\ydiagram{1})}({\bf y}) + s_{(\ydiagram{1},\ydiagram{2})}({\bf x})s_{(\ydiagram{2},\ydiagram{1})}({\bf y}) + s_{(\emptyset,\ydiagram{3})}({\bf x})s_{(\ydiagram{3},\emptyset)}({\bf y}) + s_{(\emptyset,\ydiagram{2,1,1})}({\bf x})s_{(\emptyset,\ydiagram{2,1,1})}({\bf y}) +\\ \displaybreak[3]
s_{(\ydiagram{1},\ydiagram{1,1,1})}({\bf x})s_{(\emptyset,\ydiagram{2,1,1})}({\bf y}) +s_{(\ydiagram{1},\ydiagram{2,1})}({\bf x})s_{(\emptyset,\ydiagram{2,1,1})}({\bf y}) + s_{(\ydiagram{2},\ydiagram{1,1})}({\bf x})s_{(\emptyset,\ydiagram{2,1,1})}({\bf y}) + s_{(\emptyset,\ydiagram{2,1,1})}({\bf x})s_{(\ydiagram{1},\ydiagram{1,1,1})}({\bf y}) + s_{(\ydiagram{1},\ydiagram{2,1})}({\bf x})s_{(\ydiagram{1},\ydiagram{1,1,1})}({\bf y}) +\\ 
s_{(\emptyset,\ydiagram{2,2})}({\bf x})s_{(\emptyset,\ydiagram{2,2})}({\bf y}) +s_{(\ydiagram{1},\ydiagram{2,1})}({\bf x})s_{(\emptyset,\ydiagram{2,2})}({\bf y}) + s_{(\ydiagram{2},\ydiagram{2})}({\bf x})s_{(\emptyset,\ydiagram{2,2})}({\bf y}) +s_{(\emptyset,\ydiagram{3,1})}({\bf x})s_{(\emptyset,\ydiagram{3,1})}({\bf y}) +\\
s_{(\ydiagram{1},\ydiagram{2,1})}({\bf x})s_{(\emptyset,\ydiagram{3,1})}({\bf y}) + s_{(\ydiagram{2},\ydiagram{1,1})}({\bf x})s_{(\emptyset,\ydiagram{3,1})}({\bf y}) + s_{(\ydiagram{1},\ydiagram{3})}({\bf x})s_{(\emptyset,\ydiagram{3,1})}({\bf y}) +s_{(\ydiagram{2},\ydiagram{2})}({\bf x})s_{(\emptyset,\ydiagram{3,1})}({\bf y}) +\\
s_{(\ydiagram{3},\ydiagram{1})}({\bf x})s_{(\emptyset,\ydiagram{3,1})}({\bf y}) + s_{(\emptyset,\ydiagram{2,1,1})}({\bf x})s_{(\ydiagram{1},\ydiagram{2,1})}({\bf y}) + s_{(\ydiagram{1},\ydiagram{1,1,1})}({\bf x})s_{(\ydiagram{1},\ydiagram{2,1})}({\bf y}) + s_{(\emptyset,\ydiagram{2,2})}({\bf x})s_{(\ydiagram{1},\ydiagram{2,1})}({\bf y})+\\ s_{(\emptyset,\ydiagram{3,1})}({\bf x})s_{(\ydiagram{1},\ydiagram{2,1})}({\bf y})
+ 2s_{(\ydiagram{1},\ydiagram{2,1})}({\bf x})s_{(\ydiagram{1},\ydiagram{2,1})}({\bf y})+\cdots.
\end{multline*}
\ytableausetup{boxsize=normal,aligntableaux=top}
\end{example}
\subsection{Organization}\label{subsec:organization}
Section~\ref{sec:reptheory} presents algebraic preliminaries. We begin with general notions  about
(reductive) linear algebraic groups, their representations and characters, and the irreducible decomposition problem in the representation ring. Then we discuss Gr\"obner basis notions, specifically \emph{standard monomials} and computation of $\mathbf{T}$-characters. Finally, we specialize to the setting of this paper, explaining, e.g., why the characters of irreducible representations of ${\bf L}_{{\mathbf I}|{\mathbf J}}$ are the split-Schur polynomials $s_{\underline\lambda}({\bf x})s_{\underline\mu}(\mathbf{y})$. 

Section~\ref{sec:tableau} covers the tableau notions we need: semistandard Young tableaux and their reading words, Knuth equivalence, row insertion, and ${\sf filterRSK}$ in terms of insertion. 

Section~\ref{subsec:graphs} discusses Kashiwara's \emph{crystal graphs}. Section~\ref{subsec:crystals} reviews \emph{crystal graphs} on words and tableaux, stating and proving some preparatory results.
Section~\ref{subsec:matrixcrystals} concerns the \emph{bicrystal} structure on matrices found by Danilov--Koshevoi
\cite{bicrystal2} and van Leeuwen \cite{bicrystal1}. 
We phrase results in terms of local isomorphisms of pre-crystal graphs.

Section~\ref{subsec:filterings} introduces filterings, which combined with the build-up from
Sections~\ref{sec:tableau} and~\ref{subsec:graphs}, culminates in
Theorem~\ref{cor:filterRSKgraphiso}, the ``bicrystal categorification'' 
of Theorem~\ref{thm:main}. It generalizes results of \cite{bicrystal2, bicrystal1}. 
With this theorem, the Main Theorem~\ref{thm:main} follows, in view of the ideas laid out in Section~\ref{sec:reptheory}.

Section~\ref{sec:MSVs} concerns the main case of $\mathbf{B}$-stable varieties. We define matrix Schubert varieties $\mathfrak{X}_w$ and state basic results drawn from Fulton's \cite{Fulton:duke}. We then recall the description of the standard monomials of $\C[\mathfrak{X}_w]$ from Knutson--Miller's \cite{Knutson.Miller} and the extension to arbitrary $\mathbf{B}$-stable varieties from Knutson's \cite{Knutson}. This allows us to prove Theorem~\ref{thm:crystalclosure}, showing that the $\mathbf{L}_{{\mathbf I}|{\mathbf J}}$-characters of $\mathbf{B}$-stable varieties are computed by ${\sf filterRSK}$.

Section~\ref{sec:polytopal} offers final remarks. Proposition~\ref{thm:filterRSKpolyrule} shows that $c_{\underline\lambda|\underline\mu}$ counts lattice points in a polytope. 
It allows one to reprove the classical Cauchy identity and formulate a polytopal Littlewood--Richardson rule.
Proposition~\ref{thm:feb10poly} and Remark~\ref{remark:unionofpoly} give an extension for any
$c^\mathfrak{X}_{\underline\lambda|\underline\mu}$. We then rederive the  Cauchy-type identity for classical determinantal varieties.

Appendix~\ref{sec:appendix} elaborates on the connections between our main results for matrix Schubert varieties, prior work on their Hilbert series, and open problems regarding their minimal free resolutions.

\section{Background on representation theory and commutative algebra}
\label{sec:reptheory}

\subsection{Reductive groups, maximal tori, and characters}\label{subsec:repbackground}
We assume throughout that all groups $G$ have the structure of an affine variety over the base field $\C$ of complex numbers. These are the \emph{(complex) linear algebraic groups}, so called because they are isomorphic to closed subgroups of $GL_n$ \cite[Corollary 4.10]{Milne}. Our primary reference is \cite{Milne}.

Let $V$ be a $\C$-algebra. By a \emph{(rational) representation} of $G$ we mean a homomorphism \[\rho:G\to GL(V)\]
that is also a morphism of varieties, where $GL(V)$ is the (linear algebraic) group of invertible linear transformations of $V$. $V$ is a $\C[G]$-module via the action $g\cdot v = \rho(g)v$. 

A representation $V$ of $G$ is \emph{irreducible} if it contains no nontrivial $G$-invariant subspaces.  Irreducible representations of $G$ are finite-dimensional \cite[Corollary 4.8]{Milne}. We call $G$ over $\C$ \emph{reductive} if any representation of $G$ is a direct sum of irreducible representations.

\begin{example}
Our main examples of linear algebraic groups are the matrix groups $\mathbf{GL}$, $\mathbf{L}_{{\mathbf I}|{\mathbf J}}$, and $\mathbf{T}$ from the introduction. Additionally,
any finite group $G$ may be viewed as a linear algebraic group by first using Cayley's theorem to embed $G$ in a symmetric group ${\mathfrak S}_n$ and then embedding ${\mathfrak S}_n$ in $GL_n$ as permutation matrices. All of these groups are reductive.
\end{example}

When $G$ is a reductive group, the \emph{representation ring} ${\sf Rep}(G)$ consists of formal ${\mathbb Z}$-linear combinations of isomorphism classes of finite-dimensional representations, with direct sums and tensor products of representations as the ring operations. As a ${\mathbb Z}$-module, it has a basis given by classes of  irreducible representations of $G$.  Given a representation $V$ of a reductive group $G$, one asks \emph{how to express the class of $V$ in ${\sf Rep}(G)$ in terms of this basis?} Let us call this the
``decomposition problem''.

In the case where $G = T$ is an \emph{(algebraic) torus} (i.e., $T\cong (\C^\star)^k$ for some $k$), the irreducible representations are particularly easy to describe.

\begin{theorem}[{\cite[Theorem 12.12]{Milne}}]\label{fact:torusrep}
    The irreducible representations of $T\cong (\C^\star)^k$ are one-dimensional, indexed by integer $k$-tuples $\mathbf{a} = (a_1,\dots, a_k)\in{\mathbb Z}^k$. The action of $\mathbf{t} = (t_1,\dots, t_k)\in T$ on the irreducible representation $V_{\mathbf{a}}$ is given by $\mathbf{t}\cdot v = t_1^{a_1}\dots t_k^{a_k}v$. 
\end{theorem}

\begin{definition}\label{def:genchar}
    Let $T\cong (\C^\star)^k$ be a torus. Let 
    $V = \bigoplus_{\mathbf{a}}c_{\mathbf{a}}V_{\mathbf{a}}$ be a $T$-representation. The \emph{character} of $V$ is 
    \[\chi_V = \sum_{\mathbf{a}}c_{\mathbf{a}}\mathbf{t}^\mathbf{a}\in 
    \Z[[\mathbf{t}^{\pm 1} ]] = \Z[[t_1^{\pm 1},\dots, t_k^{\pm 1}]].\] 
\end{definition}

Assume for the moment that $V$ is a finite-dimensional $T$-representation via the action $\rho:T\to GL(V)$, so $\chi_V\in \Z[\mathbf t^{\pm 1}]$ is a Laurent polynomial.
The $T$-action on $V$ defines a \emph{multigrading} on $V$:
an element $v\in V$ is \emph{$T$-homogeneous of multidegree $\mathbf{a}$} if $\textrm{span}(v)$ is isomorphic to $V_{\mathbf{a}}$ as a $T$-representation.
The \emph{weight} of $v$ is then the monomial ${\sf wt}_T(v) = \mathbf{t}^\mathbf{a}\in\Z[\mathbf{t}^{\pm 1}]$. In later sections we will work with the equivalent \emph{combinatorial weight} of $v$, which is the $k$-tuple ${\sf cwt}_T(v) = \mathbf{a}\in\Z^k$.
By Theorem~\ref{fact:torusrep} and the fact that $T$ is reductive, $V$ must have a \emph{$T$-homogeneous basis} $\mathfrak{B}$ (meaning each $v\in\mathfrak{B}$ is $T$-homogeneous).
Then the elements of $\mathfrak{B}$ form a full basis of eigenvectors for every matrix $\rho(\mathbf{t})$ ($\mathbf{t}\in T$) and the eigenvalues of $\rho(\mathbf{t})$ are the weights $\mathbf{t}^\mathbf{a}$.
The character $\chi_V$ is thus concretely realized as the function mapping $\mathbf{t}$ to the sum of the eigenvalues of $\rho(\mathbf{t})$.
That is, we have the formula
\begin{equation}
\label{eqn:chardefsequiv}
\chi_V = \textrm{Trace}(\rho(\mathbf{t}))\in\Z[\mathbf{t}^{\pm1}].
\end{equation}

Theorem~\ref{fact:torusrep} helps to understand the decomposition problem for other connected reductive groups $G$ since any such group contains a nontrivial algebraic torus (\cite[Theorem 16.60]{Milne}). A \emph{split reductive group} is a pair $(G, T)$ of a connected reductive group $G$ and a maximal torus $T$ in $G$. The maximal tori in $G$ are all conjugate to one another, so the particular choice of $T$ matters little (\cite[Theorem 17.10, 21.43]{Milne}).

When $(G, T)$ is split reductive, any $G$-representation is also a $T$-representation, and taking characters then defines a ring homomorphism 
\begin{equation}\label{eqn:Gchar}
{\sf ch}:{\sf Rep}(G)\to\Z[\mathbf{t}^{\pm 1}].
\end{equation}
If $G = T$, ${\sf ch}$ is an isomorphism: ${\sf Rep}(T)\!\cong\! \Z[\mathbf{t}^{\pm 1}]$. In general, the image of ${\sf ch}$ is determined by the \emph{Weyl group} $W$ of $G$, which is $N_G(T)/T$ where $N_G(T)$ is the normalizer of $T$ in $G$.

\begin{theorem}[{\cite[Theorem 22.38]{Milne}}]\label{fact:chariso}
    If $(G, T)$ is a split reductive group, then the character map ${\sf ch}$ of (\ref{eqn:Gchar}) is an isomorphism onto the subring $\Z[\mathbf{t}^{\pm 1}]^W$ of $W$-invariant $T$-characters. In particular, the characters of irreducible $G$-representations $V_\lambda$ form a $\Z$-module basis for $\Z[\mathbf{t}^{\pm 1}]^W\cong {\sf Rep}(G)$.
\end{theorem}

\begin{definition}\label{def:whatischaracter}
    Let $(G, T)$ be a split reductive group with Weyl group $W$. The \emph{character} of a $G$-representation $V$ is defined to be its character $\chi_V\in\Z[\mathbf{t}^{\pm 1}]$ as a $T$-representation.
\end{definition}

In principle, Theorem~\ref{fact:chariso} solves the decomposition problem for $G$, albeit in a 
manner that does not combinatorially exhibit the positivity of the coefficients. Given a 
representation $V$ of $G$, first compute its character $\chi_V \in\Z[\mathbf{t}^{\pm 1}]$. 
By Theorem~\ref{fact:chariso}, $\chi_V$ in fact lies in $\Z[\mathbf{t}^{\pm 1}]^W\cong {\sf Rep}(G)$ and there is a unique expression $\chi_V = \sum_{\lambda} c_{\lambda}\chi_{V_\lambda}$ (where the sum is over irreducible representations of $G$). Hence, 
\[[V]=\sum_{\lambda} c_{\lambda}[V_{\lambda}]\in {\sf Rep}(G),\] 
or equivalently, $V\cong_G\bigoplus_{\lambda}V_{\lambda}^{\oplus c_\lambda}$. Determination
of $\chi_V \in\Z[\mathbf{t}^{\pm 1}]$ may be difficult in general. However, standard combinatorial commutative algebra will provide the method in the context of this paper, as summarized in Proposition~\ref{thm:bigsummary}.

We need the irreducible representations for a product of split reductive groups in terms of the irreducible representations of the factors.

\begin{definition}\label{def:Feb15exterior}
    Let $V$ and $V'$ be representations of linear algebraic groups $G$ and $G'$. The \emph{(exterior) tensor product representation} $V\boxtimes V'$ of $G\times G'$ is the tensor product of $V$ and $V'$ as vector spaces with the action
    $(g, g')\cdot(v\otimes v') = (g\cdot v)\otimes (g'\cdot v')$.
\end{definition}

The following type of statement is standard. See \cite[Lemma~68]{Steinberg} for a more general statement than the second sentence
(it assumes an algebraically closed field, which is our case). The compact Lie group case is textbook,
and sufficient for our application, since for general linear groups one reduces to compact case of $U_n$ by Weyl's unitarian trick.

\begin{proposition}\label{fact:prodrep}
    If $(G, T)$ and $(G', T')$ are split reductive groups, then $(G\times G', T\times T')$ is also split reductive. The irreducible representations of $G\times G'$ are exactly tensor products $V_\lambda\boxtimes V_\mu$ of irreducible representations of the factors.
\end{proposition}

In our application, $V=\bigoplus_{i\geq 0} V_i$ is a graded, infinite-dimensional vector space where each graded component $V_i$ is a finite-dimensional representation of $G$. Hence, our analysis occurs in the completion 
$\widehat{{\sf Rep}(G)}$ of ${\sf Rep}(G)$. This allows for infinite linear combinations of isomorphism classes of representations. This causes no difficulty as computing this class in $\widehat{{\sf Rep}(G)}$ may be interpreted as solving a $\Z_{\geq0}$-indexed set of finite-dimensional decomposition problems, one for each $V_i$.
(In the case $G=T$, $\widehat{{\sf Rep}(G)}\cong \Z[[t_1^{\pm 1},\dots, t_k^{\pm 1}]]$.)

\subsection{Representations from affine varieties}\label{subsec:stdmono}
Our $G$-modules of interest come as coordinate rings of algebraic varieties. 
We now review the basics.

\begin{definition}
    Let $G$ be a linear algebraic group. A \emph{$G$-variety} $\mathfrak{X}$ is a variety equipped with a rational action of $G$.
\end{definition}

\begin{example}
The space ${\sf Mat}_{m, n}$ is a $\mathbf{GL}$-variety under the usual multiplication action. By restriction ${\sf Mat}_{m, n}$ is also an $\mathbf{L}_{{\mathbf I}|{\mathbf J}}$-variety and a $\mathbf{T}$-variety.
Any linear algebraic group $G$, being an affine variety, is also a $G$-variety under the multiplication or conjugation actions. 
\end{example}

While a $G$-variety $\mathfrak{X}$ is usually not a representation of $G$ because it has no vector space structure, the \emph{coordinate ring} $\C[\mathfrak{X}]$ is a $G$-representation via the action
\begin{equation}
\label{eqn:Feb13abc}
g\cdot f(\mathbf{x}) := f(g\inv\cdot\mathbf{x}), \text{ \ for  $g\in G,\ \mathbf{x}\in\mathfrak{X},\ f\in\C[\mathfrak{X}]$.}
\end{equation}

If $G$ is reductive and $\mathfrak{X}$ is a $G$-variety, then $\C[\mathfrak{X}]$ decomposes into a direct sum of irreducible $G$-representations.
When $(G, T)$ is a split reductive group, the $T$-action on $\C[\mathfrak{X}]$ defines a multigrading, and the $G$-character (i.e., the $T$-character) is exactly the multigraded Hilbert series of $\mathfrak{X}$ in the sense of \cite{Miller.Sturmfels}. Section~\ref{subsec:repbackground} suggests how to compute the decomposition of $\C[\mathfrak{X}]$ into $G$-irreducibles. To proceed, we introduce coordinates on $\mathfrak{X}$ in order to give an explicit $T$-homogenous basis on $\C[\mathfrak{X}]$ and compute the $T$-character $\chi_{\C[{\mathfrak{X}}]}$.

Let $\mathfrak{X}$ be an affine variety, realized in coordinates as a closed subset of $\C^N$ for some $N$. Then the coordinate ring $\C[\mathfrak{X}]$ is realized as the quotient ring $S/I(\mathfrak{X})$, where $S := \C[z_1,\dots, z_N]$ and $I(\mathfrak{X})$ is the defining ideal of $\mathfrak{X}$. As a vector space, $S$ has an infinite basis given by all monomials in the variables $z_1,\dots, z_N$. The images of these monomials under the quotient map $q:S\to S/I(\mathfrak{X})$ then generate $\C[\mathfrak{X}]$ as a vector space, but they are generally not a basis: the relations in the quotient impose linear dependencies between some of the monomials. \emph{Gr\"obner bases} allow one to cut down this generating set to a vector space basis for $S/I(\mathfrak{X})$; we recall the basic definitions with \cite{CLO} as our reference.

\begin{definition}
    Let $I\subseteq S = \C[z_1,\dots, z_N]$ be an ideal, and let $<$ be a \emph{term order}. The \emph{initial ideal} of $I$ with respect to $<$ is the monomial ideal 
    $\init_<(I) := \langle{\sf LT}_<(f)|f\in I\rangle$, 
    where ${\sf LT}_<(f)$ denotes the lead term of $f$ under the term order $<$.
\end{definition}

\begin{definition}\label{def:stdmonomialdef}
    Let $I\subseteq S$ be an ideal and $<$ a term order as above. The set of \emph{standard monomials} for $S/I$ with respect to $<$ is the set ${\sf Std}_<(S/I)$ of all monomials not in $\init_<(I)$, taken with coefficient $1$. 
\end{definition}

\begin{theorem}[{\cite[Section 2.2]{CLO}}]\label{thm:stdbasis}
    Fix a term order $<$. ${\sf Std}_<(S/I)$ is a vector space basis for $S/I$.
\end{theorem}

\begin{definition}
    A \emph{Gr\"obner basis} for $I\subseteq S$ is a finite set $\{f_i\}_{i=1}^k\subseteq I$ such that $\init_<(I) = \langle{\sf LT}(f_i)|1\leq i\leq k\rangle$. 
\end{definition}

\begin{theorem}[Buchberger's Algorithm, {\cite[Section 1.3]{CLO}}]\label{thm:buchberger}
    Any generating set of an ideal $I\subseteq S$ can be extended to a Gr\"obner basis via a finite algorithm.
\end{theorem}

Thus, a vector space basis for $S/I$ is obtained by taking generators for $I$, extending them to a Gr\"obner basis with respect to some term order $<$, and describing ${\sf Std}_<(S/I)$ as those monomials not divisible by the lead term of any element of the Gr\"obner basis.

\begin{definition}\label{def:nattorus}
    The \emph{natural torus} on $\C^N$ is $T = (\C^\star)^N$, with the action $(\C^\star)^N\times\C^N\to\C^N$ given by scaling each component. The \emph{fine multigrading} on 
    $S = \C[z_1,\dots, z_N]$ is the multigrading induced by the natural torus action.
\end{definition}

The monomials of $\C[z_1,\dots, z_N]$ are all homogeneous with respect to the natural torus action of $(\C^\star)^N$ on $\C^N$. The next definition gives sufficient conditions for the standard monomials of a $T$-variety $\mathfrak{X}$ to be $T$-homogeneous.

\begin{definition}
    Let $T\cong (\C^\star)^k$ be a torus and let $\mathfrak{X}$ be a $T$-variety. An embedding $\iota:\mathfrak{X}\to\C^N$ is \emph{$T$-compatible} if there is an algebraic group homomorphism $\iota':T\to(\C^\star)^N$ such that the $T$-action on $\mathfrak{X}$ is given by the natural torus action on $\C^N$, restricted to the image of $\iota'$. 
\end{definition}

In summary, for split reductive groups $(G, T)$ we have the following method for decomposing representations arising from affine $G$-varieties:

\begin{proposition}\label{thm:bigsummary}
    Let $(G, T)$ be a split reductive group. Let $\mathfrak{X}$ be a $G$-variety with a $T$-compatible embedding such that $\C[\mathfrak{X}]\cong S/I(\mathfrak{X})$. Express the sum of the $T$-weights of the elements of ${\sf Std}_<(S/I(\mathfrak{X}))$ as a linear combination $\sum_{\lambda} c_{\lambda}\chi_{V_{\lambda}}$ of irreducible $G$-characters. Then 
    \[\C[\mathfrak{X}]\cong \bigoplus_{\lambda} V_{\lambda}^{\oplus c_{\lambda}}.\]
\end{proposition}
\begin{proof}
Theorem~\ref{thm:stdbasis} states that ${\sf Std}_<(S/I(\mathfrak{X}))$ forms a basis for $S/I(\mathfrak{X})\cong\C[\mathfrak{X}]$. The $T$-compatible hypothesis says that each of these basis elements is $T$-homogeneous. Thus we obtain the $T$-weight space decomposition of $\C[{\mathfrak{X}}]\cong S/I(\mathfrak{X})$ as guaranteed by Theorem~\ref{fact:torusrep}. The expression indicated in the second sentence of the statement is then the
$T$-character of $\C[{\mathfrak{X}}]\cong S/I(\mathfrak{X})$ in the sense of Definition~\ref{def:whatischaracter}.
Now apply Theorem~\ref{fact:chariso} (and the discussion immediately after Definition~\ref{def:whatischaracter}).
\end{proof}

\subsection{The general linear group and its Levi groups}\label{subsec:gln}

We now specialize  to the setting of this paper. All varieties considered are $T$-stable subvarieties of ${\sf Mat}_{m, n}\cong \C^{mn}$, with $T$ a subtorus of the natural torus $(\C^\star)^{mn}$. The groups $\mathbf{L}_{\mathbf{I}|\mathbf{J}}$ used are products of general linear groups, so we begin by describing their representation theory more explicitly.

The (complex) general linear group $GL_k$ is reductive, and the group $T_k\subseteq GL_k$ of invertible diagonal matrices is a maximal torus. The associated Weyl group is the symmetric group $\mathfrak{S}_k$, so by Theorem~\ref{fact:chariso} the space of polynomial $GL_k$-characters is isomorphic to the ring of symmetric polynomials $\Lambda[t_1,\dots, t_k] := \Z[t_1,\dots, t_k]^{\mathfrak{S}_k}$, with a $\Z$-module basis given by the characters of irreducible $GL_k$-representations.

\begin{theorem}[{\cite[Theorem 8.2.2]{Fulton}}]\label{fact:GLrep}
    The irreducible polynomial representations of $GL_k$ are the \emph{Weyl modules} $V_{\lambda}$ indexed by partitions $\lambda$ with at most $k$ rows. The character of $V_\lambda$ is the \emph{Schur polynomial} $s_\lambda(t_1,\dots, t_k)$ previously defined in (\ref{eqn:Schurdef}).
\end{theorem}

First, suppose that $V$ is a finite-dimensional $GL_k$-representation, realized through a homomorphism $\rho:GL_k\to GL(V)$. Since $GL_k$ is reductive, it follows that $V = \bigoplus_\lambda V_\lambda^{c_\lambda}$ for some nonnegative integers $c_\lambda$. We concretely compute the coefficients $c_{\lambda}$ by expressing the character $\chi_V(\mathbf{t}) = \textrm{Trace}(\rho(\mathbf{t}))$ as $\sum_\lambda c_\lambda s_\lambda(\mathbf{t})$ in the basis of Schur polynomials.

\begin{example}
\label{exa:2to3}
Let $V={\sf Sym}^2({\mathbb C}^2)$ be the ${\mathbb C}$-vector subspace of $\C[x,y]$ 
with basis $\{x^2, 2xy, y^2\}$. $GL_2({\mathbb C})$ acts on $V$ by
$x\mapsto (ax+cy), \ \ y\mapsto (bx+dy)$.
This induces a change of basis
\[x^2\mapsto (ax+cy)^2, \ 
2xy\mapsto 2(ax+cy)(bx+dy), \ 
y^2\mapsto (bx+dy)^2.\]
After identifying  $V\cong {\mathbb C}^3$, we can express the ${\mathbb C}[GL_2({\mathbb C})]$-module $V$
as a linear representation, by sending a generic $2\times 2$ matrix to the change of basis matrix:
\begin{align*}
\rho: & \ GL_2({\mathbb C})\to
GL_3({\mathbb C})\\
& 
\left[\begin{matrix}
a & b \\
c & d
\end{matrix}
\right] \mapsto
\left[
\begin{matrix}
a^2 & 2ab & b^2 \\
ac & bc+ad & bd \\
c^2 & 2cd & d^2
\end{matrix}
\right].
\end{align*}
Hence,
$\rho\left(
\left[\begin{smallmatrix}
t_1 & 0 \\
0 & t_2
\end{smallmatrix}
\right]
\right)=
\left[
\begin{smallmatrix}
t_1^2 & 0 & 0 \\
0 & t_1 t_2 & 0 \\
0 & 0 & t_2^2
\end{smallmatrix}
\right]
$
and, by \eqref{eqn:chardefsequiv},  $\chi_V = t_1^2+t_1t_2+t_2^2=s_{\stableau{\phantom{} & \phantom{}}}(t_1,t_2)$. Thus $V$ is isomorphic to the Weyl module $V_{\stableau{\phantom{} & \phantom{}}}$.
\end{example}

By Proposition~\ref{fact:prodrep}, the representation theory of $GL_k$ encapsulated in Theorem~\ref{fact:GLrep} extends immediately to characterize the irreducible representations of products $GL_k\times GL_l$. They are tensor products $V_\lambda\boxtimes V_\mu$ of Weyl modules, with characters given by the corresponding products $s_\lambda(t_1,\dots, t_k)s_\mu(t'_1\dots, t'_l)$ of Schur polynomials. The following example realizes this construction explicitly using the Kronecker product of matrices.

\begin{example}\label{exa:splittensor}
Consider the $GL_2\times GL_2$ action on ${\sf Sym}^2({\mathbb C}^2)\boxtimes
{\sf Sym}^2({\mathbb C}^2)
$. A basis consists of:
\[x^2\otimes x^2, x^2\otimes (2xy),x^2\otimes y^2, (2xy)\otimes x^2, (2xy)\otimes (2xy), (2xy)\otimes y^2, y^2\otimes x^2, y^2\otimes (2xy), y^2\otimes y^2.\]
Suppose a generic element of $GL_2\times GL_2$ is 
$(g,h)=\left(
\left[\begin{smallmatrix}
a & b \\
c & d
\end{smallmatrix}
\right],
\left[\begin{smallmatrix}
q & r \\
s & t
\end{smallmatrix}
\right]
\right)
$. Then, the action induces a change of basis, e.g.,
$x^2\otimes y^2\mapsto
(ax+cy)^2\otimes (rx+ty)^2$.
The change of basis matrix is 
\[\left[
\begin{smallmatrix}
a^2 & 2ab & b^2 \\
ac & bc+ad & bd \\
c^2 & 2cd & d^2
\end{smallmatrix}
\right]\otimes
\left[\begin{smallmatrix}
q^2 & 2qr & r^2 \\
qs & rs+qt & rt \\
s^2 & 2st & t^2
\end{smallmatrix}\right]:= 
\left[\begin{matrix}
a^2\left[\begin{smallmatrix}
q^2 & 2qr & r^2 \\
qs & rs+qt & rt \\
s^2 & 2st & t^2
\end{smallmatrix}\right]
&
 2ab
 \left[\begin{smallmatrix}
q^2 & 2qr & r^2 \\
qs & rs+qt & rt \\
s^2 & 2st & t^2
\end{smallmatrix}\right]

&

b^2
 \left[\begin{smallmatrix}
q^2 & 2qr & r^2 \\
qs & rs+qt & rt \\
s^2 & 2st & t^2
\end{smallmatrix}\right]\\

ac
 \left[\begin{smallmatrix}
q^2 & 2qr & r^2 \\
qs & rs+qt & rt \\
s^2 & 2st & t^2
\end{smallmatrix}\right]

&

(bc+ad) 
 \left[\begin{smallmatrix}
q^2 & 2qr & r^2 \\
qs & rs+qt & rt \\
s^2 & 2st & t^2
\end{smallmatrix}\right]

&

bd
 \left[\begin{smallmatrix}
q^2 & 2qr & r^2 \\
qs & rs+qt & rt \\
s^2 & 2st & t^2
\end{smallmatrix}\right]\\

c^2
 \left[\begin{smallmatrix}
q^2 & 2qr & r^2 \\
qs & rs+qt & rt \\
s^2 & 2st & t^2
\end{smallmatrix}\right]

&

2cd
 \left[\begin{smallmatrix}
q^2 & 2qr & r^2 \\
qs & rs+qt & rt \\
s^2 & 2st & t^2
\end{smallmatrix}\right]

&

d^2
 \left[\begin{smallmatrix}
q^2 & 2qr & r^2 \\
qs & rs+qt & rt \\
s^2 & 2st & t^2
\end{smallmatrix}\right]
\end{matrix}\right].\]
That is, the homomorphism describing this representation is
\[\theta
\left(
\left[\begin{matrix}
a & b \\
c & d
\end{matrix}
\right],
\left[\begin{matrix}
q & r \\
s & t
\end{matrix}
\right]
\right)=
\left[
\begin{matrix}
a^2 & 2ab & b^2 \\
ac & bc+ad & bd \\
c^2 & 2cd & d^2
\end{matrix}
\right]\otimes
\left[\begin{matrix}
q^2 & 2qr & r^2 \\
qs & rs+qt & rt \\
s^2 & 2st & t^2
\end{matrix}\right]
\]
The character of this representation is 
${\rm Trace}\left(\theta
\left(
\left[\begin{smallmatrix}
t_1 & 0 \\
0 & t_2
\end{smallmatrix}
\right],
\left[\begin{smallmatrix}
t'_1 & 0 \\
0 & t'_2
\end{smallmatrix}
\right]
\right)\right)
$, and the reader
can check this indeed
equals $\chi_{\rho}(t_1,t_2)\chi_{\rho}(t'_1,t'_2)$.
\end{example}

The $GL_k$ representations we consider are constructed from the following basic example. Consider $\C^k$ as an
affine space. Viewing an element $v\in \C^k$ as a column vector, we realize $\C^k$ as a $GL_k$-variety via the right action $v\cdot g = g\inv v$. If $\mathfrak{X}\subseteq \C^k$ is a subvariety stable under this action, the formula (\ref{eqn:Feb13abc}) then gives a left $GL_k$ action on the coordinate ring $\C[\mathfrak{X}] \cong S/I(\mathfrak{X})$, where $S = \C[z_1,\dots, z_k]$. The torus $T_k\subseteq GL_k$ of invertible diagonal matrices is precisely the natural torus acting on $\C^k$ from Definition~\ref{def:nattorus}, so the $T_k$-weight of each monomial ${\sf m}$ is obtained by making the substitutions $z_i\mapsto t_i$ for all $1\leq i\leq k$. By Proposition~\ref{thm:bigsummary}, we may compute the decomposition of $\C[\mathfrak{X}]$ into irreducible $GL_k$-representations by identifying the constants $c^{\mathfrak{X}}_\lambda$ such that 
\begin{equation}
\label{eqn:Mar7abc}
\sum_{{\sf m}\in{\sf Std}_<(S/I(\mathfrak{X}))}{\sf wt}_{T_k}({\sf m}) = \sum_\lambda c^{\mathfrak{X}}_\lambda s_\lambda(\mathbf{t}).
\end{equation}
Since $V={\mathbb C}[{\mathfrak X}]$ is generally infinite-dimensional, the expression
(\ref{eqn:Mar7abc}) technically lives in the completion $\widehat{\Lambda[t_1,\ldots,t_k]}\cong {\widehat {\sf Rep(GL_k)}}$ of the ring of symmetric
polynomials, which allows infinite linear combinations
of Schur polynomials (see, e.g., \cite[pg.~291]{ECII}). 
This causes no concern since $V=\bigoplus_{i\geq 0}V_i$ is a standard graded ring where 
each $V_i$ is a finite-dimensional $GL_k$-module. All copies of the Weyl module $V_\lambda$ in $V$ must lie in the graded component $V_{|\lambda|}$, so all coefficients in (\ref{eqn:Mar7abc}) are finite.

Letting $\mathbf{I} = \{0 = i_0 <\dots <i_r = k\}$, the subgroup $L_{\mathbf{I}} = GL_{i_1-i_0}\times\dots\times GL_{i_r-i_{r-1}}$ of $k\times k$ invertible block-diagonal matrices also acts on $\C^k$ by restriction.
If $\mathfrak{X}$ is stable under this restricted action, then $\C[\mathfrak{X}]$ decomposes into a sum of \emph{split-Weyl modules} $V_{\underline\lambda} := V_{\lambda^{(1)}}\boxtimes\dots\boxtimes V_{\lambda^{(r)}}$, 
where each $\lambda^{(a)}$ is a partition with at most $i_a-i_{a-1}$ rows. The character of $V_{\underline\lambda}$ is the split-Schur polynomial $s_{\underline \lambda}(\mathbf{t})$ from the introduction (cf.~\cite[Definition~4.3]{HY1}).
The group $T_k$ is also a maximal torus for $L_{\mathbf{I}}$, so by Proposition~\ref{thm:bigsummary} again we decompose $\C[\mathfrak{X}]$ into irreducible $L_{\mathbf{I}}$-representations by
identifying constants $c^{\mathfrak{X}}_{\underline\lambda}$ such that 
\begin{equation}
\label{eqn:Mar7cde}
\sum_{{\sf m}\in{\sf Std}_<(S/I(\mathfrak{X}))}{\sf wt}_{T_k}({\sf m}) = \sum_{\underline\lambda} c^{\mathfrak{X}}_{\underline\lambda} s_{\underline\lambda}(\mathbf{t}).
\end{equation}
As with (\ref{eqn:Mar7abc}), (\ref{eqn:Mar7cde}) lives in 
$\widehat{{\sf Rep}(L_{\bf I})}\cong \widehat{\Lambda[t_1,\ldots,t_{i_1}]}\boxtimes \cdots\boxtimes \widehat{\Lambda[t_{i_{r-1}+1},\ldots,t_{i_r}]}$.

Now, the $GL_m$-action on $\C^m$ and $GL_n$-action on $\C^n$ combine to give the $\mathbf{GL} := GL_m\times GL_n$-action $(v\otimes w)\cdot (p, q) = (v\cdot p)\otimes (w\cdot q)$ on $\C^m\boxtimes \C^n$.
Identifying $\C^m\boxtimes\C^n$ with the matrix space ${\sf Mat}_{m, n}$, this is the action (\ref{eqn:howmult}) seen in the introduction.
This action restricts to a $\mathbf{L}_{\mathbf{I}|\mathbf{J}} := L_{\mathbf{I}}\times L_{\mathbf{J}}$-action for any indexing sets $\mathbf{I} = \{0=i_0<\dots<i_r = m\}$ and $\mathbf{J} = \{0<j_0<\dots<j_s = n\}$.
If $\mathfrak{X}\subseteq{\sf Mat}_{m, n}$ is a subvariety stable under this $\mathbf{L}_{\mathbf{I}|\mathbf{J}}$ action, then applying (\ref{eqn:Feb13abc}) gives us the left $\mathbf{L}_{\mathbf{I}|\mathbf{J}}$-action (\ref{eqn:coordringaction}) on the coordinate ring $\C[\mathfrak{X}] \cong S/I(\mathfrak{X})$, where $S = \C[z_{11},\dots, z_{mn}]$.
The maximal torus $\mathbf{T} = T_m\times T_n$ in $\mathbf{L}_{\mathbf{I}|\mathbf{J}}$ is a subtorus of the natural torus acting on ${\sf Mat}_{m, n}\cong \C^{mn}$: the $\mathbf{T}$-weight of a monomial ${\sf m}\in S$ is obtained by making the substitutions $z_{ij}\mapsto x_iy_j$ for each variable.
Applying Proposition~\ref{thm:bigsummary} one more time, we see that the decomposition of $\C[\mathfrak{X}]$ into irreducible $\mathbf{L}_{\mathbf{I}|\mathbf{J}}$-representations is computed by identifying constants $c^{\mathfrak{X}}_{\underline\lambda|\underline\mu}$ such that
\[\sum_{{\sf m}\in{\sf Std}_<(S/I(\mathfrak{X}))}{\sf wt}_{\mathbf{T}}({\sf m}) = \sum_{\underline\lambda|\underline\mu} c^{\mathfrak{X}}_{\underline\lambda|\underline\mu} s_{\underline\lambda}(\mathbf{x})s_{\underline\mu}(\mathbf{y})\in {\widehat{{\sf Rep}(\mathbf{L}_{\mathbf{I}|\mathbf{J}})}},\]
where ${\sf wt}_{\mathbf{T}}$ is the same thing as ${\sf wt}$ from (\ref{eqn:whatiswt}) and
\[{\widehat{{\sf Rep}(\mathbf{L}_{\mathbf{I}|\mathbf{J}})}}\!\cong\!
\left(\widehat{\Lambda[x_1,\Compactldots,x_{i_1}]}\mbox{\small\ensuremath{\boxtimes}} \Compactcdots\mbox{\small\ensuremath{\boxtimes}} \widehat{\Lambda[x_{i_{r-1}+1},\Compactldots,x_{i_r}]}\right)
\mbox{\small\ensuremath{\boxtimes}}
\left(\widehat{\Lambda[y_1,\Compactldots,y_{i_1}]}\mbox{\small\ensuremath{\boxtimes}} \Compactcdots\mbox{\small\ensuremath{\boxtimes}} \widehat{\Lambda[y_{i_{s-1}+1},\Compactldots,y_{i_s}]}\right).\]

\section{Tableaux}\label{sec:tableau}
We review some standard notions from tableau combinatorics, with \cite{Fulton} as our primary reference. Let $\lambda$ be an integer partition, which we identify with its Young diagram.

\begin{definition}
A \emph{semistandard Young tableau} $T$ of \emph{shape} $\lambda = {\sf shape}(T)$ is a filling of the cells of $\lambda$ with positive integer entries such that the entries increase weakly along rows from left to right, and the entries increase strictly along columns from top to bottom.
\end{definition}

\begin{example}
The partition $\lambda=(4,2,2)$ has Young diagram $\tableau{\ & \ & \ & \ \\
\ & \ \\ \ & \ }$. One semistandard tableau of shape $\lambda$ is $T = \tableau{1 & 1 & 2 & 3 \\
3 & 3 \\ 4 & 5 }$.
\end{example}

Let ${\sf SSYT}(\lambda,[a,b])$ be the set of semistandard Young tableaux of shape $\lambda$ where the entries come from the interval $[a,b]$. Let 
\[{\sf SSYT}(\lambda):=
{\sf SSYT}(\lambda,[1,\infty))\text{ \ \ and  \ \ }
{\sf SSYT}:=\bigcup_{\lambda} {\sf SSYT}(\lambda).\]

\begin{definition}\label{def:readingword}
    The (column) \textit{reading word} of a tableau $T$, ${\sf word}(T)$, reads the entries of each column bottom-to-top, starting from the leftmost column and proceeding rightward.
\end{definition}

\begin{example}\label{exa:Nov16abc}
If $T=\tableau{1 & 2 & 2 & 3 \\ 2 & 4 \\ 
3 & 5}$ then ${\sf word}(T) = 32154223$.
\end{example}

If we know that $w = {\sf word}(T)$ for some tableau $T$, we can reconstruct $T$ from $w$: the strictly decreasing sequences in $w$ are exactly the columns of $T$, read from bottom to top. However, many words cannot be realized as ${\sf word}(T)$ for any $T$. We now recall the notion of \emph{Knuth equivalence}, which relates any word to the reading word of a unique tableau. 

\begin{definition}\label{def:knuthequiv}
    Let $x, y, z\in\mathbb{N}$ be letters. The \textit{elementary Knuth moves} are 
    \begin{equation}\label{eqn:leftKnuthmove}
    yzx\leftrightarrow yxz\ \ (x<y\leq z)
    \end{equation}
    and 
    \begin{equation}\label{eqn:rightKnuthmove}
    zxy\leftrightarrow xzy\ \ (x\leq y < z).
    \end{equation}
    Two words $w$ and $w'$ are \textit{Knuth equivalent}, denoted $w\sim_K w'$, if they are connected by a sequence of elementary Knuth moves.
\end{definition}

\begin{theorem}\label{thm:jdt}
    For every word $w$, there is a unique $T\in {\sf SSYT}$ such that $w\sim_K {\sf word}(T)$.
\end{theorem}

There are various algorithms for constructing the tableau $T$ such that $w\sim_K{\sf word}(T)$. We employ the following \emph{row insertion algorithm}.

\begin{definition}
    The \emph{row insertion} of $x$ into $T\in {\sf SSYT}$ is a tableau denoted $T\leftarrow x$. If no element of the first row of $T$ is strictly greater than $x$, then $x$ is appended to the end of that row. Otherwise, let $y$ be the first element in the first row of $T$ such that $x < y$. Replace $y$ with $x$ and insert $y$ into the second row of $T$ using the same procedure. This process eventually terminates, producing $T\leftarrow x$.
\end{definition}

\begin{example}
    Let $T = \tableau{1 & 2 & 2 & 3\\ 2 & 5 \\ 4}$ and let $x = 1$. Inserting $1$ into the first row of $T$ bumps out a $2$, yielding $T_1 = \tableau{1 & 1 & 2 & 3\\ 2 & 5 \\ 4}$. Reinserting the displaced $2$ into the second row bumps out the $5$ to give $T_2 = \tableau{1 & 1 & 2 & 3\\ 2 & 2 \\ 4}$. Reinserting this $5$ in the third row gives $(T\leftarrow x) =  \tableau{1 & 1 & 2 & 3\\ 2 & 2 \\ 4 & 5}$.
\end{example}

\begin{definition}\label{def:insertiontab}
    The \emph{insertion tableau} of a word $w = w_1w_2\dots w_N$ is the tableau 
    \[{\sf tab}(w) := (((\emptyset\leftarrow w_1)\leftarrow w_2)\leftarrow\dots\leftarrow w_N).\]
\end{definition}

\begin{theorem}\label{thm:insertionnice}
    Any word $w$ satisfies $w\sim_K{\sf word}({\sf tab}(w))$.
\end{theorem}

\begin{corollary}\label{cor:Knuthifftab}
$w\sim_K w'$ if and only if ${\sf tab}(w)={\sf tab}(w')$
\end{corollary}
\begin{proof}
This is immediate from combining Theorems~\ref{thm:jdt} and~\ref{thm:insertionnice}.
\end{proof}
\begin{example}
    Theorem~\ref{thm:insertionnice} makes sense when the elementary Knuth transformations are interpreted via row insertion. Let $w = yzx$ and $w' = yxz$ with $x<y\leq z$, so $w\sim_K w'$ via one use of (\ref{eqn:leftKnuthmove}). Direct computation shows that ${\sf tab}(w) = \tableau{x & z\\ y} = {\sf tab}(w')$. Similarly, if $v = zxy$ and $v' = xzy$ with $x\leq y<z$, then ${\sf tab}(v) = \tableau{x & y\\ z} = {\sf tab}(v')$.
\end{example}

Definition~\ref{def:matrixwords} identifies a $M\in {\sf Mat}_{m,n}({\mathbb Z}_{\geq 0})$ with the pair $({\sf row}(M)|{\sf col}(M))$. Combining this identification with the insertion algorithm yields:

\begin{definition}
    The \emph{RSK map} sends  $M\in {\sf Mat}_{m,n}({\mathbb Z}_{\geq 0})$ to a pair of semistandard Young tableaux as follows:
    ${\sf RSK}(M) = ({\sf tab}({\sf row}(M))|{\sf tab}({\sf col}(M)))$.
\end{definition}

\begin{theorem}[RSK Correspondence]\label{thm:RSK}
    The map ${\sf RSK}$ defines a bijection between ${\sf Mat}_{m,n}({\mathbb Z}_{\geq 0})$  and tableau-pairs $(P|Q)$ of the same shape $\lambda$.
\end{theorem}

\begin{remark}\label{remark:Feb15symm}
    In standard references such as \cite{Fulton}, the map ${\sf RSK}$ is given by setting $P = {\sf tab}({\sf col}(M))$ and defining $Q$ to be a \emph{recording tableau} that keeps track of information needed to reverse the insertion algorithm. The difficult part of the proof is showing that with these conventions we in fact have 
    $Q = {\sf tab}({\sf row}(M))$; this result is sometimes called the ``symmetry theorem" \cite[Section 4.1]{Fulton}. Instead, we prefer to see Theorem~\ref{thm:RSK} as a consequence of more general results about crystal graphs; see Example~\ref{ex:RSK}.
\end{remark}

\begin{definition}
    Let $w$ be a word with ${\bf I}$-filtration ${\sf filter}_{\mathbf I}(w) = (w^{(1)},\dots w^{(r)})$. Then the \emph{${\bf I}$-filtered insertion tableau-tuple} of $w$ is
    \[{\sf tab}_{\mathbf{I}}(w) = ({\sf tab}(w^{(1)}),{\sf tab}(w^{(2)}),\dots, {\sf tab}(w^{(r)})).\]
\end{definition}

In view of Definition~\ref{def:insertiontab}, we obtain an algorithmic form of {\sf filterRSK}:

\begin{proposition}\label{prop:filterrskalg}
Given ${\bf I}$ and ${\bf J}$ as in \eqref{eqn:twointseq} and $M\in {\sf Mat}_{m,n}({\mathbb Z}_{\geq 0})$
\[{\sf filterRSK}_{{\bf I}|{\bf J}}(M)=({\sf tab}_{\bf I}({\sf row}(M)), {\sf tab}_{\bf J}({\sf col}(M))).\]
\end{proposition}
\begin{proof}
Immediate from Main Definition~\ref{alg:filteredrsk} and Theorem~\ref{thm:insertionnice}. 
\end{proof}

\section{Crystals and bicrystals}\label{subsec:graphs}
To prove Theorem~\ref{thm:main}, we need a tool from Kashiwara's theory of \emph{crystal bases}, that is, special types of graphs called \emph{crystal graphs}. We only need specific examples of crystal graphs and do not present the generalities. We refer to \cite{Shimozono, Kwon, BS}. We begin by laying out just enough general 
notions for our needs.

\begin{definition}
    A \emph{pre-crystal graph} $\mathcal{G}$ is a simple, directed graph with countably many labelled vertices and edges. The label of each vertex $v$ has an associated \emph{weight}, an ordered tuple of nonnegative integers denoted ${\sf cwt}_{\mathcal G}(v)\in {\mathbb Z}_{\geq 0}^r$, for some fixed $r\in {\mathbb Z}_{\geq 0}$. 
\end{definition}

\begin{definition}\label{def:crystalmap}
    A \emph{pre-crystal graph homomorphism} 
    $f:\mathcal{G}\to\mathcal{H}$
    is a weight-preserving set map between the vertex sets (so ${\sf cwt}_{\mathcal G}(v) = {\sf cwt}_{\mathcal H}(f(v))$ for all vertices $v\in \mathcal{G}$) that also preserves adjacency and edge labels. In other words, if $v\xrightarrow{i}v'$ in $\mathcal{G}$ then we must have $f(v)\xrightarrow{i}f(v')$ in $\mathcal{H}$. Furthermore, $f$ is a \emph{pre-crystal graph isomorphism} if it is invertible and $f\inv:\mathcal{H}\to\mathcal{G}$ is also a pre-crystal graph homomorphism.
\end{definition}

\begin{definition}
    The \emph{direct sum} $\mathcal{G}\oplus\mathcal{H}$ of two pre-crystal graphs is their disjoint union.
\end{definition}

\begin{definition}\label{def:cartprodcrystal}
    The \emph{Cartesian product} $\mathcal{G} \ \Box \ \mathcal{H}$ of two pre-crystal graphs has vertex set $\{(g, h) \ | \ g\in \mathcal{G}, h\in \mathcal{H}\}$, with ${\sf cwt}_{\mathcal{G}\Box \mathcal{H}}((g, h)) = ({\sf cwt}_{\mathcal G}(g), {\sf cwt}_{\mathcal H}(h))$. The edge-labels of $\mathcal{G} \ \Box \ \mathcal{H}$ come from the disjoint union of the edge-labelling sets of $\mathcal{G}$ and $\mathcal{H}$. There is an edge $(g, h)\xrightarrow{i^\mathcal{G}}(g', h')$ if $g\xrightarrow{i}g'$ and $h = h'$, and there is an edge $(g, h)\xrightarrow{j^\mathcal{H}}(g', h')$ if $g = g'$ and $h\xrightarrow{j}h'$.
\end{definition}

\begin{definition}
A pre-crystal graph ${\mathcal G}$ is \emph{connected} if there exists an undirected path between any two
vertices of ${\mathcal G}$ (i.e., its underlying graph is connected).
\end{definition}

Most pre-crystal graph homomorphisms we discuss will be of the following type:

\begin{definition}
    A pre-crystal graph homomorphism $f:\mathcal{G}\to\mathcal{H}$ is a \emph{local isomorphism} if its restriction $f|_{\mathcal C}$ to any connected component ${\mathcal C}$ of $\mathcal{G}$ is an isomorphism onto a connected component $\mathcal{C}'$ of $\mathcal{H}$. 
\end{definition}

\begin{lemma}\label{lemma:localisoprops}
\gap
    \begin{itemize}
        \item[(I)] If $f:\mathcal{G}\to\mathcal{H}$ is a local isomorphism then $\mathcal{G}\cong \bigoplus_{\mathcal{C}}\mathcal{C}^{\oplus m_\mathcal{C}}$, where the sum is over connected components of $\mathcal{H}$ and $m_\mathcal{C}$ is the number of connected components of $\mathcal{G}$ mapped onto $\mathcal{C}$ by~$f$.
        \item[(II)] A composition of local isomorphisms is a local isomorphism.
        \item[(III)] If $f:\mathcal{G}\to\mathcal{H}$ and $f':\mathcal{G'}\to\mathcal{H'}$ are local isomorphisms, then the product map $(f\Box f'):\mathcal{G}\Box\mathcal{H}\to\mathcal{G'}\Box\mathcal{H'}$ defined on vertices by $(v, v')\mapsto (f(v), f'(v'))$ is a local isomorphism.
    \end{itemize}
\end{lemma}
\begin{proof}
(I) and (II) are immediate. (III) follows since connected components of
$\mathcal{G}\Box\mathcal{H}$ are of the form ${\mathcal G}'\Box {\mathcal H}'$ where ${\mathcal G}'$
and ${\mathcal H}'$ are connected components of ${\mathcal G}$ and ${\mathcal H}$, respectively.
\end{proof}

\subsection{Crystals of words and tableaux}\label{subsec:crystals}
Next, we define crystal graphs with vertices labelled by words. Let $w = w_1w_2\dots w_N$ be a word on the alphabet $[n]$ and fix $i\in [n-1]$.

\begin{definition}
    The \textit{ith bracket operator} ${\sf bracket}_i$ associates a word on the alphabet $\{(, )\}$ to $w$ by recording a ``$)$" for each $i$ and a ``$($" for each $i+1$ (maintaining the order of the letters).
\end{definition}
    
\begin{definition}
    Let $w_e = i+1$ and $w_f = i$ be the letters of $w$ associated to the leftmost unmatched ``$($" and rightmost unmatched ``$)$" of ${\sf bracket}_i(w)$ respectively. The \textit{crystal raising operator} $e_i$ sends $w$ to the word obtained by changing $w_e$ to $i$. Analogously, the \textit{crystal lowering operator} $f_i$ sends $w$ to the word obtained by changing $w_f$ to $i+1$. If no such letters $w_e$ or $w_f$ exist, the operators output the special symbol $\varnothing$.
\end{definition}

\begin{definition}
    The \emph{word crystal graph} $\mathcal{W}_{[a, b]}$ has vertices labelled by words $w$ on the alphabet $[a, b]$. The weight ${\sf cwt}_{\mathcal{W}_{[a, b]}}(w)$ of a vertex
    $w$ is $(y_a,y_{a+1},\ldots,y_{b})$ where $y_j$ is the number of $j$'s appearing in $w$. There is an edge $w\stackrel{i}{\longrightarrow}v$ if and only if $v = f_i(w)$. In the case where $[a, b] = [1, n]$ we use the abbreviated notation $\mathcal{W}_n:= \mathcal{W}_{[1, n]}$.
\end{definition}

\begin{example}\label{ex:wordcrystal}
Figure~\ref{fig:W3} shows
    the connected components of $\mathcal{W}_3$ containing $211$ and $121$. Each vertex is labelled by a word $w$, with ${\sf bracket}_1(w)$ written above it and ${\sf bracket}_2(w)$ below. We highlight  the rightmost unmatched ``)" in ${\sf bracket}_1(w)$ using red (if it exists) and the rightmost unmatched ``)" in ${\sf bracket}_2(w)$ in blue. If $f_1(w) = v$, we draw a red edge directed from $w$ to $v$, and when $f_2(w) = v$ we draw a blue edge directed from $w$ to $v$. For clarity, we label the edges of the crystal graphs with ``$f_i$'' rather than ``$i$''.
    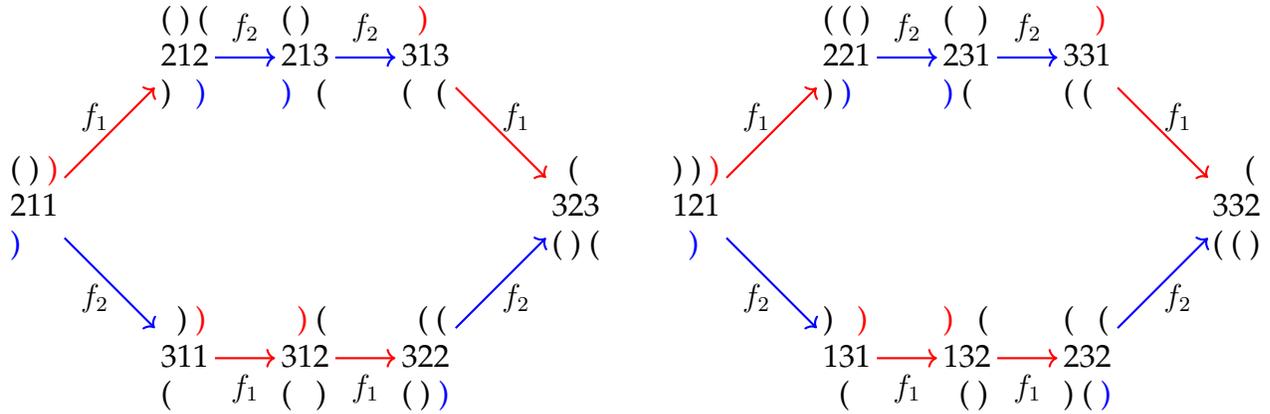
\begin{figure}
    \begin{center}
    \begin{tikzpicture}[scale = 0.39, align = left]
        \draw[red, thick][->] (1, 6) -- (4, 9);
        \draw[red, thick][->] (14, 9) -- (17, 6);
        \draw[blue, thick][->] (6, 10) -- (8, 10);
        \draw[blue, thick][->] (10, 10) -- (12, 10);
        \draw[blue, thick][->] (1, 4) -- (4, 1);
        \draw[blue, thick][->] (14, 1) -- (17, 4);
        \draw[red, thick][->] (6, 0) -- (8, 0);
        \draw[red, thick][->] (10, 0) -- (12, 0);
        \draw (0, 5) node{( ) \textcolor{red}{)} \\ 211 \\ \textcolor{blue}{)}  };
        \draw (5, 10) node{( ) ( \\ 212 \\ ) \ \ \textcolor{blue}{)}};
        \draw (9, 10) node{( ) \\ 213 \\ \textcolor{blue}{)} \ \ (};
        \draw (13, 10) node{ \ \ \textcolor{red}{)} \\ 313 \\ ( \ \ (};
        \draw (5, 0) node{\ \ ) \textcolor{red}{)} \\ 311 \\ ( };
        \draw (9, 0) node{\ \ \textcolor{red}{)} ( \\ 312 \\ ( \ \ ) };
        \draw (13, 0) node{\ \ ( ( \\ 322 \\ ( ) \textcolor{blue}{)}};
        \draw (18, 5) node{\ \ ( \ \\ 323 \\ ( ) (};
        \draw (2,8) node{$f_1$};
        \draw (2,2) node{$f_2$};
        \draw (7,11) node{$f_2$};
        \draw (11,11) node{$f_2$};
        \draw (16,8) node{$f_1$};
        \draw (7,-1) node{$f_1$};
        \draw (11,-1) node{$f_1$};
        \draw (16,2) node{$f_2$};

        \draw[red, thick][->] (23, 6) -- (26, 9);
        \draw[red, thick][->] (36, 9) -- (39, 6);
        \draw[blue, thick][->] (28, 10) -- (30, 10);
        \draw[blue, thick][->] (32, 10) -- (34, 10);
        \draw[blue, thick][->] (23, 4) -- (26, 1);
        \draw[blue, thick][->] (36, 1) -- (39, 4);
        \draw[red, thick][->] (28, 0) -- (30, 0);
        \draw[red, thick][->] (32, 0) -- (34, 0);
        \draw (22, 5) node{) ) \textcolor{red}{)} \\ 121 \\  \ \ \textcolor{blue}{)}  };
        \draw (27, 10) node{( ( ) \\ 221 \\ )  \textcolor{blue}{)}};
        \draw (31, 10) node{( \ \ ) \\ 231 \\ \textcolor{blue}{)} (};
        \draw (35, 10) node{ \ \ \ \ \textcolor{red}{)} \\ 331 \\ ( (};
        \draw (27, 0) node{) \ \ \textcolor{red}{)} \\ 131 \\ \ \ ( };
        \draw (31, 0) node{\textcolor{red}{)} \ \ ( \\ 132 \\ \ \ (  ) };
        \draw (35, 0) node{( \ \ ( \\ 232 \\ ) ( \textcolor{blue}{)}};
        \draw (40, 5) node{\ \ \ \ ( \ \\ 332 \\ ( ( )};
        \draw (24,8) node{$f_1$};
        \draw (24,2) node{$f_2$};
        \draw (29,11) node{$f_2$};
        \draw (33,11) node{$f_2$};
        \draw (38,8) node{$f_1$};
        \draw (29,-1) node{$f_1$};
        \draw (33,-1) node{$f_1$};
        \draw (38,2) node{$f_2$};
    \end{tikzpicture}
    \end{center}
    \caption{ \label{fig:W3} The connected components of $\mathcal{W}_3$ containing $211$ and $121$}
    \end{figure}
\end{example}

This lemma is well-known; we include a proof for convenience.
\begin{lemma}\label{lemma:crystalinverse}
    The crystal operators $e_i$ and $f_i$ are inverses whenever their outputs are not $\varnothing$.
\end{lemma}
\begin{proof}
    Removing all matched parentheses from ${\sf bracket}_i(w)$ produces a string of the form $)\dots )(\dots ($. This shows that the rightmost unmatched ``$)$" in ${\sf bracket}_i(w)$ is in the same position as the leftmost unmatched ``$($" in ${\sf bracket}_i(f_i(w))$ whenever $f_i(w)\neq \varnothing$. Similarly, the leftmost unmatched ``$($" in ${\sf bracket}_i(w)$ is in the same position as the rightmost unmatched ``$)$" in ${\sf bracket}_i(e_i(w))$ whenever $e_i(w)\neq\varnothing$.
\end{proof}

\begin{definition}\label{def:tabcrystalinduced}
    The \emph{tableau crystal graph} $\mathcal{B}_{\lambda, [a, b]}$ is the induced subgraph of $\mathcal{W}_n$ on the vertices $\{{\sf word}(T) : T\in{\sf SSYT}(\lambda, [a, b])\}$. We often use the abbreviation $\mathcal{B}_{\lambda, n} := \mathcal{B}_{\lambda, [1, n]}$.
\end{definition}

We will typically refer to vertices of $\mathcal{B}_{\lambda, [a, b]}$ by tableaux rather than their
reading words.

\begin{definition}
    A word $w\in\mathcal{W}_{[a, b]}$ has \textit{highest weight} if $e_i(w) = \varnothing$ for all $i$. Equivalently, $w$ has highest weight if the vertex it labels in $\mathcal{W}_{[a, b]}$ is a source.
\end{definition}

\begin{theorem}[{\cite[Proposition~2.44]{Shimozono}}]\label{thm:crystalgraph}
    The graph $\mathcal{B}_{\lambda, [a,b]}$ is the connected component of $T_{\lambda,[a,b]}$ (Definition~\ref{def:highestweight}) in $\mathcal{W}_{[a, b]}$. This vertex is the unique source in $\mathcal{B}_{\lambda, [a, b]}$.
\end{theorem}

\begin{example}\label{ex:tableaucrystal}
    Figure~\ref{fig:B21-3} displays $\mathcal{B}_{\lambda, 3}$ for $\lambda = (2, 1)$, which is the connected component of $\mathcal{W}_3$ containing $211$ from Example~\ref{ex:wordcrystal} by Theorem~\ref{thm:crystalgraph}.

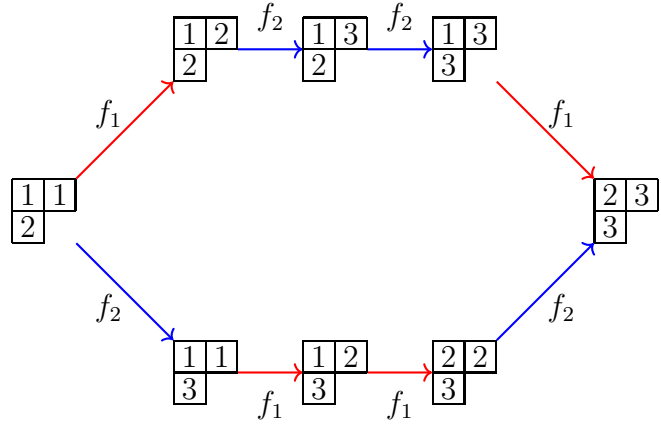
\begin{figure}
    \begin{center}
    \begin{tikzpicture}[scale = 0.43]
        \draw[red, thick][->] (1, 6) -- (4, 9);
        \draw[red, thick][->] (14, 9) -- (17, 6);
        \draw[blue, thick][->] (6, 10) -- (8, 10);
        \draw[blue, thick][->] (10, 10) -- (12, 10);
        \draw[blue, thick][->] (1, 4) -- (4, 1);
        \draw[blue, thick][->] (14, 1) -- (17, 4);
        \draw[red, thick][->] (6, 0) -- (8, 0);
        \draw[red, thick][->] (10, 0) -- (12, 0);
        \draw (0, 5) node{\tableau{1 & 1\\ 2}};
        \draw (5, 10) node{\tableau{1 & 2\\ 2}};
        \draw (9, 10) node{\tableau{1 & 3\\ 2}};
        \draw (13, 10) node{\tableau{1 & 3\\ 3}};
        \draw (5, 0) node{\tableau{1 & 1\\ 3}};
        \draw (9, 0) node{\tableau{1 & 2\\ 3}};
        \draw (13, 0) node{\tableau{2 & 2\\ 3}};
        \draw (18, 5) node{\tableau{2 & 3\\ 3}};
        \draw (2,8) node{$f_1$};
        \draw (2,2) node{$f_2$};
        \draw (7,11) node{$f_2$};
        \draw (11,11) node{$f_2$};
        \draw (16,8) node{$f_1$};
        \draw (7,-1) node{$f_1$};
        \draw (11,-1) node{$f_1$};
        \draw (16,2) node{$f_2$};
    \end{tikzpicture}
    \end{center}
      \caption{\label{fig:B21-3} $B_{\lambda,3}$ for $\lambda=(2,1)$}
    \end{figure}
\end{example}

Globally, ${\mathcal W}_{[a,b]}$ consists of copies of $\mathcal{B}_{\lambda, [a, b]}$ for various $\lambda$. The precise statement is Proposition~\ref{thm:crystalmain}, which we prove using two lemmas.

\begin{lemma}[``coplactic" property of $f_i$ and $e_i$, {\cite[Theorem 8.4]{BS}}]\label{lemma:coplactic}
    Fix two words $w$ and $w'$ on $[a, b]$ such that $w\sim_K w'$ and fix $i$. Then $f_i(w)\sim_K f_i(w')$ and $e_i(w)\sim_K e_i(w')$ (by convention, $\varnothing\sim_K\varnothing$). If $f_i(w),f_i(w')\neq \varnothing$ (resp. $e_i(w),e_i(w')\neq \varnothing$), then the converse holds. 
\end{lemma}

The following is also well-known. One reference is \cite[Section~4]{Kwon}:
\begin{lemma}\label{lemma:crystalknuth}
    If $w\sim_K w'$, then $w$ and $w'$ do not lie in the same connected component of $\mathcal{W}_{[a, b]}$, unless $w=w'$.
\end{lemma}

\begin{proposition}\label{thm:crystalmain}
    The map
    \[{\sf tab}: {\mathcal W}_{[a, b]}\to \bigoplus_{\lambda} 
    {\mathcal B}_{\lambda, [a, b]},\]
    defined on vertices by $w\mapsto {\sf tab}(w)$, is a local isomorphism. 
\end{proposition}
\begin{proof}
\noindent
(\emph{${\sf tab}$ is a pre-crystal graph homomorphism}):
By definition, ${\sf tab}$ is weight-preserving. It remains to check adjacency is preserved under ${\sf tab}$.
Suppose $f_i(w) = w^\downarrow$, so $w\stackrel{i}{\longrightarrow} w^\downarrow$ is an edge in $W_{[a,b]}$.  
By Theorem~\ref{thm:insertionnice}, $w\sim_K {\sf word}({\sf tab}(w))$ and
$w^\downarrow\sim_K {\sf word}({\sf tab}(w^\downarrow))$. 
Hence by
Lemma~\ref{lemma:coplactic} and Theorem~\ref{thm:crystalgraph} combined,
\begin{equation}
\label{eqn:Mar8uuu}
f_i({\sf word}({\sf tab}(w)))\sim_K w^{\downarrow} \sim_K {\sf word}({\sf tab}(w^\downarrow)).
\end{equation}
By Definition~\ref{def:tabcrystalinduced}, (\ref{eqn:Mar8uuu}) means 
$f_i({\sf tab}(w)) = {\sf tab}(w^\downarrow)$, as desired.

\noindent
\emph{(${\sf tab}$ is a local isomorphism)}:
Let ${\mathcal G}$ be a connected component of $\mathcal{W}_{[a, b]}$.
We claim the restricted map ${\sf tab}|_{\mathcal G}$ is an isomorphism onto some $\mathcal{B}_{\lambda, [a,b]}$.
To show injectivity, suppose $w,w'\in {\mathcal G}$ satisfy ${\sf tab}(w)={\sf tab}(w')$.
Then $w\sim_K w'$ by Corollary~\ref{cor:Knuthifftab}, so $w = w'$ by Lemma~\ref{lemma:crystalknuth}.
To show surjectivity, suppose that for some $w\in\mathcal{G}$ and $T\in\mathcal{B}_{\lambda, [a,b]}$ there is an edge ${\sf tab}(w)\stackrel{i}{\longrightarrow} T$ in $\mathcal{B}_{\lambda, [a,b]}$.
Then ${\sf word}(T) = f_i({\sf word}({\sf tab}(w)))$, and since $w\sim_K{\sf word}({\sf tab}(w))$ by Theorem~\ref{thm:insertionnice}, Lemma~\ref{lemma:coplactic} immediately shows that $T = {\sf tab}(f_i(w))$.
Thus ${\sf tab}|_{\mathcal G}$ is surjective, and this argument also shows that the inverse map ${\sf tab}|_{\mathcal G}\inv$ on the vertices is a pre-crystal graph homomorphism.
Thus ${\sf tab}|_{\mathcal G}$ is a pre-crystal graph isomorphism onto
${\mathcal B}_{\lambda,[a,b]}$, so ${\sf tab}$ is a local isomorphism.
\end{proof}

\begin{corollary}\label{cor:uniquesourceW}
    Any connected component $\mathcal{G}$ of $\mathcal{W}_{[a, b]}$ has a unique source vertex $v$ (i.e., a vertex of highest-weight). Moreover, ${\sf tab}(v) = T_{\lambda, [a, b]}$ for some partition $\lambda$, and then $\mathcal{G}\cong \mathcal{B}_{\lambda, [a, b]}$.
\end{corollary}
\begin{proof}
By Proposition~\ref{thm:crystalmain}, ${\sf tab}$ restricts to a pre-crystal graph isomorphism ${\mathcal G}\cong {\mathcal B}_{\lambda,[a,b]}$ for some $\lambda$.
Now apply Theorem~\ref{thm:crystalgraph}.
\end{proof}

One can see that Figure~\ref{fig:W3} agrees with Proposition~\ref{thm:crystalmain}, Lemma~\ref{lemma:coplactic}, and Corollary~\ref{cor:uniquesourceW}.

\begin{definition}\label{def:Knuthcrystalgraph}
    The \emph{Knuth crystal graph} $\mathcal{K}_{[a, b]}$ has vertices labelled by Knuth equivalence classes of words on $[a, b]$, with the weight of a class defined via representatives 
    (${\sf cwt}_{{\mathcal K}_{[a,b]}}([w]_K) := {\sf cwt}_{{\mathcal W}_{[a,b]}}(w)$). There is an edge $C\stackrel{i}{\longrightarrow}C'$ if for some choice of representatives $C = [w]_K$ and $C' = [w']_K$, we have $w' = f_i(w)$. In other words, $\mathcal{K}_{[a, b]}$ is the quotient graph $\mathcal{W}_{[a, b]}/\sim_K$.
\end{definition}

Since $w\sim_K w'$ implies ${\sf cwt}_{{\mathcal W}_{[a,b]}}(w)={\sf cwt}_{{\mathcal W}_{[a,b]}}(w')$, 
${\sf cwt}_{{\mathcal K}_{[a,b]}}([w]_K)$ is well-defined.

\begin{proposition}\label{prop:pancake}
The map \[[-]_K:{\mathcal W}_{[a,b]}\to {\mathcal K}_{[a,b]}\] that sends $w\mapsto [w]_K$ 
is a local isomorphism. Thus, every connected component ${\overline {\mathcal G}}$ of ${\mathcal K}_{[a,b]}$
is isomorphic to ${\mathcal B}_{\lambda,[a,b]}$ for some $\lambda$.
\end{proposition}
\begin{proof}
From the definition of $\mathcal{K}_{[a, b]}$, $[-]_K$ is a pre-crystal graph homomorphism.
Fix a connected component ${\mathcal G}$ of ${\mathcal W}_{[a,b]}$.
By Corollary~\ref{cor:uniquesourceW}, ${\mathcal G}$ has a unique
source vertex $v$ and ${\mathcal G}\cong {\mathcal B}_{\lambda,[a,b]}$ where $\lambda$ is the shape of
${\sf tab}(v)$.
Lemma~\ref{lemma:crystalknuth} shows $[-]_K|_{\mathcal G}$ is injective.
Now, suppose $[u]_K\stackrel{i}{\longrightarrow}[w]_K$ is an edge in ${\mathcal K}_{[a,b]}$ for some $u\in\mathcal{G}$.
By definition, there exists $u'\in [u]_K, w'\in [w]_K$ such that $u'\stackrel{i}{\longrightarrow} w'$.
By Lemma~\ref{lemma:coplactic}, $f_i(u)\sim_K f_i(u')=w'$.
Hence $f_i(u)\in {\mathcal G}$ and $[f_i(u)]_K=[w']_K$.
This proves that $[-]_K|_{\mathcal G}$ is surjective onto the connected component ${\overline {\mathcal G}}$ of ${\mathcal K}_{[a,b]}$ containing $[v]_K$, and also shows that the inverse map $([-]_K|_{\mathcal G})\inv$ is a pre-crystal graph homomorphism. Thus $[-]_K$ is a local isomorphism as desired. 

For the second sentence, given $[w]_K\in {\overline {\mathcal G}}$, let ${\mathcal G}_w$ be the connected component of ${\mathcal W}_{[a,b]}$ containing $w$.
Then ${\mathcal G}_w\cong {\mathcal B}_{\lambda,[a,b]}$ for some $\lambda$ by Corollary~\ref{cor:uniquesourceW}, and the first sentence of the lemma says ${\mathcal G}_w\cong {\overline {\mathcal G}}$.
\end{proof}

\begin{corollary}\label{cor:tablocaliso}
    The local isomorphism ${\sf tab}:\mathcal{W}_{[a, b]}\to \bigoplus_\lambda \mathcal{B}_{\lambda, [a, b]}$ descends to a pre-crystal graph isomorphism 
    \[{\overline {\sf tab}}:\mathcal{K}_{[a, b]}\to \bigoplus_\lambda \mathcal{B}_{\lambda, [a, b]}.\]
\end{corollary}
\begin{proof}
Proposition~\ref{prop:pancake} and Proposition~\ref{thm:crystalmain} show that
${\overline {\sf tab}}$ is a local isomorphism. It is clearly surjective. Finally, Corollary~\ref{cor:Knuthifftab} shows that different connected components of $\mathcal{K}_{[a, b]}$
map to distinct $\mathcal{B}_{\lambda, [a, b]}$, so ${\overline {\sf tab}}$ is a pre-crystal graph isomorphism.
\end{proof}

\subsection{Bicrystals of matrices}\label{subsec:matrixcrystals}

A matrix $M\in {\sf Mat}_{m,n}({\mathbb Z}_{\geq 0})$
is uniquely determined by ${\sf row}(M)$ and ${\sf col}(M)$ (Definition~\ref{def:matrixwords}). In isolation, we understand these two
``halves'' of the information encoding $M$ by Proposition~\ref{thm:crystalmain}. The connected components of $\mathcal{W}_m$ and $\mathcal{W}_n$ containing  ${\sf row}(M)$ and ${\sf col}(M)$, respectively,
are isomorphic to certain tableaux crystal graphs $\mathcal{B}_{\lambda, m}$ and $\mathcal{B}_{\mu, n}$ (in fact, $\lambda=\mu$ by Remark~\ref{remark:Feb15symm}). This subsection is about the ``bicrystal'' structure, due to
\cite{bicrystal2, bicrystal1},
on ${\sf Mat}_{m, n}(\Z_{\geq 0})$ that intermingles the two ``halves''. 

\begin{definition}\label{def:feb11bicrystal}
    Let $M = [m_{i, j}]\in {\sf Mat}_{m,n}({\mathbb Z}_{\geq 0})$, and let $(i+1, b)$ and $(i, a)$ be the coordinates of the entries of $M$ yielding the letters of ${\sf row}(M)$ altered by $e_i$ and $f_i$ respectively. The \textit{row crystal raising operator} $e_i^{\sf row}$ sends $M$ to the matrix $e_i^{\sf row}(M)$ obtained by subtracting one from $m_{i+1, b}$ and adding one to $m_{i, b}$. Similarly, the \textit{row crystal lowering operator} $f_i^{\sf row}$ sends $M$ to the matrix $f_i^{\sf row}(M)$ obtained by subtracting one from $m_{i, a}$ and adding one to $m_{i+1, a}$. When $e_i({\sf row}(M)) = \varnothing$ or $f_i({\sf row}(M)) = \varnothing$, we define the corresponding row bicrystal operators to output the special symbol $\varnothing$ instead of a matrix.
    The \textit{column bicrystal operators} $e_j^{\sf col}$ and $f_j^{\sf col}$ are defined analogously using ${\sf col}(M)$ (or using transposes: $e_j^{\sf col}(M) = (e_j^{\sf row}(M^t))^t$ and $f_j^{\sf col}(M) = (f_j^{\sf row}(M^t)^t)$).
\end{definition}

\begin{definition}
    The \emph{matrix bicrystal graph} $\mathcal{M}_{m, n}$ is a pre-crystal graph with vertices $M\in{\sf Mat}_{m, n}(\Z_{\geq 0})$ with weights 
    \[{\sf cwt}_{{\mathcal M}_{m,n}}(M):=({\sf cwt}_{{\mathcal W}_m}({\sf row}(M)), 
    {\sf cwt}_{{\mathcal W}_n}({\sf col}(M))
    )\in {\mathbb Z}^{m+n}.\] 
    There is an $i^{\sf row}$-labelled edge from $M$ to $M'$ if and only if $M' = f_i^{\sf row}(M)$. There is a $j^{\sf col}$-labelled edge from $M$ to $M'$ if and only if $M' = f_j^{\sf col}(M)$.
\end{definition}

\begin{definition}
    Let $\mathcal{M}^{\sf row}_{m, n}$ and $\mathcal{M}^{\sf col}_{m, n}$ to be the pre-crystal subgraphs of $\mathcal{M}_{m, n}$ using only $i^{\sf row}$-labelled or $j^{\sf col}$-labelled edges respectively. 
    Set 
    \[{\sf cwt}_{\mathcal{M}^{\sf row}_{m, n}}(M):={\sf cwt}_{{\mathcal W}_{m}}({\sf row}(M)), \ {\sf cwt}_{\mathcal{M}^{\sf col}_{m, n}}(M):={\sf cwt}_{{\mathcal W}_{n}}({\sf col}(M)).\]
\end{definition}

\begin{example}\label{ex:matrixcrystal}
    Figure~\ref{fig:matrixcry} depicts the connected component $\mathcal{H}$ of $\left[ \begin{smallmatrix} 1 & 1 & 0\\ 1 & 0 & 0\\ 0 & 0 & 0 \end{smallmatrix}\right]$ in $\mathcal{M}^{\sf row}_{3, 3}$. A red arrow from $M$ to $M'$ indicates that $M' = f_1^{\sf row}(M)$ and a blue arrow indicates that $M' = f_2^{\sf row}(M)$. Notice ${\mathcal H}\cong \mathcal{B}_{\stableau{\ & \ \\ \ }, 3}$ is a pre-crystal graph isomorphism, under the map $M\mapsto {\sf tab}({\sf row}(M))$. This is an instance of Lemma~\ref{lemma:matrixcrystalreasonable} (below) combined with Proposition~\ref{thm:crystalmain}.
\begin{figure}
    \begin{center}
    \begin{tikzpicture}[scale = 0.45]
        \draw[red, thick][->] (2, 7) -- (4, 8.5);
        \draw[red, thick][->] (20, 8.5) -- (22, 7);
        \draw[blue, thick][->] (8, 10) -- (10, 10);
        \draw[blue, thick][->] (14, 10) -- (16, 10);
        \draw[blue, thick][->] (2, 3) -- (4, 1.5);
        \draw[blue, thick][->] (20, 1.5) -- (22, 3);
        \draw[red, thick][->] (8, 0) -- (10, 0);
        \draw[red, thick][->] (14, 0) -- (16, 0);
        \draw (0, 5) node{$\begin{bmatrix} 1 & 1 & 0 \\ 1 & 0 & 0 \\ 0 & 0 & 0 \end{bmatrix}$};
        \draw (6, 10) node{$\begin{bmatrix} 0 & 1 & 0 \\ 2 & 0 & 0 \\ 0 & 0 & 0 \end{bmatrix}$};
        \draw (12, 10) node{$\begin{bmatrix} 0 & 1 & 0 \\ 1 & 0 & 0 \\ 1 & 0 & 0 \end{bmatrix}$};
        \draw (18, 10) node{$\begin{bmatrix} 0 & 1 & 0 \\ 0 & 0 & 0 \\ 2 & 0 & 0 \end{bmatrix}$};
        \draw (6, 0) node{$\begin{bmatrix} 1 & 1 & 0 \\ 0 & 0 & 0 \\ 1 & 0 & 0 \end{bmatrix}$};
        \draw (12, 0) node{$\begin{bmatrix} 1 & 0 & 0 \\ 0 & 1 & 0 \\ 1 & 0 & 0 \end{bmatrix}$};
        \draw (18, 0) node{$\begin{bmatrix} 0 & 0 & 0 \\ 1 & 1 & 0 \\ 1 & 0 & 0 \end{bmatrix}$};
        \draw (24, 5) node{$\begin{bmatrix} 0 & 0 & 0 \\ 0 & 1 & 0 \\ 2 & 0 & 0 \end{bmatrix}$};
        \draw (2.5,8.5) node{$f_1^{\sf row}$};
        \draw (2.5,1.5) node{$f_2^{\sf row}$};
        \draw (9,11) node{$f_2^{\sf row}$};
        \draw (15,11) node{$f_2^{\sf row}$};
        \draw (21.5,8.5) node{$f_1^{\sf row}$};
        \draw (9,-1) node{$f_1^{\sf row}$};
        \draw (15,-1) node{$f_1^{\sf row}$};
        \draw (21.5,1.5) node{$f_2^{\sf row}$};
    \end{tikzpicture}
    \end{center}
    \caption{\label{fig:matrixcry} A connected component $\mathcal{H}$ of ${\mathcal M}^{\sf row}_{3,3}$}
  \end{figure}
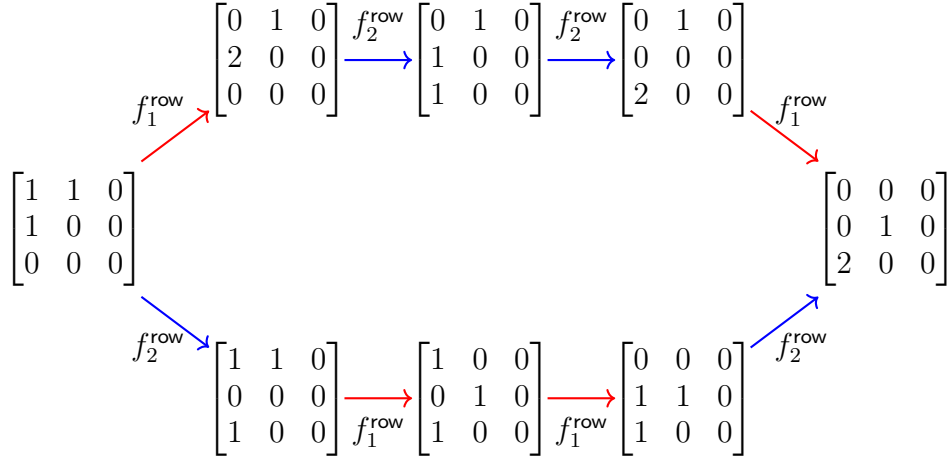
\end{example}

\begin{lemma}\label{lemma:matrixcrystalreasonable}
  The maps
    \[{\sf row}:\mathcal{M}^{\sf row}_{m, n}\to \mathcal{W}_m\quad \textrm{and}\quad {\sf col}:\mathcal{M}^{\sf col}_{m, n}\to \mathcal{W}_n,\]
    defined on vertices by $M\mapsto {\sf row}(M)$ and $M\mapsto {\sf col}(M)$ respectively are local isomorphisms.
\end{lemma}
\begin{proof}
We prove that ${\sf row}$ is a local isomorphism; the argument for ${\sf col}$ is entirely analogous.

\noindent
\emph{(${\sf row}$ is a pre-crystal graph homomorphism)}: By definition, ${\sf row}$ is weight-preserving.
Now, the row bicrystal operators are defined such that ${\sf row}(f_i^{\sf row}(M))=f_i({\sf row}(M))$. This means that ${\sf row}$ preserves adjacency.

\noindent
\emph{(${\sf row}$ is a local isomorphism)}: Fix a connected component ${\mathcal G}$ of $\mathcal{M}^{\sf row}_{m, n}$.
Fix $M_0\in {\mathcal G}$. We must show ${\sf row}|_{\mathcal G}$ is an isomorphism onto the connected component $\mathcal{H}$ of $\mathcal{W}_m$ that contains ${\sf row}(M_0)$. To see injectivity, let $M\in \mathcal{G}$ be arbitrary. Since $M_0$ and $M$ are connected by row bicrystal
operations, it follows from the definitions that ${\sf cwt}_{{\mathcal W}_n}({\sf col}(M)) = {\sf cwt}_{{\mathcal W}_n}({\sf col}(M_0))$. Now, any matrix $U\in{\sf Mat}_{m,n}(\Z_{\geq0})$ is uniquely reconstructible from 
${\sf row}(U)$ and ${\sf cwt}_{{\mathcal W}_n}({\sf col}(U))$. Thus $M$ is uniquely identified in $\mathcal{G}$ by its row word, so the restriction ${\sf row}|_{\mathcal{G}}$ is injective.
To show surjectivity, suppose there is an edge ${\sf row}(M)\stackrel{i}{\longrightarrow} w$ in $\mathcal{H}$. Then $f_i({\sf row}(M)) = w$, so by the definition of the row bicrystal operators, $w = {\sf row}(f_i^{\sf row}(M))$. Thus ${\sf row}|_{\mathcal{G}}$ is surjective onto $\mathcal{H}$, and this argument also shows that the inverse map ${\sf row}|_{\mathcal{G}}\inv$ is a pre-crystal graph homomorphism. Hence, ${\sf row}|_{\mathcal{G}}$ is an isomorphism onto ${\mathcal H}$ 
and ${\sf row}$ is a local isomorphism.
\end{proof}

Given Lemma~\ref{lemma:matrixcrystalreasonable}, one might expect a local isomorphism $\mathcal{M}_{m, n}\to \mathcal{W}_m\Box\mathcal{W}_n$ defined by $M\mapsto ({\sf row}(M)|{\sf col}(M))$.
However, this map is not a pre-crystal graph homomorphism (in general, $f_i^{\sf row}(M)\not\mapsto (f_i({\sf row}(M))|{\sf col}(M))$).
On the other hand, Proposition~\ref{thm:biwordcrystalequiv} shows that we \emph{do} obtain a local isomorphism when we replace $\mathcal{W}_m\Box\mathcal{W}_n$ with the quotient graph $\mathcal{K}_m\Box\mathcal{K}_n$.
To establish this result, we need the next two (known) lemmas. They can be proved via careful but elementary analysis of the matrix bicrystal operators.

\begin{lemma}[{\cite[Proposition 5.8]{bicrystal1}}]\label{lemma:crystalscommute1}
    Let $M \in{\sf Mat}_{m, n}(\Z_{\geq 0})$. Then
    \[{\sf col}(e_i^{\sf row}(M))\sim_K{\sf col}(M)\sim_K{\sf col}(f_i^{\sf row}(M)),\]
    assuming $e_i^{\sf row}(M)\neq \varnothing$ and $f_i^{\sf row}(M)\neq \varnothing$ for the two equivalences respectively. Similarly,
    \[{\sf row}(e_j^{\sf col}(M))\sim_K{\sf row}(M)\sim_K{\sf row}(f_j^{\sf col}(M)),\]
    assuming $e_j^{\sf col}(M)\neq \varnothing$ and $f_j^{\sf col}(M)\neq \varnothing$ for the two equivalences respectively.
\end{lemma}
\begin{lemma}[{\cite[Lemma 1.4.7]{bicrystal1}}]\label{lemma:crystalscommute2}
    The row and column bicrystal operators on $M\in{\sf Mat}_{m, n}(\Z_{\geq 0})$ commute. That is, for all $i$ and $j$,
    \begin{align*}
        f_i^{\sf row}(f_j^{\sf col}(M)) = f_j^{\sf col}(f_i^{\sf row}(M)), &\quad e_i^{\sf row}(e_j^{\sf col}(M)) = e_j^{\sf col}(e_i^{\sf row}(M))\\
        e_i^{\sf row}(f_j^{\sf col}(M)) = f_j^{\sf col}(e_i^{\sf row}(M)), &\quad f_i^{\sf row}(e_j^{\sf col}(M)) = e_j^{\sf col}(f_i^{\sf row}(M)).
    \end{align*}
\end{lemma}

Proposition~\ref{thm:biwordcrystalequiv} below is essentially classical RSK, associating $M\in{\sf Mat}_{m, n}(\Z_{\geq 0})$ to pairs of Knuth equivalence classes (which can then be labelled uniquely with tableaux). However, the local isomorphism reformulation will be convenient in the sequel.

\begin{proposition}\label{thm:biwordcrystalequiv}
    The map
    \[{\sf biword}: \mathcal{M}_{m, n}\to \mathcal{K}_m\Box\mathcal{K}_n,\]
    defined by $M\mapsto (\ [{\sf row}(M)]_K\ |\ [{\sf col}(M)]_K\ )$, is a local isomorphism.
\end{proposition}
\begin{proof}
\indent
\emph{(${\sf biword}$ is a pre-crystal graph homomorphism)}:
    By definition, ${\sf biword}$ preserves weights. It remains to check that it preserves adjacency. 
    Suppose $M\stackrel{i^{\sf row}}\longrightarrow M'$ is an edge in $\mathcal{M}_{m, n}$. By Lemma~\ref{lemma:crystalscommute1}, $[{\sf col}(M)]_K=[{\sf col}(M')]_K$. By Lemma~\ref{lemma:matrixcrystalreasonable},
    ${\sf row}(M)\stackrel{i}{\longrightarrow}{\sf row}(M')$ in ${\mathcal W}_m$. Therefore, by Definition~\ref{def:Knuthcrystalgraph}, $[{\sf row}(M)]_K\stackrel{i}{\longrightarrow}[{\sf row}(M')]_K$ in ${\mathcal K}_m$. Thus, ${\sf biword}(M)\stackrel{i^{\sf row}}{\longrightarrow} {\sf biword}(M')$ is
    an edge in $\mathcal{K}_m\Box\mathcal{K}_n$. The argument for an edge $M\stackrel{j^{\sf col}}\longrightarrow M'$ is similar.

\noindent
\emph{(${\sf biword}$ is a local isomorphism)}: Let $\mathcal{G}$ be a connected component of $\mathcal{M}_{m, n}$.
Fix $M_0\in {\mathcal G}$.

We show ${\sf biword}|_{\mathcal G}$ is injective. Suppose $M_1, M_2\in {\mathcal G}$ and ${\sf biword}(M_1) = {\sf biword}(M_2)$.
Fix a path ${\mathcal P}_1$ from $M_0$ to $M_1$, and a path ${\mathcal P}_2$ from $M_0$ to $M_2$.
By Lemma~\ref{lemma:crystalscommute2}, we may assume that ${\mathcal P}_1$ starts with a subpath
from $M_0$ to some $M_1'$ lying entirely in $\mathcal{M}^{\sf row}_{m, n}$ followed by a subpath from $M_1'$ to $M_1$
lying entirely in $\mathcal{M}^{\sf col}_{m, n}$.
Similarly, one defines $M_2'$.
Hence $M'_1$ and $M'_2$ both lie in the connected component of $M_0$ in $\mathcal{M}^{\sf row}_{m, n}$.
Moreover, by Lemma~\ref{lemma:crystalscommute1}, ${\sf row}(M_1')\sim_K {\sf row}(M_1)$ and ${\sf row}(M_2')\sim_K {\sf row}(M_2)$. Since ${\sf row}(M_1)\sim_K {\sf row}(M_2)$ by assumption, we see that ${\sf row}(M_1')\sim_K {\sf row}(M_2')$ by transitivity. Thus, $M'_1 = M'_2$ by Lemma~\ref{lemma:matrixcrystalreasonable} and Lemma~\ref{lemma:crystalknuth}.
Now, $M_1$ and $M_2$ both lie in the connected component of $M'_1 = M'_2$ in ${\mathcal M}^{\sf col}_{m, n}$. Since 
${\sf col}(M_1)\sim_K {\sf col}(M_2)$ by assumption, Lemma~\ref{lemma:matrixcrystalreasonable} and Lemma~\ref{lemma:crystalknuth} imply that $M_1 = M_2$. Thus ${\sf biword}|_{\mathcal G}$ is injective.    

To show surjectivity, suppose ${\sf biword}(M)\stackrel{i^{{\mathcal K}_m}}{\longrightarrow} ([u]_K|[{\sf col}(M)]_K)$ is an edge in $\mathcal{K}_m\Box\mathcal{K}_n$ for some $M\in\mathcal{G}$ (the case ${\sf biword}(M)\stackrel{j^{{\mathcal K}_n}}{\longrightarrow} ([{\sf row}(M)]_K|[v]_K)$ is similar). Thus, there exists an edge
$[{\sf row}(M)]_K\stackrel{i}{\longrightarrow}[u]_K$ in ${\mathcal K}_m$. By Lemma~\ref{lemma:matrixcrystalreasonable}, the connected component of ${\mathcal M}_{m,n}^{\sf row}$
containing $M$ is isomorphic to the connected component of ${\mathcal W}_m$ containing ${\sf row}(M)$ which
in turn is isomorphic, by Proposition~\ref{prop:pancake}, to the connected component of ${\mathcal K}_m$
containing $[{\sf row}(M)]_K$. From this, one sees there is an edge $M\stackrel{i^{\sf row}}{\longrightarrow} M'$ in ${\mathcal M}_{m,n}$ such that ${\sf biword}(M')=([u]_K|[{\sf col}(M)]_K)$. Thus
${\sf biword}|_{\mathcal G}$ is surjective onto the connected component ${\mathcal H}$ 
of $\mathcal{K}_m\Box\mathcal{K}_n$ that contains ${\sf biword}(M_0)$. It follows that
${\sf biword}|_{\mathcal G}:{\mathcal G}\to {\mathcal H}$ is a bijection of vertices, and this argument also shows that ${\sf biword}|_{\mathcal G}^{-1}$ preserves adjacency, completing the proof. 
\end{proof}

Proposition~\ref{thm:biwordcrystalequiv} encapsulates our claim from the introduction that the bicrystal operators on matrices are a ``pull-back" of the crystal operators on words. We illustrate the point with the commutative cube in Figure~\ref{fig:cube}.


\begin{figure}
    \[\adjustbox{scale=1,center}{%
    \begin{tikzcd}
	M &&&& {f_i^{\sf row}(M)} \\
	& {f_j^{\sf col}(M)} &&&& {f_i^{\sf row}(f_j^{\sf col}(M))} \\
	{([u]_K|[v]_K)} &&&& {([f_i(u)]_K|[v]_K)} \\
	& {([u]_K|[f_j(v)]_K)} &&&& {([f_i(u)]_K|[f_j(v)]_K)}
	\arrow["{f_j^{\sf col}}"{description}, color={rgb,255:red,11;green,168;blue,16}, from=1-1, to=2-2]
	\arrow["{f_i^{\sf row}}"{description}, color={rgb,255:red,150;green,55;blue,251}, from=1-1, to=1-5]
	\arrow["{f_i^{\sf row}}"{description}, color={rgb,255:red,150;green,55;blue,251}, from=2-2, to=2-6]
	\arrow["{f_j^{\sf col}}"{description}, color={rgb,255:red,11;green,168;blue,16}, from=1-5, to=2-6]
	\arrow["{{\sf biword}}"{description}, from=1-1, to=3-1]
	\arrow["{{\sf biword}}"{description, pos=0.3}, from=2-2, to=4-2]
	\arrow["{f_j^{\mathcal{K}_n}}"{description}, color={rgb,255:red,11;green,168;blue,16}, from=3-5, to=4-6]
	\arrow["{{\sf biword}}"{description, pos=0.3}, from=2-6, to=4-6]
	\arrow["{{\sf biword}}"{description, pos=0.3}, from=1-5, to=3-5]
	\arrow["{f_j^{\mathcal{K}_n}}"{description}, color={rgb,255:red,11;green,168;blue,16}, from=3-1, to=4-2]
	\arrow["{f_i^{\mathcal{K}_m}}"{description}, color={rgb,255:red,150;green,55;blue,251}, from=4-2, to=4-6]
	\arrow["{f_i^{\mathcal{K}_m}}"{description}, color={rgb,255:red,150;green,55;blue,251}, from=3-1, to=3-5]
    \end{tikzcd}}\]
    \caption{\label{fig:cube} An illustration of Proposition~\ref{thm:biwordcrystalequiv}, where $u = {\sf row}(M)$ and $v = {\sf col}(M)$.}
\end{figure}
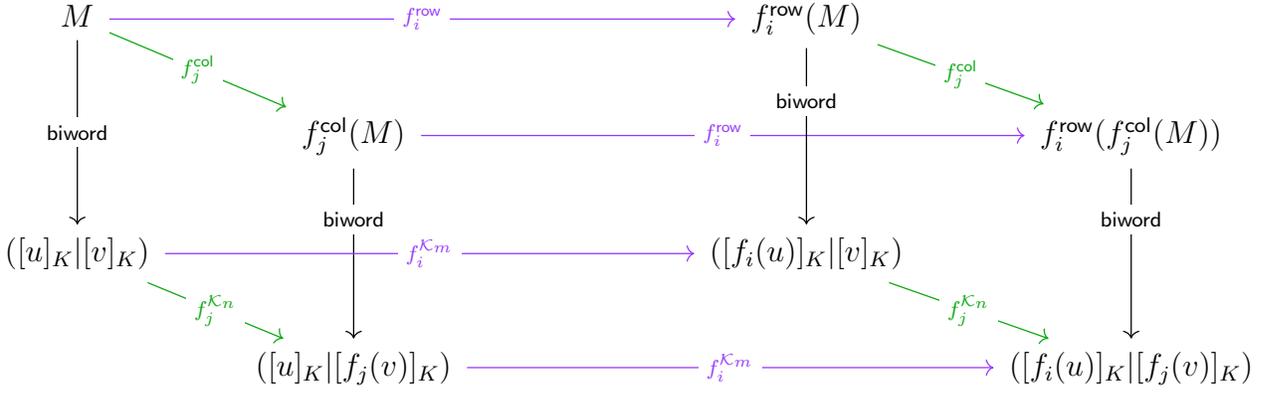

\section{Filterings and the proof of Theorem~\ref{thm:main}}\label{subsec:filterings}
This section shows how the results of Section~\ref{subsec:crystals} generalize using the ${\sf filterRSK}_{{\mathbf I}|{\mathbf J}}$ map, leading to our proof of Theorem~\ref{thm:main}. As in (\ref{eqn:twointseq}), fix sequences of nonnegative integers
\[
{\bf I} = \{0 = i_0<\dots<i_r = m\} \text{ \  and \ } {\bf J} = \{0 = j_0<\dots<j_s = n\}.
\]

\begin{definition}
    The \emph{${\bf I}$-filtered word crystal graph} $\mathcal{W}_m^{\bf I}$ is the subgraph of $\mathcal{W}_m$ obtained by deleting all edges labelled by elements $i\in {\bf I}$.
\end{definition}

\begin{proposition}\label{thm:filteredwordcrystal}
    Let ${\bf I}$ be as in (\ref{eqn:twointseq}). The map
   ${\sf filter}_{\mathbf{I}}: \mathcal{W}^{\mathbf{I}}_m \to \mathcal{W}_{[1, i_1]}\Box \mathcal{W}_{[i_1+1, i_2]}\Box\cdots\Box\mathcal{W}_{[i_{r-1}+1, m]}$,
   defined on vertices by $w\mapsto {\sf filter}_{\mathbf{I}}(w)$, is a local isomorphism.
\end{proposition}

Let us emphasize that ${\sf filter}_{\bf I}$ is not an isomorphism. For example, if $m=2$
and ${\bf I}=\{0,1,2\}$ then ${\sf filter}_{\bf I}(12)={\sf filter}_{\bf I}(21)=(1,2)$.

\begin{proof}
\noindent\emph{(${\sf filter}_{\mathbf{I}}$ is a pre-crystal graph homomorphism)}: By Definition~\ref{def:cartprodcrystal},
\begin{align*}
{\sf cwt}_{\mathcal{W}_{[1, i_1]}\Box\cdots\Box\mathcal{W}_{[i_{r-1}+1, m]} }({\sf filter}_{\mathbf{I}}(w)) &=
 ({\sf cwt}_{ \mathcal{W}_{[1, i_1]}}(w^{(1)}),\dots, {\sf cwt}_{ \mathcal{W}_{[i_{r-1}+1, m]}}(w^{(r)})) = {\sf cwt}_{ \mathcal{W}^{\mathbf{I}}_m}(w).
\end{align*}
Hence ${\sf filter}_\mathbf{I}$ is weight-preserving.

Now we show that  ${\sf filter}_\mathbf{I}$ preserves adjacency. Suppose $w\stackrel{i}{\longrightarrow}w'$ is an
edge in  $\mathcal{W}^{\mathbf{I}}_m$. By definition, ${\sf filter}_{{\mathbf I}}(w')$ and ${\sf filter}_{{\mathbf I}}(w)$ agree in each component except the $k$-th, where $i_{k-1}+1\leq i\leq i_k$.
In that component, $w'^{(k)} = f_i(w^{(k)})$. That is, $w^{(k)}\stackrel{i}{\longrightarrow} w'^{(k)}$ is an
edge in ${\mathcal W}_{[i_{k-1}+1,i_k]}$. Thus, by Definition~\ref{def:cartprodcrystal}, 
${\sf filter}_{\bf I}(w)\stackrel{i}{\longrightarrow}{\sf filter}_{\bf I}(w')$ is an edge 
in the Cartesian product pre-crystal graph $\mathcal{W}_{[1, i_1]}\Box \mathcal{W}_{[i_1+1, i_2]}\Box\cdots\Box\mathcal{W}_{[i_{r-1}+1, m]}$, as required.

\noindent\emph{(${\sf filter}_\mathbf{I}$ is a local isomorphism)}:
Fix $\mathcal{G}$ a connected component of $\mathcal{W}_m^{\mathbf{I}}$ and $w_0\in {\mathcal G}$.
We need to show that ${\sf filter}_{\mathbf{I}}$ restricts to an isomorphism $\mathcal{G}\to\mathcal{G}^{(1)}\Box\cdots\Box\mathcal{G}^{(k)}$, where $\mathcal{G}^{(k)}$ is the connected component of $\mathcal{W}_{[i_{k-1}+1, i_k]}$ containing $w_0^{(k)}$. 

First we show that ${\sf filter}_{\bf I}|_{\mathcal G}$ is injective. Suppose 
${\sf filter}_{\bf I}|_{\mathcal G}(u)={\sf filter}_{\bf I}|_{\mathcal G}(v)=(\pi^{(1)},\ldots,\pi^{(r)})$ 
for some $u,v\in~{\mathcal G}$. Both $u$ and $v$ are interweavings of the words $\pi^{(1)},\ldots,\pi^{(r)}$, i.e., the letters appearing in $\pi^{(t)}$ are in the same order in both $u$ and $v$, but in possibly different positions. However, any sequence of crystal operations $f_i,e_i$ ($i\not\in {\bf I}$) applied to $u$ does not affect the set of positions of $u$ that are occupied by  $[i_{t-1}+1,i_t]$. Hence $u$ and $v$ cannot be in the same connected component
unless $u=v$.

For surjectivity, suppose ${\sf filter}_{\bf I}(w)\stackrel{i}{\longrightarrow}{\alpha}$ is an edge in $\mathcal{W}_{[1, i_1]}\Box \mathcal{W}_{[i_1+1, i_2]}\Box\cdots\Box\mathcal{W}_{[i_{r-1}+1, m]}$ for some $w\in\mathcal{G}$, where $i\notin\mathbf{I}$. Suppose
$k$ is the unique index such that $i\in [i_{k-1}+1, i_k]$. Since ${\sf bracket}_i(w) = {\sf bracket}_i(w^{(k)})$,
clearly ${\sf filter}_{\bf I}(f_i(w))=\alpha$. 

Thus, the inverse map $({\sf filter}_{\bf I}|_{\mathcal{G}})^{-1}$ exists and preserves adjacency as required. 
\end{proof}

Taking the quotient of $\mathcal{W}^{\mathbf{I}}_m$ by $\sim_K$ produces the \emph{$\mathbf{I}$-filtered Knuth crystal graph} $\mathcal{K}^{\mathbf{I}}_m$. We now show that ${\sf filter}_{\mathbf{I}}$ descends to a local isomorphism on these quotients.

\begin{corollary}\label{cor:filterknuthequiv}
    Let ${\bf I}$ be as in (\ref{eqn:twointseq}). The map
   \[{\overline {\sf filter}}_{\mathbf{I}}: \mathcal{K}^{\mathbf{I}}_m \to \mathcal{K}_{[1, i_1]}\Box \cdots\Box\mathcal{K}_{[i_{r-1}+1, m]},\]
   defined on vertices using representatives $([w]_K\mapsto [{\sf filter}_{\mathbf{I}}(w)]_K)$, is a local isomorphism.
\end{corollary}
\begin{proof}
\noindent
\emph{(${\overline {\sf filter}}_{\mathbf{I}}$ is well-defined)}:  We need 
    that $w\sim_K w'$ implies ${\sf filter}_{\mathbf{I}}(w)\sim_K {\sf filter}_{\mathbf{I}}(w')$ (where Knuth equivalence of filtered words is defined component-wise). The only elementary Knuth moves \eqref{eqn:leftKnuthmove}, \eqref{eqn:rightKnuthmove} on $w$ that alter ${\sf filter}_{\mathbf{I}}(w)$ are those swapping two elements $x, z\in [i_{k-1}+1, i_k]$ for some $1\leq k\leq r$. Such a move is of the form $xzy\leftrightarrow zxy$ or $yxz \leftrightarrow yzx$ for $y\in [i_{k-1}+1, i_k]$, and is therefore also an elementary Knuth move on $w^{(k)}$, as desired.

\noindent
\emph{(${\overline {\sf filter}}_{\mathbf{I}}$ is a local isomorphism)}: Fix a connected component
${\overline{\mathcal G}}$ of ${\mathcal K}_m^{\bf I}$ and a vertex $\sigma_0$ in $\overline{\mathcal G}$. Pick a representative $w_0$ for $\sigma_0$. $\overline{\mathcal{G}}$ is isomorphic to the connected component $\mathcal{G}$ of $\mathcal{W}^{\mathbf{I}}_m$ containing $w_0$, by Proposition~\ref{prop:pancake}. The proof follows from
the commuting diagram:
\[\begin{tikzcd}
	{\mathcal{G}} && {\mathcal{G}_1\Box\mathcal{G}_2\Box\cdots\Box\mathcal{G}_r} \\
	\\
	{\overline{\mathcal{G}}} && {\overline{\mathcal{G}}_1\Box\overline{\mathcal{G}}_2\Box\cdots\Box\overline{\mathcal{G}}_r}
	\arrow["{[-]_K}"', tail reversed, from=1-1, to=3-1]
	\arrow["{{\sf filter}_{\mathbf{I}}}", tail reversed, from=1-1, to=1-3]
	\arrow["{\overline{{\sf filter}}_{\mathbf{I}}}"', dashed, from=3-1, to=3-3]
	\arrow["\cong", from=1-1, to=3-1]
	\arrow["{[-]_K\Box[-]_K\Box\cdots\Box[-]_K}", tail reversed, from=1-3, to=3-3]
	\arrow["\cong"', from=1-3, to=3-3]
	\arrow["\cong"', from=1-1, to=1-3]
\end{tikzcd}\]
By Proposition~\ref{thm:filteredwordcrystal}, ${\sf filter}_{\mathbf{I}}|_{\mathcal{G}}$ is an isomorphism from $\mathcal{G}$ to the product graph $\mathcal{G}_1\Box\mathcal{G}_2\Box\cdots\Box\mathcal{G}_r$, where $\mathcal{G}_k$ is the connected component of $w_0^{(k)}$ in $\mathcal{W}_{[i_{k-1}+1, i_k]}$. By Proposition~\ref{prop:pancake}, taking Knuth equivalence classes of filtered words then gives a third isomorphism from $\mathcal{G}_1\Box\mathcal{G}_2\Box\cdots\Box\mathcal{G}_r$ to the quotient graph $\overline{\mathcal{G}}_1\Box\overline{\mathcal{G}}_2\Box\cdots\Box\overline{\mathcal{G}}_r$ in $\mathcal{K}_{[1, i_1]}\Box\cdots\Box\mathcal{K}_{[i_{r-1}+1, m]}$. 
\end{proof}

\begin{definition}
    The \emph{$({\bf I}|{\bf J})$-filtered matrix bicrystal graph} $\mathcal{M}_{m, n}^{{\bf I}|{\bf J}}$ is the subgraph of $\mathcal{M}_{m, n}$ obtained by deleting all edges $i^{\sf row}$ with $i\in \mathbf{I}$ and $j^{\sf col}$ with $j\in \mathbf{J}$.
\end{definition}

\begin{definition}
    An \emph{(${\bf I}$-filtered) partition-tuple} is a sequence of partitions $\underline\lambda = (\lambda^{(1)},\dots, \lambda^{(r)})$, where each $\lambda^{(k)}$ has at most $i_k-i_{k-1}$ rows. An \emph{(${\bf I}$-filtered) tableau-tuple} of \emph{shape} $\underline\lambda$ is an ordered sequence of tableaux $\underline T = (T^{(1)},\ldots, T^{(r)})$ with each 
    $T^{(k)}\in {\sf SSYT}(\lambda^{(k)}, [i_{k-1}+1, i_k])$.
\end{definition}

\begin{definition}
    Let $\underline\lambda = (\lambda^{(1)},\dots, \lambda^{(r)})$ be an $\mathbf{I}$-filtered partition-tuple. The \emph{${\bf I}$-filtered tableau crystal graph} $\mathcal{B}^{\bf I}_{\underline\lambda, m}$ is the Cartesian product $\mathcal{B}_{\lambda^{(1)}, [1, i_1]}\Box\mathcal{B}_{\lambda^{(2)}, [i_1+1, i_2]}\Box\cdots\Box\mathcal{B}_{\lambda^{(r)}, [i_{r-1}+1, m]}$.
\end{definition}

We can now show that the map ${\sf filterRSK}_{{\mathbf I}|{\mathbf J}}$ of Definition~\ref{alg:filteredrsk} is a local isomorphism.

\begin{proposition}\label{thm:filtercrystalequiv}
    Let ${\bf I}$ and ${\bf J}$ be as in (\ref{eqn:twointseq}). Then the map
    \[{\sf filterRSK}_{\mathbf{I}|\mathbf{J}}:\mathcal{M}^{\mathbf{I}|\mathbf{J}}_{m, n}\to \bigoplus_{\underline\lambda, \underline\mu}(\mathcal{B}_{\underline\lambda, m}\Box\mathcal{B}_{\underline\mu, n}),\]
    given on vertices by mapping $M\mapsto{\sf filterRSK}_{\mathbf{I}|\mathbf{J}}(M)$, is a local isomorphism.
\end{proposition}
\begin{proof}
    By Proposition~\ref{thm:biwordcrystalequiv}$, {\sf biword}: \mathcal{M}_{m, n}\to \mathcal{K}_m\Box\mathcal{K}_n$ is a local isomorphism. It follows immediately (from the proof) that ${\sf biword}: \mathcal{M}^{\mathbf{I}|\mathbf{J}}_{m, n}\to \mathcal{K}^{\mathbf{I}}_m\Box\mathcal{K}^{\mathbf{J}}_n$ (defined the same on vertices) is also a local isomorphism, since on both the domain and codomain we are simply deleting all edges with labels in $\mathbf{I}$ or $\mathbf{J}$.
    
    We may then understand ${\sf filterRSK}_{\mathbf{I}|\mathbf{J}}$ as a composition of the maps ${\sf biword}: \mathcal{M}^{\mathbf{I}|\mathbf{J}}_{m, n}\to \mathcal{K}^{\mathbf{I}}_m\Box\mathcal{K}^{\mathbf{J}}_n$, ${\overline {\sf filter}}_{\mathbf{I}}: \mathcal{K}^{\mathbf{I}}_m \to \mathcal{K}_{[1, i_1]}\Box \mathcal{K}_{[i_1+1, i_2]}\Box\cdots\Box\mathcal{K}_{[i_{r-1}+1, m]}$, and ${\overline {\sf tab}}:\mathcal{K}_{[a, b]}\to \bigoplus_\lambda \mathcal{B}_{\lambda, [a, b]}$. These three maps are local isomorphisms by the previous paragraph, Corollary~\ref{cor:filterknuthequiv}, and Corollary~\ref{cor:tablocaliso} respectively. Now apply Lemma~\ref{lemma:localisoprops}, parts (II) and (III). 
\end{proof}

\begin{corollary}\label{cor:highestweightmatrix}
    Any connected component $\mathcal{G}$ of $\mathcal{M}^{\mathbf{I}|\mathbf{J}}_{m, n}$ has a unique source vertex $M$ (i.e., a vertex of highest-weight). Moreover, ${\sf filterRSK}_{\mathbf{I}|\mathbf{J}}(M) = (T_{\underline\lambda}|T_{\underline\mu})$ for some partition-tuples $\underline\lambda$ and $\underline\mu$, and then $\mathcal{G}\cong \mathcal{B}_{\underline\lambda, m}\Box\mathcal{B}_{\underline\mu, n}$.
\end{corollary}
\begin{proof}
    By Proposition~\ref{thm:filtercrystalequiv}, ${\sf filterRSK}$ restricts to an isomorphism $\mathcal{G}\cong\mathcal{B}_{\underline\lambda, m}\Box\mathcal{B}_{\underline\mu, n}$ for some $\underline\lambda$ and $\underline\mu$. Now apply Theorem~\ref{thm:crystalgraph} along with the fact that any source vertex in a Cartesian product of graphs is given by a product of source vertices in each factor.
\end{proof}

\begin{theorem}[Bicrystal categorification of Theorem~\ref{thm:main}]\label{cor:filterRSKgraphiso}
If ${\mathfrak X}\subseteq {\sf Mat}_{m,n}$ is ${\bf L}_{\bf I|\bf J}$-bicrystal closed, then 
the monomials in ${\sf Std}_<(S/I(\mathfrak{X}))$ (interpreted as matrices) are vertices of an induced pre-crystal subgraph ${\mathcal S}_{\mathfrak X}$ of 
${\mathcal M}_{m,n}^{\bf I|\bf J}$. Now ${\sf filterRSK}_{{\bf I}|{\bf J}}$ gives a pre-crystal graph
isomorphism
\begin{equation}\label{eqn:Feb28fff}
{\mathcal S}_{\mathfrak X}\cong
\bigoplus_{\underline\lambda|\underline\mu}(\mathcal{B}_{\underline\lambda, m}\Box\mathcal{B}_{\underline\mu, n})^{\oplus c_{\underline\lambda|\underline\mu}^{\mathfrak X}},
\end{equation}
where the sum is over all partition-tuples $(\underline\lambda|\underline\mu)$. The multiplicities
$c_{\underline\lambda|\underline\mu}^{\mathfrak X}$ counts highest-weight matrices $M\in {\sf Std}_<(S/I(\mathfrak{X}))$
of shape $(\underline\lambda|\underline\mu)$.
\end{theorem}
\begin{proof}
The first sentence is by the definition of ${\mathfrak X}$ being ${\bf L}_{\bf I|\bf J}$-bicrystal closed.
    Any pre-crystal graph is the disjoint union of its connected components; the hypothesis of being
    ${\bf L}_{\bf I|\bf J}$-bicrystal closed says that a connected component ${\mathcal G}$ of ${\mathcal S}_{\mathfrak X}$
    is a connected component  of ${\mathcal M}_{m,n}^{\bf I|\bf J}$.
Therefore, we are done by Proposition~\ref{thm:filtercrystalequiv}, Lemma~\ref{lemma:localisoprops}(I),
and Corollary~\ref{cor:highestweightmatrix}.
    \end{proof}

\noindent
\emph{Proof of Theorem \ref{thm:main}:} Section~\ref{sec:reptheory} and specifically
Proposition~\ref{thm:bigsummary} implies (\ref{eqn:char=hilb}), since each standard monomial ${\sf m}$ spans
a weight space with weight ${\sf wt}({\sf m})$ with respect to the algebraic torus action
of $\mathbf{T}=({\mathbb C}^\star)^m\times ({\mathbb C}^\star)^n$ (see Section~\ref{subsec:gln}). 
Now, the result follows from
Theorem~\ref{cor:filterRSKgraphiso} 
combined with the hypothesis that ${\mathfrak X}$ is ${\bf L}_{{\bf I}|{\bf J}}$-bicrystalline.
Specifically, in that theorem, the pre-crystal graph isomorphism is, by definition, weight-preserving. These
weights agree with ${\sf wt}_{\mathbf{T}}({\sf m})$ and 
the weight of semistandard tableau in its contribution to a Schur polynomial. Thus the 
equality in \eqref{eqn:themainequality} follows from the isomorphism.
\qed

\section{Main application: Matrix Schubert varieties} \label{sec:MSVs}

\subsection{Matrix Schubert varieties}\label{subsec:MSVs}

Let $B_k\subset GL_k$ denote the Borel group of lower triangular matrices. Then $\mathbf{B} = B_m\times B_n$ acts on ${\sf Mat}_{m,n}$ on the right by 
$M\cdot (b_1,b_2)=b_1^{-1}M (b_2^{-1})^{T}$. We recount the construction of all $\mathbf{B}$-stable varieties in ${\sf Mat}_{m, n}$ due to Fulton \cite{Fulton:duke}.

\begin{definition}
    A \emph{partial permutation matrix} $M_w\in{\sf Mat}_{m, n}$ is a matrix with at most one $1$ in each row and column and $0$s everywhere else. The indexing \emph{partial permutation} is a function $w:[m]\to[n]\cup\{\infty\}$, where $w(i) = j\in[n]$ if $M_w$ has a $1$ in position $(i, j)$ and $w(i) = \infty$ if $M_w$ has no $1$ in row $i$. The \emph{rank function} for $M_w$ is denoted $r_w:[m]\times[n]\to \Z_{\geq 0}$; it maps $(i, j)$ to the rank of the northwest $i\times j$ submatrix of $M_w$.
\end{definition}

\begin{theorem}[{\cite[Lemma 3.1]{Fulton:duke}}]\label{thm:Mar6abc}
    Every $\mathbf{B}$-orbit in ${\sf Mat}_{m, n}$ contains a unique $M_w$.
\end{theorem}

\begin{defthm}[{\cite[Proposition 3.3]{Fulton:duke}}]\label{thm:MSV}
   The \emph{matrix Schubert variety} ${\mathfrak X}_w\subseteq{\sf Mat}_{m, n}$ is the Zariski closure of the orbit $M_w\cdot\mathbf{B}$. Each ${\mathfrak X}_w$ is irreducible.
\end{defthm} 

The following immediate consequence of Definition-Theorem~\ref{thm:MSV} is known:

\begin{corollary}\label{thm:borelunion}
    Any $\mathbf{B}$-stable variety $\mathfrak{X}\subseteq{\sf Mat}_{m, n}$ is a finite union of matrix Schubert varieties.
\end{corollary}
\begin{proof}
There are only finitely many partial permutation matrices in ${\sf Mat}_{m, n}$. Now apply Theorem~\ref{thm:Mar6abc}
and Definition-Theorem~\ref{thm:MSV}.
\end{proof}

In light of Corollary~\ref{thm:borelunion}, we restrict our attention to matrix Schubert varieties $\mathfrak{X}_w\subseteq{\sf Mat}_{m, n}$. Equations for these varieties were given in \cite{Fulton:duke}. We recall the standard permutation combinatorics needed to describe them.

The \emph{graph} of a partial permutation matrix $M_w$ is an $m\times n$ grid with a $\bullet$ symbol in the entries where $M_w$ has a $1$ and blank spaces elsewhere. The \emph{Rothe diagram} of $M_w$, denoted $D(w)$, consists of all boxes in $[m]\times[n]$ not weakly below or right of a $\bullet$ in the graph of $M_w$. 
The \emph{essential set} $E(w)$ of $M_w$ is comprised of the maximally southeast boxes of each connected component of $D(w)$, i.e.,
$E(w)=\{(i,j)\in D(w): (i,j+1), (i+1,j)\not\in D(w)\}$.
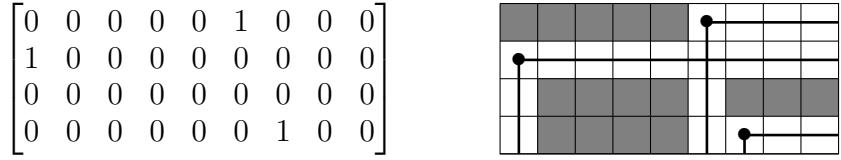
\begin{figure}
\centering
\begin{tikzpicture}[scale = 0.5]
    \fill [gray, opacity  = 0.25] (0,4) rectangle (5,3);
    \fill [gray, opacity  = 0.25] (1,2) rectangle (5,0);
    \fill [gray, opacity  = 0.25] (6,2) rectangle (9,1);
    \draw (0,4)--(9,4)--(9,0)--(0,0)--(0,4);
    \draw (1,0) -- (1,4);
    \draw (2,0) -- (2,4);
    \draw (3,0) -- (3,4);
    \draw (4,0) -- (4,4);
    \draw (5,0) -- (5,4);
    \draw (6,0) -- (6,4);
    \draw (7,0) -- (7,4);
    \draw (8,0) -- (8,4);
    \draw (0,0) -- (9,0);
    \draw (0,1) -- (9,1);
    \draw (0,2) -- (9,2);
    \draw (0,3) -- (9,3);
 
    \node at (-8,2) {$\left[\begin{matrix}
0 & 0 & 0 & 0 & 0 & 1 & 0 & 0 & 0\\
1 & 0 & 0 & 0 & 0 & 0 & 0 & 0 & 0\\
0 & 0 & 0 & 0 & 0 & 0 & 0 & 0 & 0\\
0 & 0 & 0 & 0 & 0 & 0 & 1 & 0 & 0\\
\end{matrix}\right]$};

    \node at (5.5,3.5) {$\bullet$};
    \node at (0.5,2.5) {$\bullet$};
    \node at (6.5,0.5) {$\bullet$};
    \draw[line width = 0.35mm] (5.5,0) -- (5.5,3.5) -- (9,3.5);
    \draw[line width = 0.35mm] (0.5,0) -- (0.5,2.5) -- (9,2.5);
    \draw[line width = 0.35mm] (6.5,0) -- (6.5,0.5) -- (9,0.5);
    \end{tikzpicture}
    \caption{The partial permutation matrix $M_w$ for $w = {61\infty 7}$ and its Rothe diagram; the boxes of $D(w)$ are shaded.
    \label{fig:61infinity7}}
\end{figure}
\begin{definition}
    The \emph{Coxeter length} of $w$ is the size of its Rothe diagram, $\ell(w):=\#\{D(w)\}$.
    The \emph{row descent} positions of $M_w$, denoted $\desc_{\sf row}(w)$, consists of row indices $i$ such that the rightmost box in row $i$ of $D(w)$ is strictly right of the rightmost box in row $i+1$. Similarly, the \emph{column descent} positions of $M_w$ is the set $\desc_{\sf col}(w)$ of column indices $j$ such that the bottom-most box in row $j$ of $D(w)$ is strictly below the bottom-most box in row $j+1$.
\end{definition}

\begin{remark}\label{remark:wbar}
    By \cite[Equation 3.4]{Fulton:duke}, any $M_w\in{\sf Mat}_{m, n}$ is indexed by a unique permutation $\overline{w}\in {\mathfrak S}_{m+n}$ with $\ell({\overline w})=\ell(w)$. Indeed, $M_w$ is the northwest $m\times n$ submatrix of the \emph{permutation matrix} in ${\sf Mat}_{m+n, m+n}$ with $1$s in positions $(i, {\overline w}(i))$ and $0$s elsewhere. The definitions above for partial permutation matrices agree with the corresponding definitions for the indexing permutation $\overline{w}$. In particular, $\desc_{\sf row}(w) = \desc(\overline{w})$ and $\desc_{\sf col}(w) = \desc(\overline{w}\inv)$.
\end{remark}

\begin{example}
Figure~\ref{fig:61infinity7} depicts $M_w$ and $D(w)$ for $w=61\infty 7$. Here, $\ell(w)=16$, 
$\desc_{\sf row}(w)=\{1,3\}$ whereas $\desc_{\sf col}(w)=\{5\}$. Referring to Remark~\ref{remark:wbar},
${\overline w}=61(10)7234589(11)(12)(13)$.
\end{example}

We now state a concrete description of matrix Schubert varieties.
Make the identification $\C[{\sf Mat}_{m, n}]=\C[z_{ij}:1\leq i\leq m,1\leq j\leq n]$ where $z_{ij}$ is the $(i,j)$-coordinate function.

\begin{theorem}[{\cite[Proposition~3.3]{Fulton:duke}}]\label{theorem:Fultonsresult}
    ${\mathfrak X}_w\subseteq {\sf Mat}_{m, n}$ is the set of $m\times n$ matrices $M$ such that the rank of the northwest $i\times j$ submatrix of $M$ has rank $\leq r_w(i, j)$.

    Let $Z=(z_{ij})_{1\leq i\leq m, 1\leq j\leq n}$ be the generic $m\times n$ matrix and set $Z_{ij}$ to be the northwest $i\times j$ submatrix of $Z$.
    Then the defining ideal of $\mathfrak{X}_w\subseteq{\sf Mat}_{m, n}$ is the \emph{Schubert determinantal ideal}
    \begin{equation}\label{eqn:thegenerators}
        I_w := I(\mathfrak{X}_w) = \langle \text{rank $r_w(i, j)+1$ minors of $Z_{ij}$, $(i,j)\in E(w)$}\rangle.
    \end{equation}
    This ideal is prime. Under this embedding ${\mathfrak X}_w\subseteq {\sf Mat}_{m, n}$ has codimension 
    $\ell(w)$.
\end{theorem}

By definition, each $\mathfrak{X}_w$ is a $\mathbf{B}$-variety. They are therefore also $\mathbf{T}$-varieties by restriction. The $\mathbf{T}$-character of $\mathfrak{X}_w$ (in the guise of the multigraded Hilbert series) has been extensively studied in \cite{Knutson.Miller, KM:adv}. Our point of departure from \cite{Fulton:duke, Knutson.Miller, KM:adv} begins by observing that each $\mathfrak{X}_w$ is also a $\mathbf{L}_{{\mathbf I}|{\mathbf J}}$-variety for appropriate choices of $\mathbf{I}$ and $\mathbf{J}$.

\begin{proposition}\label{prop:leviacts}
    The matrix Schubert variety $\mathfrak{X}_w$ is $\mathbf{L}_{{\mathbf I}|{\mathbf J}}$-stable with respect to the right action $(g,g')\cdot A = g^{-1} A (g'^{-1})^T$ whenever $\desc_{{\sf row}}(w)\subset \mathbf{I}$ and $\desc_{{\sf col}}(w)\subset \mathbf{J}$.
\end{proposition}
\begin{proof}
    By restriction, it is enough to show that ${\bf L} = \mathbf{L}_{{\mathbf I}|{\mathbf J}}$ acts on~${\mathfrak X}_w$ in the case that
    \[{\bf I} = \{0\}\cup \desc_{{\sf row}}(w)\cup \{m\}= \{i_0<i_1<\ldots<i_r\}\]
    and
    \[{\bf J}= \{0\}\cup \desc_{{\sf col}}(w)\cup \{n\} = \{j_0<j_1<\ldots<j_s\}.\] 
     Let $I\in {\bf L}$ be the identity and $A\in {\sf Mat}_{m,n}$.  Of course, $I\cdot A = A$; it is
     straightforward to check that for $(g,g'),(h,h')\in {\bf L}$,  
     $(g,g')\cdot ((h,h')\cdot A)=((h,h')\cdot (g,g'))\cdot A$,
     so one has a right action on ${\sf Mat}_{m,n}$. Clearly this matrix-multiplication action is rational.

    It remains to check that this action restricts to ${\mathfrak X}_w$. To see this, first
fix $1\leq k\leq r$ and let 
\begin{equation}\label{eqn:Feb26type1}
g = (I_{i_1 - i_0},I_{i_2-i_1},\ldots, \tilde{g},\ldots,I_{i_r - i_{r-1}}),
\end{equation}
be the block matrix where $I_{a}$ is the $a\times a$ identity matrix and $\tilde{g}\in GL_{i_k - i_{k-1}}$. Let $A\in {\mathfrak X}_w$ and, for any $(i,j)$, let $A_{(i,j)}$ denote the northwest $i\times j$ submatrix of $A$. By \eqref{eqn:thegenerators}, if suffices to show that, given ${\bf e} = (e_1,e_2)\in E(w)$, $\text{rank}(g\cdot A)_{\bf e}\leq r_w(e_1,e_2)$. Indeed, 
if $e_1 \leq i_{k-1}$, then $(g\cdot A)_e = A_e$ and the result follows. Thus, we may assume $e_1 > i_{k-1}$. Then \[(g\cdot A)_{\bf e} = \begin{bmatrix}I_{i_{k-1}} & 0 & 0\\ 0 & \tilde{g}^{-1} & 0\\ 0& 0 & I_{i_{e_1 - i_k}}\end{bmatrix} A_{\bf e} := \tilde{G}A_{\bf e},\] where $\tilde{G}$ is the indicated $e_1\times e_1$ block-diagonal matrix. Since $\tilde{G}$ has full rank,
    $$\text{rank}((g\cdot A)_{\bf e}) = \text{rank}\left(\tilde{G} A_{\bf e}\right)\leq \text{rank}(A_{\bf e})\leq r_w(e_1,e_2),$$ as desired. Similarly, fix $1\leq k\leq s$ and let 
\begin{equation}\label{eqn:Feb26type2}
g' = (I_{j_1 - j_0},\ldots, \tilde{g'},\ldots, I_{j_s - j_{s-1}}),
\end{equation}
where $\tilde{g'}\in GL_{j_k-j_{k-1}}$. If $e_2 \leq j_{k-1}$, then $(g'\cdot A)_{\bf e} = A_{\bf e}$ and the result follows. Assume $j_{k-1} < e_2$ and let \[\tilde{G'} = \begin{bmatrix}I_{j_{k-1}} & 0 & 0\\ 0 & (\tilde{g'}^{-1})^T & 0\\ 0& 0 & I_{j_{e_2 - j_k}}\end{bmatrix}.\] As before, \[\text{rank}((g\cdot A)_{\bf e}) = \text{rank}(A_{\bf e}\tilde{G'}) \leq\text{rank}(A_{\bf e})\leq r_w(e_1,e_2).\]
    
    Given  $(h,h')\in\mathbf{L}$, $h$ factorizes into a product of matrices of the form (\ref{eqn:Feb26type1}), and
    $h'$ factorizes into matrices of the form (\ref{eqn:Feb26type2}). Thus, 
    the claim the right action restricts holds.
\end{proof}

As explained in Section~\ref{subsec:stdmono} (see \eqref{eqn:Feb13abc}), by Proposition~\ref{prop:leviacts}, $\C[\mathfrak{X}_w]$ is a $\mathbf{L}_{{\mathbf I}|{\mathbf J}}$-representation for any ${\bf I}\supseteq \desc_{{\sf row}}(w)$ and ${\bf J}\supseteq \desc_{{\sf col}}(w)$. We now check that 
${\mathfrak X}_w$ is ${\bf L}_{{\bf I}|{\bf J}}$-crystal closed in order to apply Theorem~\ref{thm:main}
to compute the $\mathbf{L}_{{\mathbf I}|{\mathbf J}}$-character of $\C[\mathfrak{X}_w]$.

\subsection{\texorpdfstring{Standard monomials of $\mathfrak{X}_w$ and proof of Theorem~\ref{thm:crystalclosure}}{Standard monomials and bicrystallinity of matrix Schubert varieties}}\label{subsec:MSVstdmono}
Fix an \emph{antidiagonal term order} $\prec$ on $S=\C[z_{11},\dots, z_{mn}]$, i.e., one that picks the
antidiagonal term of a minor. One example is pure lexicographic order obtained by setting $z_{ab}\succ z_{cd}$ if $a<c$, or $a=c$ and $b>d$. This is Knutson--Miller's Gr\"obner basis theorem:

\begin{theorem}[{\cite[Theorem~B]{Knutson.Miller}}]\label{thm:KMgrob}
Fulton's generators (\ref{eqn:thegenerators}) for $I_w$ form a Gr\"obner basis with respect to $\prec$.
\end{theorem}

If $g$ is a Fulton generator, we write $Z_g$ for the submatrix of $[z_{ij}]$ of which $g$ is the determinant. Any monomial ${\sf m}$ in the $z_{ij}$ can be naturally identified with a nonnegative integer matrix,
via its exponent vector. Identifying ${\sf m}$ with the corresponding matrix,
Theorem~\ref{thm:KMgrob} describes the standard monomials of ${\mathfrak X}_w$ in terms of matrices, as follows:

\begin{corollary}\label{cor:groebner}
    ${\sf m}\in {\sf Std}_{\prec}(S/I_w)$ if and only if for any Fulton generator $g=\det(Z_g)$ of $I_w$ from
    (\ref{eqn:thegenerators}), the product of the entries of ${\sf m}$ along the main antidiagonal of the submatrix $Z_g$ is $0$.
\end{corollary}

The standard monomials of any $\mathbf{B}$-stable variety $\mathfrak{X}\subset{\sf Mat}_{m, n}$ are described in terms of the standard monomials of matrix Schubert varieties. This follows immediately from work of Knutson on Frobenius splittings \cite{Knutson}, although it is not explicitly stated there.\footnote{See also \cite{Bertiger} for Bertiger's construction of a Gr\"obner basis for arbitrary $\mathbf{B}$-stable varieties.} A more explicit statement appears in \cite[Proposition 1.4]{WeigandtASM}.

\begin{proposition}[cf.~\cite{Knutson}]\label{thm:stdmonounion}
    Let $\mathfrak{X}\subseteq{\sf Mat}_{m, n}$ be a $\mathbf{B}$-stable variety. Write $\mathfrak{X} = \bigcup_{i=1}^k \mathfrak{X}_{w^{(i)}}$ as a union of matrix Schubert varieties (by Corollary~\ref{thm:borelunion}). Then the set of standard monomials for $\mathfrak{X}$, with respect to $\prec$, is 
${\sf Std}_{\prec}({\mathfrak X})=\bigcup_{i=1}^k {\sf Std}_{\prec}(S/I_{w^{(i)}})$.    
\end{proposition}
\begin{proof}
    Let $I\!=\!\bigcap_{i=1}^k I_{w^{(i)}}$ be the ideal of ${\mathfrak X}$ in $S$. 
    By \cite[Section 7.2]{Knutson}, 
    ${\sf init}_{\prec}(I)\!=\!\bigcap_{i=1}^k {\sf init}_{\prec}(I_{w^{(i)}})$.
    The claim follows since, by definition, ${\sf Std}_{\prec}({\mathfrak X})={\sf Std}_{\prec}(S/I)$ are the monomials in ${\sf init}_{\prec}(I)^c$ taken with coefficient $1$.
\end{proof}

We next prove Theorem~\ref{thm:crystalclosure} for the case $\mathfrak{X} = \mathfrak{X}_w$ is a matrix Schubert variety.

\begin{theorem}\label{thm:MSVLtype}
    Let $\mathfrak{X}_w\subseteq{\sf Mat}_{m, n}$ where $\desc_{{\sf row}}(w)\subseteq \mathbf{I}$ and $\desc_{{\sf col}}(w)\subseteq\mathbf{J}$. Then
    $\mathfrak{X}_w$ is ${\bf L}_{\mathbf{I}|\mathbf{J}}$-bicrystalline.
\end{theorem}
\begin{proof}
By Proposition~\ref{prop:leviacts}, ${\mathfrak X}_w$ is $\mathbf{L}_{{\mathbf I}|{\mathbf J}}$-stable. It remains to show that $\mathfrak{X}_w$ is $\mathbf{L}_{{\mathbf I}|{\mathbf J}}$-bicrystal closed. We focus first on the $\mathbf{I}$-filtered row operators $e_i^{\sf row}$ and $f_i^{\sf row}$. 
Since $e_i^{\sf row}$ and $f_i^{\sf row}$ are inverses whenever their outputs are not $\varnothing$ by Lemma~\ref{lemma:crystalinverse}, the 
$\mathbf{L}_{{\mathbf I}|{\mathbf J}}$-bicrystal closed claim follows from the (slightly stronger)
statement that if ${\sf m}\not\in {\sf Std}_<(S/I_w)$ (identified with a matrix in ${\sf Mat}_{m,n}({\mathbb Z}_{\geq 0})$), then $e_i^{\sf row}({\sf m})\not\in {\sf Std}_<(S/I_w)$ for all $i$ and $f_i^{\sf row}({\sf m})\not\in {\sf Std}_<(S/I_w)$ for $i\not\in\desc_{{\sf row}}(w)$. We prove this latter statement. 

Since ${\sf m}\not\in {\sf Std}_<(S/I_w)$, by Corollary~\ref{cor:groebner}, there exists a Fulton generator $g=\det(Z_g)$ for $I_w$ such that the product of all entries of ${\sf m}$ along the antidiagonal $A_g$ of $Z_g$ is nonzero. 
Let $R\subseteq [m]$ and $C\subseteq [n]$ be the row and column indices, respectively, of the minor $g$.

First, we argue that $e_i^{\sf row}({\sf m})\not\in {\sf Std}_<(S/I_w)$ by constructing a Fulton generator $g'$ for $I_w$ such that the product of all entries of $e_i^{\sf row}({\sf m})$ 
along $A_{g'}$ is positive.

\medskip
\noindent
{\sf Case e1:} ($i,i+1\not\in R$) Here, $e_i^{\sf row}$ does not affect any of the entries of $A_g$ in ${\sf m}$. Take $g' = g$.

\medskip
\noindent
{\sf Case e2:} ($i\in R, i+1\not\in R$) All entries of $A_g$ in $e_i^{\sf row}({\sf m})$ are only larger (by at most $1$)
in comparison to the same entry in ${\sf m}$, so we may again take $g' = g$.

\medskip
\noindent
{\sf Case e3:} ($i\not\in R, i+1\in R$) If $e_i^{\sf row}$ does not affect the entry of $A_g$ in row $i+1$,
take $g'=g$. Otherwise, we may take $g'$ to be the minor defined by row indices $R' = (R\setminus\{i+1\})\cup \{i\}$ and column indices $C$ (which is also a Fulton generator for $I_w$ by (\ref{eqn:thegenerators})).

\medskip\noindent
{\sf Case e4:} ($i,i+1\in R$) Let $b$ be as in Definition~\ref{def:feb11bicrystal}. If $m_{i+1, b}$ does not lie on the antidiagonal $A_g$ or $m_{i+1, b}\geq 2$, then we may take $g' = g$. Otherwise, since we assume $e_i^{\sf row}({\sf m})\neq \varnothing$, the entry ${\sf m}_{i+1,b}=1$ corresponds to an unmatched
``$($'' in ${\sf bracket}_i({\sf row}({\sf m}))$. Let $b'$ be the first column to the right of $b$ in $C$, so $(i,b')\in A_g$. None of the ``$)$'' in 
${\sf bracket}_i({\sf row}({\sf m}))$ associated to ${\sf m}_{i,b'}>0$ match with this aforementioned ``$($''. In particular, the leftmost ``$)$'' associated to ${\sf m}_{i,b'}$ matches with a ``$($'' associated
to ${\sf m}_{i+1,c'}>0$ for some $b<c'<b'$. Take $g'$ to be the minor defined by row indices $R$ and column indices $C' = (C\setminus\{b\})\cup\{c'\}$ (which is also a Fulton generator for $I_w$ by (\ref{eqn:thegenerators})).

Similarly, we show that $f^{\sf row}_i({\sf m})\not\in {\sf Std}_{<}(S/I_w)$ when $i\not\in\desc_{{\sf row}}(w)$ by constructing a Fulton generator $g''$ for $I_w$ such that the product of entries of $f^{\sf row}_i({\sf m})$ along $A_{g''}$ is positive.

\medskip
\noindent
{\sf Case f1:} ($i,i+1\not\in R$) Same argument as {\sf Case e1}, take $g'' = g$.

\medskip
\noindent
{\sf Case f2:} ($i\in R, i+1\not\in R$) If $f_{i}^{\sf row}$ does not affect the entry of $A_g$
in row $i$, let $g''=g$. Otherwise, let $g''$ be the minor that uses the rows $R''=(R\setminus \{i\})\cup \{i+1\}$ and columns $C$ (which is also a Fulton generator of $I_w$ provided $i\not\in\desc_{{\sf row}}(w)$ by (\ref{eqn:thegenerators})).

\medskip
\noindent
{\sf Case f3:} ($i\not\in R, i+1\in R$) Same argument as {\sf Case e2}, take $g'' = g$. 

\medskip
\noindent
{\sf Case f4:} ($i,i+1\in R$) We use left-right ``flipped'' version of {\sf Case e4}. Let $a$ be as in Definition~\ref{def:feb11bicrystal}. If $m_{i, a}$ does not lie on  $A_g$ or $m_{i, a}\geq 2$, then we may take $g'' = g$. Otherwise, since we assume $f_i^{\sf row}({\sf m})\neq \varnothing$, the entry ${\sf m}_{i,a}=1$ corresponds to an unmatched
``$)$'' in ${\sf bracket}_i({\sf row}({\sf m}))$. Let $a'$ be the first column to the left of $a$ in $C$, so $(i+1, a')\in A_g$. None of the ``$($'' in 
${\sf bracket}_i({\sf row}({\sf m}))$ associated to ${\sf m}_{i+1,a'}>0$ match with this aforementioned ``$)$''. In particular, the rightmost ``$($'' associated to ${\sf m}_{i+1,a'}>0$ matches with a ``$)$'' associated
to ${\sf m}_{i,c''}>0$ for some $a'<c''<a$. Take $g''$ to be the minor defined by row indices $R$ and column indices $C'' = (C\setminus \{a\})\cup\{c''\}$ (which is a Fulton generator of $I_w$ by (\ref{eqn:thegenerators})).

    The statements for $f_j^{\sf col}$ and $e_j^{\sf col}$ hold by taking transposes of all matrices involved.
\end{proof}

\begin{proof}[Proof of Theorem~\ref{thm:crystalclosure}]
    Let 
    $\mathfrak{X}\subseteq {\sf Mat}_{m,n}$ be $\mathbf{B}$-stable. By Corollary~\ref{thm:borelunion},  $\mathfrak{X}=\bigcup_{i=1}^k \mathfrak{X}_{w^{(i)}}$ is a union of matrix Schubert varieties. By Proposition~\ref{prop:leviacts}, any $\mathbf{L}_{{\mathbf I}|{\mathbf J}}$ satisfying the conditions in the theorem statement acts on $\mathfrak{X}_{w^{(i)}}$ and hence on $\mathfrak{X}$. 
    Furthermore, $\mathfrak{X}$ is $\mathbf{L}_{{\mathbf I}|{\mathbf J}}$-bicrystal closed by Proposition~\ref{thm:stdmonounion} combined with Theorem~\ref{thm:MSVLtype}.
\end{proof}

\section{Final remarks}\label{sec:polytopal}

Abusing notation, write an $({\bf I}|{\bf J})$-filtered partition-tuple $(\underline\lambda|\underline\mu)$ as a sequence of non-negative integers $(\lambda_1,\dots, \lambda_m\vert\mu_1,\dots,\mu_n)$ such that $\lambda_i\geq\lambda_{i+1}$ if $i\notin \mathbf{I}$ and $\mu_j\geq\mu_{j+1}$ if $j\notin \mathbf{J}$.

Let us return to the basic case $\mathfrak{X} = {\sf Mat}_{m, n}$.
Here $c_{\underline\lambda|\underline\mu}=c^{{\sf Mat}_{m, n}}_{\underline\lambda|\underline\mu}$, as in Example~\ref{exa:isgeneralizedLR}.
We set the convention that if $M = [m_{i, j}]\in{\sf Mat}_{m, n}(\Z_{\geq 0})$, then $m_{i,j} = 0$ if $i > m$ or $j > n$.
We can characterize highest-weight matrices of shape $(\underline\lambda|\underline\mu)$ as follows:

\begin{proposition}\label{thm:filterRSKpolyrule}
    The coefficients $c_{\underline\lambda|\underline\mu}$ count the number of $M = [m_{i, j}]\in {\sf Mat}_{m,n}(\Z_{\geq 0})$ that are lattice points of the polytope $\mathcal{P}_{\underline\lambda|\underline\mu}$ defined by the \emph{highest-weight conditions}
    \begin{align}\label{eqn:feb9abc}
        \sum_{k=j}^{n}m_{i+1, k}\leq\sum_{k=j}^{n}m_{i, k+1}\ \mathrm{for}\ \mathrm{all}\ i\notin \mathbf{I}\ \mathrm{and}\ j\in[n],\\ 
        \sum_{k=i}^{m}m_{k, j+1}\leq\sum_{k=i}^{m}m_{k+1, j}\ \mathrm{for}\ \mathrm{all}\ j\notin \mathbf{J}\ \mathrm{and}\ i\in[m], \label{eqn:feb9cde}
    \end{align}
    the \emph{shape conditions}
    \begin{align}\label{eqn:feb9xx1}
        \sum_{j=1}^n m_{i,j} &= \lambda_i\ \mathrm{for}\ \mathrm{all}\  i\in[m],\\
        \sum_{i=1}^m m_{i,j} &= \mu_j\ \mathrm{for}\  \mathrm{all}\  j\in[n], \label{eqn:feb9xx2}
    \end{align}
    and the \emph{nonnegativity constraints} $m_{ij}\geq 0$ for all $i\in [m],j\in [n].$
\end{proposition}
\begin{proof}
    $M = [m_{i,j}]\in {\sf Mat}_{m,n}({\mathbb Z}_{\geq 0})$ has \emph{highest weight} if $e_i^{\sf row}(M) = 0 = e_j^{\sf col}(M)$ for all $i\notin {\bf I}$ and $j\notin {\bf J}$. Writing out the parenthesis-matching rules for these matrix crystal operators numerically expresses the highest-weight conditions as  (\ref{eqn:feb9abc}) and (\ref{eqn:feb9cde}).
        
    The highest weight matrices $M$ are exactly those sent to some highest-weight tableau-tuple by ${\sf filterRSK}_{{\mathbf I}|{\mathbf J}}$. The shape of a highest-weight matrix $M$ is therefore uniquely determined by its weight, i.e., the collection of its row and column sums. Thus a highest-weight matrix $M$ has shape $(\underline\lambda|\underline\mu)$ if and only if it satisfies 
    (\ref{eqn:feb9xx1}) and (\ref{eqn:feb9xx2}).
\end{proof}

\begin{example}[Contingency tables]
    Let ${\bf I} = \{0\}\cup [m]$ and ${\bf J} = \{0\}\cup [n]$, so $\mathbf{L}_{{\mathbf I}|{\mathbf J}} = \mathbf{T}$ is the $(m+n)$-torus acting on ${\sf Mat}_{m, n}$.
    Then ${\sf filterRSK}_{{\mathbf I}|{\mathbf J}}$ computes the $\mathbf{T}$-character of $\C[{\sf Mat}_{m, n}]$ by mapping the exponent matrix $M\in{\sf Mat}_{m, n}(\Z_{\geq 0})$ of each monomial to the exponent vector ${\sf cwt}_{\mathcal{M}_{m, n}}(M)$ of its torus weight.
    Here, there are no $({\bf I}|{\bf J})$-filtered crystal moves, so 
    (\ref{eqn:feb9abc}) and (\ref{eqn:feb9cde})
    are trivial.
    Thus, (\ref{eqn:feb9xx1}) and (\ref{eqn:feb9xx2}) say
        $c_{\underline\lambda\vert\underline\mu}$ counts \emph{contingency tables}, i.e., $M\in {\sf Mat}_{m,n}({\mathbb Z}_{\geq 0})$ whose $i$th row sums to $\lambda_i$ and $j$th column sums to $\mu_j$ for all $i\in[m]$ and $j\in[n]$. 
\end{example}

\begin{example}[The classical Cauchy identity]\label{ex:RSK}
    \!Let ${\bf I} = \{0, m\}$ and ${\bf J} = \{0, n\}$. Then ${\sf filterRSK}_{{\mathbf I}|{\mathbf J}}$ is the usual ${\sf RSK}$ map sending each $M = [m_{i, j}]\in {\sf Mat}_{m,n}({\mathbb Z}_{\geq 0})$ to a tableau-pair $(P|Q)$.
    Applying (\ref{eqn:feb9abc}) with $j = n$, we see that $m_{i+1, n} = m_{i, n+1}$ for all $i\in[m-1]$. But $m_{i, n+1} = 0$ for all such $i$ since these entries lie entirely outside of $M$, so in fact $m_{i+1, n} = 0$ for $i\in[m-1]$.
    Similarly, applying (\ref{eqn:feb9cde}) with $i = m$ shows that $m_{m, j+1} = 0$ for all $j\in[n-1]$. Iterating this argument shows that $m_{i+k, n+1-k} = m_{n+1-k, j+k} = 0$ for all $i\in[m-1]$, $j\in[n-1]$, and $k\in\Z_{\geq 0}$, so any matrix $M$ satisfying (\ref{eqn:feb9abc}) and (\ref{eqn:feb9cde}) must be northwest-triangular.

    Similar iterative arguments on the inequalities (or equivalently, in terms of brackets) on the non-zero entries of $M$ show that all entries on each antidiagonal must be equal. In other words, the matrices satisfying the highest-weight conditions 
    (\ref{eqn:feb9abc}) and (\ref{eqn:feb9cde})
    are of the form 
    $\left[\begin{smallmatrix}
        a & b & c & d & e\\
        b & c & d & e & 0\\
        c & d & e & 0 & 0\\
        d & e & 0 & 0 & 0\\
        e & 0 & 0 & 0 & 0
    \end{smallmatrix}\right]$.
    Now, it is easy to see that further imposing (\ref{eqn:feb9xx1}) and (\ref{eqn:feb9xx2})
    shows that $c_{\lambda\vert\mu} = 1$ if $\lambda = \mu$ and $0$ otherwise, by Proposition~\ref{thm:filterRSKpolyrule}. This recovers the classical Cauchy identity (\ref{eqn:Cauchy}) and proves Theorem~\ref{thm:RSK}.
\end{example}

\begin{example}[A polytopal Littlewood--Richardson rule]\label{LR}
    Let $\mathbf{I} = \{0, t, m\}$ for some $0<t<m$ and let $\mathbf{J} = \{0, n\}$. Write $(\underline\lambda|\underline\mu) = (\lambda^{(1)}, \lambda^{(2)}|\mu)$ for an $({\bf I}|{\bf J})$-filtered tableau-tuple. Then $c_{\underline\lambda|\underline\mu}$ is the \emph{Littlewood-Richardson coefficient}  
    $c_{\lambda^{(1)}, \lambda^{(2)}}^\mu$ (for partitions of the right size). Thus, Proposition~\ref{thm:filterRSKpolyrule} specializes to a polytopal rule for Littlewood--Richardson coefficients.  
\end{example}


Proposition~\ref{thm:filterRSKpolyrule} extends to  $c^{\mathfrak{X}}_{\underline{\lambda}|\underline{\mu}}$ where $\mathfrak{X}\subseteq{\sf Mat}_{m, n}$ is any $\mathbf{L}_{\mathbf{I}|\mathbf{J}}$-bicrystalline variety.

\begin{proposition}\label{thm:feb10poly}
    Let $\mathfrak{X}\subseteq{\sf Mat}_{m, n}$ be ${\bf L}_{{\mathbf I}|{\mathbf J}}$-bicrystalline. Let $\mathcal{P}_{\underline\lambda|\underline\mu}$ be the polytope in Proposition~\ref{thm:filterRSKpolyrule}, and let $\mathcal{P}^{\mathfrak{X}}$ be the set of lattice points corresponding to elements of ${\sf Std}_<(S/I(\mathfrak{X}))$. Then $c^\mathfrak{X}_{\underline\lambda|\underline\mu}$ counts lattice points in 
    $\mathcal{P}^{\mathfrak{X}}_{\underline\lambda|\underline\mu} = \mathcal{P}^\mathfrak{X}\cap\mathcal{P}_{\underline\lambda|\underline\mu}$.
\end{proposition}
\begin{proof} By Theorem~\ref{thm:main},
${\sf filterRSK}_{{\mathbf I}|{\mathbf J}}$ computes the $\mathbf{L}_{{\mathbf I}|{\mathbf J}}$-character of $\mathfrak{X}$. By Proposition~\ref{thm:filterRSKpolyrule}, the coefficient $c_{\underline\lambda|\underline\mu} = c^{{\sf Mat}_{m, n}}_{\underline\lambda|\underline\mu}$ counts the integer points in $\mathcal{P}_{\underline\lambda|\underline\mu}$. The desired coefficient $c^{\mathfrak{X}}_{\underline\lambda|\underline\mu}$ counts only the points in this polytope corresponding elements of ${\sf Std}_<(S/I(\mathfrak{X}))$.
\end{proof}

\begin{remark}\label{remark:unionofpoly}
    Proposition~\ref{thm:feb10poly} is not a \emph{polytopal} rule for $c^{\mathfrak{X}}_{\underline\lambda|\underline\mu}$, since the set $\mathcal{P}^\mathfrak{X}$ is not generally a polytope. However, if $\init_<(I({\mathfrak X}))$ is squarefree, one can give a serviceable description of $\mathcal{P}^\mathfrak{X}$ as a union of polytopes using a prime decomposition of $\init_{<}(I(\mathfrak{X}))$. This case includes all $\mathbf{B}$-stable varieties by the description of their initial ideals given previously.
\end{remark}

\begin{example}[Classical determinantal varieties]\label{exa:classical}
    Let ${\bf I} = \{0, m\}$ and ${\bf J} = \{0, n\}$. Let $\mathfrak{X}_k\subseteq{\sf Mat}_{m, n}$ be the variety of rank $\leq k$ matrices. It is known that a Gr\"obner basis for $I({\mathfrak X}_k)$ with respect to antidiagonal term order consists of the $k+1$-order minors of the generic $m\times n$ matrix (this is an instance of Theorem~\ref{thm:KMgrob}). The same highest-weight and contingency table conditions hold as in Example~\ref{ex:RSK}, but if $k\neq m$ or $k\neq n$ then there are also nontrivial standard monomial conditions. The standard monomials of $\mathfrak{X}_k$ correspond to matrices with at least one $0$ on each antidiagonal after the $k$th (where the first antidiagonal is the northwest corner). Combining  with the highest-weight and contingency table restrictions from before yields this known Cauchy-type identity:  \begin{equation}
\label{eqn:anotherCauchy}
\chi_{{\mathbb C}[{\mathfrak X}_k]}
 =\sum_{\lambda,\ell(\lambda)\leq k} s_{\lambda}(\mathbf{x}) s_{\lambda}(\mathbf{y}).
\end{equation}
\end{example}

\appendix
\section{Towards equivariant minimal free resolutions}\label{sec:appendix}
In this supplement to the main body of the paper, we discuss the relationship between our results, prior work on the Hilbert series of matrix Schubert varieties, and open problems on their minimal free resolutions. The perspective for these problems comes from combinatorial commutative algebra rather than representation theory; we will assume some familiarity and refer to the book of Miller--Sturmfels \cite{Miller.Sturmfels} as a standard reference. Fix a split reductive group $(G, T)$, and let $\mathfrak{X}$ be a $G$-variety with a $T$-compatible embedding such that $\complexes[\mathfrak{X}]\cong S/I(\mathfrak{X})$ (where $S = \complexes[z_1,\dots, z_N]$ for some $N$). As mentioned previously in Section~\ref{subsec:stdmono}, the $T$-character of $\complexes[\mathfrak{X}]$ is known in commutative algebra and algebraic geometry as the \emph{$T$-multigraded Hilbert series} of (the chosen embedding of) $\mathfrak{X}$. A classical result shows that the characters of coordinate rings, viewed as power series, can be written compactly as rational functions.

\begin{theorem}[Hilbert--Serre,\ {\cite[Theorem 8.20]{Miller.Sturmfels}}]\label{thm:hilbertserre}
    Let $S = \complexes[z_1,\dots, z_N]$, let $T\cong(\complexes^\star)^k$ be a compatible torus acting on $S$, and let $\mathfrak{X}$ be a variety with a $T$-compatible embedding such that $\complexes[\mathfrak{X}] = S/I(\mathfrak{X})$. Then there exists a polynomial $K_{\complexes[\mathfrak{X}]}(t_1,\dots, t_k)\in\integers[t_1,\dots, t_k]$, the \emph{$K$-polynomial} of $\complexes[\mathfrak{X}]$, such that
    \[\chi_{\complexes[\mathfrak{X}]} = \sum_{{\sf m}\in \mathrm{Std}_<(S/I(\mathfrak{X}))}{\sf wt}_T({\sf m}) = \frac{K_{\complexes[\mathfrak{X}]}(t_1,\dots, t_k)}{\prod_{i=1}^N (1-{\sf wt}_T(z_i))}.\footnote{The reference \cite[Theorem 8.20]{Miller.Sturmfels} applies in a more general context where the $K$-polynomial may be a Laurent polynomial. The $T$-compatible hypothesis in our statement guarantees that $\complexes[\mathfrak{X}]$ is a \emph{polynomial} rather than rational $T$-representation, ensuring that $K_{\complexes[\mathfrak{X}]}(t_1,\dots, t_k)$ is a bonafide polynomial.}\]
\end{theorem}

The $K$-polynomial of $\complexes[\mathfrak{X}]$ carries significance beyond determining the Hilbert series. Indeed, it is the Euler characteristic of the \emph{minimal free resolution} of $\complexes[\mathfrak{X}]$ as an $S$-module \cite[Proposition 8.23]{Miller.Sturmfels}. 
Determining the minimal free resolution of $\complexes[\mathfrak{X}]$ is usually very difficult, as such a resolution encodes many important algebraic properties of $\complexes[\mathfrak{X}]$ (e.g. the projective dimension, Castelnuovo--Mumford regularity, and $T$-multidegrees of the generators in a minimal generating set for $I(\mathfrak{X})$). One can also compute the \emph{$T$-multidegree} of $\complexes[\mathfrak{X}]$ as the lowest-total-degree homogeneous component of the \emph{twisted $K$-polynomial} $K_{\complexes[\mathfrak{X}]}(1-t_1,\dots, 1-t_k)$. In summary, we have three polynomials/power series of interest, each of which could be said to describe the Hilbert series of $\complexes[\mathfrak{X}]$:
\begin{enumerate}
    \item The $T$-character $\chi_{\complexes[\mathfrak{X}]}$ (the Hilbert series itself),
    \item The $K$-polynomial $K_{\complexes[\mathfrak{X}]}(t_1,\dots, t_k)$ (the Hilbert numerator),
    \item The twisted $K$-polynomial $K_{\complexes[\mathfrak{X}]}(1-t_1,\dots, 1-t_k)$ (a generalization of multidegree).
\end{enumerate}
Given any one of these three objects, the others can be determined algebraically via variable substitutions, cross-multiplication, and polynomial long division. However, given a \emph{combinatorial} formula for one object, these algebraic manipulations introduce complicated alternating sums and do \emph{not} yield any obvious combinatorial formula for the other two.

Now specialize to the setting of our paper, where $S = \complexes[{\sf Mat}_{m, n}]$, $G = \mathbf{L}_{\mathbf{I}|\mathbf{J}}$ for some indexing sets $\mathbf{I}$ and $\mathbf{J}$, and $T = \mathbf{T}$ is the $(m+n)$-torus acting by scaling rows and columns of our matrix (so ${\sf wt}_\mathbf{T}(z_{ij}) = x_iy_j$). Applying Proposition~\ref{thm:bigsummary} as in Section~\ref{subsec:gln} along with Theorem~\ref{thm:hilbertserre}, we know that if $\mathfrak{X}$ is a $\mathbf{L}_{\mathbf{I}|\mathbf{J}}$-stable variety with a $\mathbf{T}$-compatible embedding, then for some constants $c_{\underline\lambda|\underline\mu}^\mathfrak{X}$,
\begin{equation}\label{eqn:equivhilbertserre}
    \chi_{\complexes[\mathfrak{X}]} = \sum_{\underline\lambda|\underline\mu} c^{\mathfrak{X}}_{\underline\lambda|\underline\mu} s_{\underline\lambda}(\mathbf{x})s_{\underline\mu}(\mathbf{y}) = \frac{K_{\complexes[\mathfrak{X}]}(\mathbf{x},\mathbf{y})}{\prod_{i, j} (1-x_iy_j)}.
\end{equation}
The denominator of the rational function in (\ref{eqn:equivhilbertserre}) is stable under any permutation of the $x$'s or of the $y$'s. In particular, it is split-symmetric for the splitting determined by $\mathbf{I}$ and~$\mathbf{J}$. 
Since the other expression for the $\mathbf{T}$-character of $\complexes[\mathfrak{X}]$ in (\ref{eqn:equivhilbertserre}) is also split-symmetric, it follows that $K_{\complexes[\mathfrak{X}]}(t_1,\dots, t_k)$ has the same split-symmetry. In other words, there exist (possibly negative) integers $b_{\underline\lambda|\underline\mu}^\mathfrak{X}$ such that
\begin{equation}\label{eqn:equivK}
    K_{\complexes[\mathfrak{X}]}(\mathbf{x}, \mathbf{y}) = \sum_{\underline\lambda|\underline\mu} b^\mathfrak{X}_{\underline\lambda|\underline\mu} s_{\underline\lambda}(\mathbf{x})s_{\underline\mu}(\mathbf{y}).
\end{equation}
The variable substitutions $x_i\mapsto 1-x_i$ and $y_j\mapsto 1-y_j$ do not alter the split-symmetry of $K_{\complexes[\mathfrak{X}]}(\mathbf{x},\mathbf{y})$, so there also exist integers $d_{\underline\lambda|\underline\mu}^\mathfrak{X}$ such that
\begin{equation}\label{eqn:equivK2}
    K_{\complexes[\mathfrak{X}]}(\mathbf{1}-\mathbf{x}, \mathbf{1}-\mathbf{y}) = \sum_{\underline\lambda|\underline\mu} d^\mathfrak{X}_{\underline\lambda|\underline\mu} s_{\underline\lambda}(\mathbf{x})s_{\underline\mu}(\mathbf{y}).
\end{equation}
\begin{problem}\label{prob:3Qs}
    Given a $\mathbf{L}_{\mathbf{I}|\mathbf{J}}$-stable variety $\mathfrak{X}\subseteq{\sf Mat}_{m, n}$, give combinatorial rules for the following coefficients:
    \begin{enumerate}
        \item the $c_{\underline\lambda|\underline\mu}^\mathfrak{X}$ in the $\mathbf{L}_{\mathbf{I}|\mathbf{J}}$-equivariant expression (\ref{eqn:equivhilbertserre}) for $\chi_{\complexes[\mathfrak{X}]}$,
        \item the $b_{\underline\lambda|\underline\mu}^\mathfrak{X}$ in the $\mathbf{L}_{\mathbf{I}|\mathbf{J}}$-equivariant expression (\ref{eqn:equivK}) for $K_{\complexes[\mathfrak{X}]}(\mathbf{x},\mathbf{y})$,
        \item the $d_{\underline\lambda|\underline\mu}^\mathfrak{X}$ in the $\mathbf{L}_{\mathbf{I}|\mathbf{J}}$-equivariant expression (\ref{eqn:equivK2}) for $K_{\complexes[\mathfrak{X}]}(\mathbf{1}-\mathbf{x},\mathbf{1}-\mathbf{y})$.
    \end{enumerate}
\end{problem}
Our primary purpose in this paper was to solve Problem~\ref{prob:3Qs}(1) for matrix Schubert varieties $\mathfrak{X}_w$ and their unions. Indeed, Main Theorem~\ref{thm:main} solves Problem~\ref{prob:3Qs}(1) for all $\mathbf{L}_{\mathbf{I}|\mathbf{J}}$-bicrystalline varieties, which includes all $\mathbf{B}$-stable varieties (i.e., all unions of matrix Schubert varieties) by Main Theorem~\ref{thm:crystalclosure} and Corollary~\ref{thm:borelunion}.

Prior work on the Hilbert series of matrix Schubert varieties has focused on describing their $K$-polynomials. In the combinatorial commutative algebra literature, a fundamental tool for computing Hilbert series of coordinate rings $\complexes[\mathfrak{X}]$ is a simple formula for the $K$-polynomial using the combinatorics of the \emph{Stanley--Reisner complex} associated to some Gr\"obner degeneration of $\mathfrak{X}$ \cite[Theorem 1.13]{Miller.Sturmfels}. The Hilbert series of matrix Schubert varieties $\mathfrak{X}_w$ were originally computed using techniques along these lines; in \cite[Theorem A]{Knutson.Miller}, Knutson--Miller showed that the $K$-polynomial of $\mathfrak{X}_w$ is a \emph{double Grothendieck polynomial} using a Stanley--Reisner complex with facets indexed by the \emph{reduced pipe dreams}\footnote{Also known as \emph{rc-graphs}, see \cite{Bergeron.Billey, Fomin.Kirillov}.} for~$w$. Unfortunately, $K$-polynomial formulas from Stanley--Reisner complexes are usually highly cancellative and thus fail to solve Problem~\ref{prob:3Qs}(2)---even the monomial expansion of $K_{\complexes[\mathfrak{X}]}(\mathbf{x},\mathbf{y})$ is not combinatorially obvious from such a formula.
\begin{remark}
    There are several conventions for the definition of Grothendieck polynomials, which can lead to genuine confusion and differences. We pause to clarify our choices. We define
    \begin{align*}
        \mathfrak{G}_w(\mathbf{1}-\mathbf{xy}) &:= K_{\complexes[\mathfrak{X}_w]}(\mathbf{x}, \mathbf{y}),\\
        \mathfrak{G}_w(\mathbf{x}+\mathbf{y}-\mathbf{xy}) &:= K_{\complexes[\mathfrak{X}_w]}(\mathbf{1}-\mathbf{x},\mathbf{1}-\mathbf{y}).
    \end{align*}
    Our notation is justified by the pipe dream formulas often used to define Grothendieck polynomials combinatorially (see, e.g., \cite[Corollary 5.4]{KM:adv}).
    Summing the weights of all pipe dreams for $w$ yields $\mathfrak{G}_w(\mathbf{1}-\mathbf{xy})$ when the weight of a ``cross'' tile in position $(i, j)$ of a pipe dream is taken to be $(1-x_iy_j)$, and it yields $\mathfrak{G}_w(\mathbf{x}+\mathbf{y}-\mathbf{xy})$ when this weight is taken to be $(x_i+y_j-x_iy_j)$. 
    In their papers \cite{Knutson.Miller, KM:adv}, Knutson--Miller denote our $\mathfrak{G}_w(\mathbf{1}-\mathbf{xy})$ by $\mathfrak{G}_w(\mathbf{x},\mathbf{y})$ and our $\mathfrak{G}_w(\mathbf{x}+\mathbf{y}-\mathbf{xy})$ by $\mathfrak{G}_w(\mathbf{1}-\mathbf{x},\mathbf{1}-\mathbf{y})$, referring to the former as Grothendieck polynomials. However, in the algebraic combinatorics literature the term ``Grothendieck polynomial'' usually refers to $\mathfrak{G}_w(\mathbf{x}+\mathbf{y}-\mathbf{xy})$, the polynomial whose lowest-degree homogeneous component is the double Schubert polynomial $\mathfrak{S}(\mathbf{x},\mathbf{y})$. There are also \emph{single} Grothendieck polynomials
    \begin{align*}
        \mathfrak{G}_w(\mathbf{1}-\mathbf{x}) &:= K_{\complexes[\mathfrak{X}_w]}(\mathbf{x}, \mathbf{1}),\\
        \mathfrak{G}_w(\mathbf{x}) &:= K_{\complexes[\mathfrak{X}_w]}(\mathbf{1}-\mathbf{x},\mathbf{0}).
    \end{align*}
    Here also our notation is justified by the weights used in the pipe dream formulas for these polynomials. We write $\mathfrak{G}_w(\mathbf{y})$ to denote $\mathfrak{G}_w(\mathbf{x})$ after the substitution $x_i\mapsto y_i$. This is not the same as $K_{\complexes[\mathfrak{X}_w]}(\mathbf{0},\mathbf{1}-\mathbf{y})$; in fact, $\mathfrak{G}_w(\mathbf{y}) = K_{\complexes[\mathfrak{X}_{w\inv}]}(\mathbf{0}, \mathbf{1}-\mathbf{y})$.
    To lessen confusion, we will use our precise notation and henceforth avoid the term ``Grothendieck polynomial''.
\end{remark}

The combinatorics of $\mathfrak{G}_w(\mathbf{x}+\mathbf{y}-\mathbf{xy})$ is generally better understood than the combinatorics of $\mathfrak{G}_w(\mathbf{1}-\mathbf{xy})$. Indeed, chaining together several known formulas yields a complete solution to Problem~\ref{prob:3Qs}(3) for matrix Schubert varieties $\mathfrak{X}_w$. We sketch the derivation, referring to sources for details. First, we have a Cauchy identity due to Fomin--Kirillov \cite{FKGroth} for $\mathfrak{G}_w(\mathbf{x}+\mathbf{y}-\mathbf{xy})$ (generalizing their \cite[Theorem 8.1]{Fomin.Kirillov}), where $u\cdot v$ denotes the \emph{Demazure product} of permutations $u$ and $v$:
\begin{equation}\label{eqn:GrothCauchy}
    \mathfrak{G}_w(\mathbf{x}+\mathbf{y}-\mathbf{xy}) = \sum_{\substack{u, v\\ u\cdot v = w}}(-1)^{\ell(uvw)}\mathfrak{G}_{v}(\mathbf{x})\mathfrak{G}_{u\inv}(\mathbf{y}).
\end{equation}
In \cite{LenartGroth}, Lenart gave a multiplicity-free rule expanding $\mathfrak{G}_w(\mathbf{x})$ into the basis of single Schubert polynomials $\mathfrak{S}_v(\mathbf{x})$. See \cite[Theorem 1.5]{Weigandt} for a statement in terms of pipe dreams. Thus there is a combinatorial rule for the coefficients $a_{v, w}\in\{0, 1\}$ in the expansion
\begin{equation}\label{eqn:GrothSchub}
    \mathfrak{G}_w(\mathbf{x}) = \sum_{v}(-1)^{\ell(wv)}a_{v, w}\mathfrak{S}_v(\mathbf{x}).
\end{equation}
Finally, for any $\mathbf{I}\supset\desc_{\sf row}(w)$, Buch--Kresch--Tamvakis--Yong have given a combinatorial rule \cite[Corollary 3]{BKTY} for the nonnegative coefficients $c^w_{\underline\lambda}$ in the $\mathbf{L}_\mathbf{I}$-equivariant expansion of the Schubert polynomial $\mathfrak{S}_w(\mathbf{x})$ as a sum of products $s_{\underline\lambda}(\mathbf{x})$ of Schur polynomials:
\begin{equation}\label{eqn:splitSchub}
    \mathfrak{S}_w(\mathbf{x}) = \sum_{\underline\lambda}c_{\underline\lambda}^ws_{\underline\lambda}(\mathbf{x}).
\end{equation}
Note that since $\mathfrak{S}_w(\mathbf{x})$ is a homogeneous polynomial of degree $\ell(w)$, $c_{\underline\lambda}^w$ is nonzero only when $|\underline\lambda| = \ell(w)$. 
Combining these three combinatorial formulas (along with the facts that $(-1)^{\ell(u)+\ell(v)} = (-1)^{\ell(uv)}$ and $\ell(w) = \ell(w\inv)$ for all permutations $u$, $v$, $w$) gives the following non-cancellative solution to Problem~\ref{prob:3Qs}(3) for $\mathfrak{X}_w$.
\begin{theorem}\label{thm:equivtwistedK}
    In the case where $\mathfrak{X} = \mathfrak{X}_w$ is a matrix Schubert variety, $\mathbf{I}\supset\desc_{\sf row}(w)$, and $\mathbf{J}\supset\desc_{\sf col}(w)$, the $\mathbf{L}_{\mathbf{I}|\mathbf{J}}$-equivariant coefficients $d^\mathfrak{X}_{\underline\lambda|\underline\mu}$ in Problem~\ref{prob:3Qs}(3) are given by the formula
    \[(-1)^{\ell(w)+|\underline\lambda|+|\underline\mu|}\cdot d_{\underline\lambda|\underline\mu}^{\mathfrak{X}_w} = \sum_{\substack{u, v, A, B\\ u\cdot v = w\\ a_{A,v}=a_{B, u\inv}=1}}c_{\underline\lambda}^Ac_{\underline\mu}^B.\]
\end{theorem}
\begin{proof}
\begin{align*}
    K_{\complexes[\mathfrak{X}_w]}(\mathbf{1}-\mathbf{x},\mathbf{1}-\mathbf{y}) &=: \mathfrak{G}_w(\mathbf{x}+\mathbf{y}-\mathbf{xy})\\
    &\stackrel{\eqref{eqn:GrothCauchy}}{=} \sum_{\substack{u, v\\ u\cdot v = w}}(-1)^{\ell(uvw)}\mathfrak{G}_{v}(\mathbf{x})\mathfrak{G}_{u\inv}(\mathbf{y})\\
    &\stackrel{\eqref{eqn:GrothSchub}}{=} \sum_{\substack{A, B, u, v\\ u\cdot v = w}} (-1)^{\ell(uvw)+\ell(vA)+\ell(u\inv B)}a_{A,v}a_{B, u\inv}\mathfrak{S}_A(\mathbf{x})\mathfrak{S}_B(\mathbf{y})\\
    &\stackrel{\eqref{eqn:splitSchub}}{=} \sum_{\substack{A, B, u, v,\underline\lambda, \underline\mu\\ u\cdot v = w}}(-1)^{\ell(ABw)}a_{A,v}a_{B,u\inv}c_{\underline\lambda}^Ac_{\underline\mu}^Bs_{\underline\lambda}(\mathbf{x})s_{\underline\mu}(\mathbf{y})\\
    &= \sum_{\underline\lambda,\underline\mu}(-1)^{\ell(w)+|\underline\lambda|+|\underline\mu|}\Biggl(\sum_{u, v, A, B}c_{\underline\lambda}^Ac_{\underline\mu}^B\Biggr) s_{\underline\lambda}(\mathbf{x})s_{\underline\mu}(\mathbf{y}),
\end{align*}
where in the last line the sum is over permutations $u$, $v$, $A$, $B$ such that $u\cdot v = w$ and $a_{A,v}=a_{B, u\inv}=1$.
\end{proof}

It would be interesting to give a crystal-theoretic interpretation for the formula in Theorem~\ref{thm:equivtwistedK}. Some work in this direction has been carried out by Lenart, who gave an alternate proof of the combinatorial rule (\ref{eqn:splitSchub}) for $c_{\underline\lambda}^w$ using operators closely related (via \emph{plactification}) to the crystal operators we consider \cite[Theorem 5.1]{LenartSchub}.

In contrast to Problems \ref{prob:3Qs}(1) and (3), and in spite of its position midway between the well-studied combinatorics of RSK and pipe dreams, Problem \ref{prob:3Qs}(2) appears difficult to attack combinatorially. 
This is unfortunate, as a combinatorial solution would hint towards the form of the $\mathbf{L}_{\mathbf{I}|\mathbf{J}}$-equivariant minimal free resolution of $\complexes[\mathfrak{X}_w]$.\footnote{Two pieces of these resolutions are more amenable to combinatorial study than the others. 
The first piece is the first nontrivial step of the resolution, where the ranks of the free modules that appear count elements of a minimal generating set for $I(\mathfrak{X}_w)$. 
These minimal generating sets are known by work of Gao and the third author \cite[Theorem 1.6]{Gao.Yong}, and their description can be made equivariant. 
The second piece comes from the tail of the resolution, where the top-degree homogeneous component of $K_{\complexes[\mathfrak{X}_w]}(\mathbf{x},\mathbf{y})$ determines the unique \emph{extremal Betti number} of $\complexes[\mathfrak{X}_w]$. 
Since the top-degree components of $K_{\complexes[\mathfrak{X}_w]}(\mathbf{x},\mathbf{y})$ and $K_{\complexes[\mathfrak{X}_w]}(\mathbf{1}-\mathbf{x},\mathbf{1}-\mathbf{y})$ agree up to sign, Theorem~\ref{thm:equivtwistedK} implicitly gives an equivariant description of the extremal Betti number. 
The top-degree component of $K_{\complexes[\mathfrak{X}_w]}(\mathbf{x},\mathbf{y})$ was used to determine the Castelnuovo--Mumford regularity of $\complexes[\mathfrak{X}_w]$ in a series of papers culminating in Pechenik--Speyer--Weigandt's work \cite{PSW}, with interest continuing since then (see, e.g., the pipe dream models recently introduced by Chou--Yu \cite{Chou.Yu}).} Minimal free resolutions of $\complexes[\mathfrak{X}_w]$ play an important role in determining the structure of singularities of Schubert varieties in the complete flag variety; these singularities have been long studied. See the recent survey \cite{WY.Survey} of Woo and the third author and references therein for details.

Large group actions help to organize minimal free resolution computations. A famous instance of this principle comes from Boij--S\"oderberg theory, where Eisenbud--Fl\o ystad--Weyman \cite{EFW} used equivariant methods to construct free resolutions with pure Betti diagrams. Another instance more relevant to our discussion is Lascoux's construction \cite{Lascoux} of $\mathbf{GL}$-equivariant minimal free resolutions for the classical determinantal varieties $\mathfrak{X}_k\subseteq{\sf Mat}_{m, n}$ (see also the exposition in \cite[Chapter 6]{Weyman}). This construction solves Problem~\ref{prob:3Qs}(2) for all $\mathbf{GL}$-stable varieties in ${\sf Mat}_{m, n}$; we record the solution in our notation.
\begin{definition}
    The \emph{rank} of a partition $\lambda$, denoted ${\sf rank}(\lambda)$, is the side length of the largest square that fits inside the Young diagram for $\lambda$ (called the \emph{Durfee square} of~$\lambda$). The \emph{$k$-stabilization} of $\lambda$ is the partition $\lambda_{(k)}$ obtained by adding $k$ rows to $\lambda$ of length ${\sf rank}(\lambda)$.
\end{definition}

\begin{theorem}[{\cite{Lascoux}, \cite[Theorem 6.1.4]{ Weyman}}]\label{thm:classicK}
    The $K$-polynomial of the classical determinantal variety $\mathfrak{X}_k\subseteq{\sf Mat}_{m, n}$ of matrices with rank $\leq k$ is
    \[K_{\complexes[\mathfrak{X}_k]}(\mathbf{x},\mathbf{y}) = \sum_{\lambda\subseteq[m-k]\times[n-k]}(-1)^{|\lambda|}s_{\lambda_{(k)}}(\mathbf{x})s_{(\lambda')_{(k)}}(\mathbf{y}).\]
\end{theorem}

Theorem~\ref{thm:classicK} elegantly solves Problem~\ref{prob:3Qs}(2) when $\mathfrak{X} = \mathfrak{X}_k$ and $\mathbf{L}_{\mathbf{I}|\mathbf{J}} = \mathbf{GL}$:
\[b^{\mathfrak{X}_k}_{\lambda, \mu} = \begin{cases}
    (-1)^{|\nu|} & \text{ if $\lambda = \nu_{(k)}$ and $\mu = (\nu')_{(k)}$ for some $\nu\subseteq[m-k]\times[n-k]$,}\\
    0 & \text{ otherwise.}
\end{cases}\]
Comparing the solutions to Problem~\ref{prob:3Qs}(1) and (2) for $\mathfrak{X}_k$ given in Example~\ref{exa:classical} and Theorem~\ref{thm:classicK} respectively, we obtain the formula
\begin{equation}\label{eqn:stabcauchy}
    \frac{\sum_{\nu\subseteq[m-k]\times[n-k]}s_{\nu_{(k)}}(\mathbf{x})s_{(\nu')_{(k)}}(\mathbf{y})}{\prod_{\substack{1\leq i\leq m\\ 1\leq j\leq n}}(1-x_iy_j)} = \sum_{\ell(\lambda)\leq k}s_\lambda(\mathbf{x})s_\lambda(\mathbf{y}).
\end{equation}
When $k = \min\{m, n\}$, the numerator on the left hand side of (\ref{eqn:stabcauchy}) is $1$ and we recover the classical Cauchy identity. When $k = 0$, the right hand side of (\ref{eqn:stabcauchy}) is $1$ and clearing denominators yields the \emph{dual Cauchy identity}
\begin{equation}\label{eqn:dualcauchy}
    \sum_{\nu\subseteq[m]\times[n]}s_\nu(\mathbf{x})s_{\nu'}(\mathbf{y}) = \prod_{\substack{1\leq i\leq m\\ 1\leq j\leq n}}(1-x_iy_j).
\end{equation}
There is a \emph{dual RSK} correspondence proving (\ref{eqn:dualcauchy}) combinatorially \cite[Theorem 7.14.2]{ECII}. This bijection is related to a bicrystal structure on binary matrices noticed by van Leeuwen \cite[Proposition 3.3.4]{bicrystal1} and Danilov--Koshevoi \cite[Section C8]{bicrystal2}. It arises from the representation theory of the exterior algebra $\bigwedge({\sf Mat}_{m, n})$, which appears here as a \emph{Koszul complex} resolving the coordinate ring $\complexes[\mathfrak{X}_0]\cong\complexes$ as a $\complexes[{\sf Mat}_{m, n}]$-module \cite[Proposition 1.28]{Miller.Sturmfels}. However, when $0 < k < \min\{m, n\}$ we know of no combinatorial or crystal-based argument for (\ref{eqn:stabcauchy}). It would suffice to give a combinatorial proof for the following corollary of Theorem~\ref{thm:classicK}:
\begin{corollary}\label{cor:classicstab}
    The stabilization map $(\lambda, \mu)\mapsto (\lambda_{(1)}, \mu_{(1)})$ gives a bijection between the terms appearing in the $K$-polynomials of $\mathfrak{X}_k\subseteq{\sf Mat}_{m, n}$ and $\mathfrak{X}_{k+1}\subseteq{\sf Mat}_{m+1, n+1}$.
\end{corollary}
As a first step towards attacking Problem~\ref{prob:3Qs}(2) for matrix Schubert varieties $\mathfrak{X}_w\subseteq{\sf Mat}_{m, n}$, we conjecture a generalization of Corollary~\ref{cor:classicstab}.
\begin{definition}
    The \emph{$k$th back-stabilization} of a partial permutation $w:[m]\to[n]\cup\{\infty\}$ is the partial permutation $1^k\times w:[m+k]\to[n+k]\cup\{\infty\}$ defined by the formula
    \[(1^k\times w)(i) = \begin{cases}
        i & i\leq k,\\
        w(i-k)+k & i > k,
    \end{cases}\]
    using the convention that $\infty + k = \infty$.
\end{definition}
\begin{definition}
    The \emph{$k$-stabilization} of a partition-tuple $\underline\lambda = \left(\lambda^{(1)},\lambda^{(2)},\dots,\lambda^{(r)}\right)$, denoted $\underline\lambda_{(k)}$, is the partition-tuple obtained by replacing $\lambda^{(1)}$ with its $k$-stabilization, i.e.,
    \[\underline{\lambda}_{(k)} = \left(\lambda_{(k)}^{(1)}, \lambda^{(2)},\dots, \lambda^{(r)}\right).\]
\end{definition}
\begin{conjecture}\label{conj:schubstab}
    For any partial permutation $w:[m]\to[n]\cup\{\infty\}$, there exists a \emph{stabilization number} $N_w$ such that the stabilization map $(\underline\lambda,\underline\mu)\mapsto(\underline{\lambda}_{(1)}, \underline{\mu}_{(1)})$ gives a bijection between the terms appearing in the $K$-polynomials of $\mathfrak{X}_{1^d\times w}\subseteq{\sf Mat}_{m+d, n+d}$ and $\mathfrak{X}_{1^{d+1}\times w}\subseteq{\sf Mat}_{m+d+1, n+d+1}$ whenever $d\geq N_w$.
\end{conjecture}
\begin{remark}
    If $w:[m]\to[n]\cup\{\infty\}$ has no descents, then the matrix Schubert variety $\mathfrak{X}_w$ is in fact a classical determinantal variety $\mathfrak{X}_k$ for some $k$. In this case, the $d$th back-stabilization $1^d\times w$ of $w$ also has no descents, and $\mathfrak{X}_{1^d\times w} = \mathfrak{X}_{k+d}\subseteq{\sf Mat}_{m+d, n+d}$. Thus Corollary~\ref{cor:classicstab} proves that Conjecture~\ref{conj:schubstab} holds for partial permutations $w$ with no descents, and moreover shows that one can take $N_w = 0$ in this case. It is not true that $N_w = 0$ in general. For example, the $K$-polynomial of $\mathfrak{X}_w$ for $w = 25134$ is:
    \ytableausetup{boxsize=0.3em}
    \begin{align*}
        &s_{(\emptyset)}({\bf x})s_{(\emptyset,\emptyset)}({\bf y}) - s_{(\ydiagram{1})}({\bf x})s_{(\ydiagram{1},\emptyset)}({\bf y}) - s_{(\ydiagram{1,1})}({\bf x})s_{(\emptyset,\ydiagram{1,1})}({\bf y}) + s_{(\ydiagram{1,1})}({\bf x})s_{(\ydiagram{2},\emptyset)}({\bf y}) + s_{(\ydiagram{2,1})}({\bf x})s_{(\emptyset,\ydiagram{1,1,1})}({\bf y})\\ 
        &+ s_{(\ydiagram{2,1})}({\bf x})s_{(\ydiagram{1},\ydiagram{1,1})}({\bf y}) - s_{(\ydiagram{3,1})}({\bf x})s_{(\ydiagram{1},\ydiagram{1,1,1})}({\bf y}) - s_{(\ydiagram{2,2})}({\bf x})s_{(\ydiagram{1},\ydiagram{1,1,1})}({\bf y}) - s_{(\ydiagram{2,2})}({\bf x})s_{(\ydiagram{2},\ydiagram{1,1})}({\bf y})\\ 
        &+ s_{(\ydiagram{3,2})}({\bf x})s_{(\ydiagram{2},\ydiagram{1,1,1})}({\bf y}).
    \end{align*} This polynomial has $10$ terms. However, the $K$-polynomial for $1\times w = 136245$ is:
    \begin{align*}
        &s_{(\emptyset)}({\bf x})s_{(\emptyset,\emptyset)}({\bf y}) - s_{(\ydiagram{1,1})}({\bf x})s_{(\ydiagram{1,1},\emptyset)}({\bf y}) - s_{(\ydiagram{1,1,1})}({\bf x})s_{(\emptyset,\ydiagram{1,1,1})}({\bf y}) \textcolor{red}{- s_{(\ydiagram{1,1,1})}({\bf x})s_{(\ydiagram{1},\ydiagram{1,1})}({\bf y})} + s_{(\ydiagram{1,1,1})}({\bf x})s_{(\ydiagram{2,1},\emptyset)}({\bf y})\\ 
        &+ s_{(\ydiagram{2,1,1})}({\bf x})s_{(\ydiagram{1},\ydiagram{1,1,1})}({\bf y}) + s_{(\ydiagram{2,1,1})}({\bf x})s_{(\ydiagram{1,1},\ydiagram{1,1})}({\bf y}) - s_{(\ydiagram{3,1,1})}({\bf x})s_{(\ydiagram{1,1},\ydiagram{1,1,1})}({\bf y}) - s_{(\ydiagram{2,2,2})}({\bf x})s_{(\ydiagram{2,1},\ydiagram{1,1,1})}({\bf y})\\ 
        &- s_{(\ydiagram{2,2,2})}({\bf x})s_{(\ydiagram{2,2},\ydiagram{1,1})}({\bf y}) + s_{(\ydiagram{3,2,2})}({\bf x})s_{(\ydiagram{2,2},\ydiagram{1,1,1})}({\bf y}).
    \end{align*} This polynomial has $11$ terms (with the ``new'' term highlighted in red), so $N_{25134} \geq 1$. The equivariant expansions of the $K$-polynomials for $1^2\times w$ and $1^3\times w$ also have $11$ terms. Taking $1$-stabilizations of all partition-tuples in the expansion for $1^2\times w$ gives the partition-tuples appearing in the expansion for $1^3\times w$, suggesting that $N_{25134} = 2$.
\end{remark}
This paper represents a first attempt at using $\mathbf{L}_{\mathbf{I}|\mathbf{J}}$-equivariant methods to analyze the Hilbert series of matrix Schubert varieties. In future work, we hope to further develop our perspective and make progress towards a full solution to Problem~\ref{prob:3Qs}(2). 

\section*{Acknowledgements}
We thank Andrew Hardt for a helpful conversation about \cite{KM:adv} and Shiliang Gao for pointing us to the motivating reference \cite{Sturmfels} on standard monomials and Hilbert series. We are grateful to the anonymous referee whose comments improved the exposition of the paper and motivated the addition of the Appendix. We thank Liam Buttitta, Claudiu Raicu, Steven Sam, Linus Setiabrata, and Keller VandeBogert for helpful conversations about the matterial in the Appendix. AS was supported by a Susan C.~Morisato IGL graduate student scholarship and an NSF graduate fellowship.
AY was supported by a Simons Collaboration grant. The authors were partially supported by an NSF RTG in Combinatorics (DMS 1937241).

\end{document}